\documentclass[11pt]{amsart}

\usepackage{a4wide}
\usepackage[utf8]{inputenc}
\usepackage[T1]{fontenc} 
\usepackage[english]{babel}
\usepackage{amsmath}
\usepackage{amsfonts}
\usepackage{amssymb} 
\usepackage{amsthm}
\usepackage{graphicx}
\usepackage{subcaption}
\usepackage{enumerate}
\usepackage{color}
\usepackage{bbm}
\usepackage{verbatim}

\usepackage{mathtools}
\usepackage{xstring}
\usepackage{bm}

\usepackage[ruled,vlined]{algorithm2e}
\usepackage{nomencl}
\usepackage{pdfpages}

\usepackage{tikz}
\usepackage{psfrag}
\usepackage{subcaption}

\usepackage{ifthen}
\usepackage{xstring}

\usepackage{float}
\usepackage{setspace}

\newcommand{\todo}[1]{}    
\newcommand{\mycomment}[1]{}
\newcommand{\DD}{\displaystyle}
\newcommand{\uu}{{u}}

\newcommand{\ui}{{\underline{i}}}
\newcommand{\uj}{{\underline{j}}}
\newcommand{\uk}{{\underline{k}}}

\newcommand{\tauh}{\tau,h}

\usepackage{comment}

\newcommand{\blue}[1]{{\color{blue} #1}}
\newcommand{\forben}[1]{}

\newcommand{\red}[1]{{\color{red} #1}}

\newcommand{\utau}{\overline{u}_{\tau}}

\newcommand{\WW}{W}

\newcommand*{\bigtimes}{\mathop{\raisebox{-.5ex}{\hbox{\huge{$\times$}}}}}
\newcommand{\R}{\mathbb{R}}
\newcommand{\E}[1]{\mathbb{E}\left[#1\right]} 
\renewcommand{\d}{\,\mathrm{d}}
\newcommand{\D}{\mathcal{D}}
\newcommand{\HS}[2]{L_2(#1,#2)} 

\newcommand{\bfx}{\bf x}
\renewcommand{\bfx}{\mathbf{x}}
\newcommand{\bfphi}{\bm{\phi}}
\newcommand{\obfphi}{\overline{\bm{\phi}}}
\newcommand{\bfpsi}{\bm{\psi}}
\newcommand{\bfchi}{\bm{\chi}}

\newcommand{\Lp}[1][p]{{\mathbb{L}^{#1}}} 
\newcommand{\Hk}[1][k]{{\mathbb{H}^{#1}}} 
\newcommand{\Honez}{{\mathbb{H}^1_0}}

\newcommand{\Wkp}[2]{{\mathbb{W}^{#1,#2}}} 
\newcommand{\Wkpz}[2]{{\mathbb{W}_0^{#1,#2}}} 

\newcommand{\Ltwo}{\Lp[2]}

\newcommand{\Hmone}{\Hk[-1]}

\newcommand{\Vh}{\mathbb{V}_h}
\newcommand{\vh}{v_h}
\newcommand{\mDeltaD}{(-\Delta)}
\newcommand{\Vv}{\mathbb{V}}
\newcommand{\Hh}{\mathbb{H}}
\newcommand{\Kk}{\mathbb{K}}

\newcommand{\dual}[1]{
\noexpandarg
\StrLen{#1}[\mystringlen]
\ifthenelse{\mystringlen=1}{{#1}'}{(#1)'}
}
\newcommand{\dpairing}[3]{\left\langle #1, #2\right \rangle_{\dual{#3}\times #3}}

\newtheorem{theorem}{Theorem}[section]
\newtheorem{example}{Example}[section]
\newtheorem{definition}[theorem]{Definition}
\newtheorem{rem}[theorem]{Remark}
\newtheorem{assumption}{Assumption}
\newtheorem{lemma}[theorem]{Lemma}
\newtheorem{corollar}[theorem]{Corollary}
\newtheorem{cor}[theorem]{Corollary}

\begin{document}

\title[Numerical approximation of singular-degenerate SPDEs]{Numerical approximation of singular-degenerate parabolic stochastic PDEs}

\author{\v{L}ubom\'{i}r Ba\v{n}as}
\address{Department of Mathematics, Bielefeld University, 33501 Bielefeld, Germany}
\email{banas@math.uni-bielefeld.de}
\author{Benjamin Gess}
\address{Department of Mathematics, Bielefeld University, 33501 Bielefeld, Germany and Max Planck Institute for Mathematics in the Sciences, Inselstr. 22, 04103 Leipzig, Germany}
\email{bgess@math.uni-bielefeld.de}
\author{Christian Vieth}
\address{Department of Mathematics, Bielefeld University, 33501 Bielefeld, Germany}
 \email{cvieth@math.uni-bielefeld.de}

\begin{abstract}
We study a general class of singular degenerate parabolic stochastic partial differential equations (SPDEs) which include, in particular,
the stochastic porous medium equations and the stochastic fast diffusion equation.
We propose a fully discrete numerical approximation of the considered SPDEs based on the very weak formulation.
By exploiting the monotonicity properties of the proposed formulation
we prove the convergence of the numerical approximation towards the unique solution.
Furthermore, we construct an implementable finite element scheme for the spatial discretization of the very weak formulation
and provide numerical simulations to demonstrate the practicability of the proposed discretization.
\end{abstract}

\maketitle

\section{Introduction}
In this paper we study the numerical approximation of a class of singular-degenerate parabolic stochastic partial differential equations
\begin{equation}\label{spme0}
\d u = [\Delta (|u|^{p-2}u) + f]\d t +\sigma(u)\d \WW \qquad\mathrm{in}\,\, (0,T)\times\mathcal{D}\,,
\end{equation}
where $\mathcal{D}\subset \mathbb{R}^d$, $d\geq 1$ is a bounded, open domain and $\sigma(u) \WW$ is a multiplicative noise term
which will be specified below.

The above equation for $p>2$ is the stochastic \textit{porous medium equation} and for $p\in (1,2)$ the equation corresponds to the stochastic \textit{fast diffusion equation};
the case $p=2$ yields the \textit{stochastic heat equation}.

Stochastic quasilinear diffusion equations of the type \eqref{spme0} appear in several contexts, including, interacting branching diffusion processes \cite{DGG20}, self-organized criticality  \cite{BBDPR09,G15}, and non-equilibrium fluctuations in non-equilibrium statistical mechanics \cite{FG20,DSZ16}. We next present three of such instances in more detail.

{As a first example, consider the $\mathbb{H}^{-1}$ gradient flow structure of the porous medium equation
$$\partial_t\uu=-K_\uu\left(\frac{\delta E}{\delta\uu}(\uu)\right)=\Delta (|\uu|^{p-2}\uu)$$
with Onsager operator $K_\uu = -\Delta$ and energy $E(\uu)=\frac{1}{p}\int |\uu|^p\ dx$. The corresponding fluctuating system, in accordance with the GENERIC framework of non-equilibrium thermodynamics (see \cite{O05}), then reads
\begin{align}
  d\uu&=-K_\uu\left(\frac{\delta E}{\delta\uu}(\uu)\right)+B_\uu \d \WW \label{eq:fluct_SPDE_Generic}\,,\\
  &=\Delta (|u|^{p-2}u)+\sqrt{2\kappa_B}\ \mathrm{div}(\d \WW)\,, \label{eq:fluct_SPDE_6}
\end{align}
with $B_\uu B_\uu^*=2\kappa_B K_\uu$, $\kappa_B$ the Boltzmann constant and $\WW$ a vector-valued space-time white noise. Notably, the stochastic PDE \eqref{eq:fluct_SPDE_6} is super-critical and, thus, lacks a well-posedness theory. The results of the present paper are applicable to approximate versions of  \eqref{eq:fluct_SPDE_6}, that is, to
\begin{equation}\label{eq:fluct_SPDE_7}
  d\uu=\Delta (|u|^{p-2}u)+\sqrt{2\kappa_B}\ \mathrm{div}(\d \tilde\WW),
\end{equation}
where $\tilde\WW$ is a trace-class Wiener process in $\mathbb{L}^2$; in this case, in one spatial dimension, the stochastic perturbation $\mathrm{div}(\tilde \WW)$ still is less-regular than space-time white noise.

The second class of examples arises from fluctuations in non-equilibrium statistical mechanics. This leads to stochastic PDE of the general type
\begin{equation}\label{eq:fluct_SPDE}
  du=\Delta\alpha(u)\d t+\varepsilon^{\frac{1}{2}}\nabla\cdot(g(u)dW_t),
\end{equation}
where $dW$ denotes space-time white noise, with the Dean-Kawasaki stochastic PDE
$$du=\Delta u\ \d t+\varepsilon^{\frac{1}{2}}\nabla\cdot(\sqrt u dW_t),$$
as a model example, see for example \cite{D96,LKR19,DFVE14}. Stochastic PDE of this type serve as continuum models for interacting particle systems, including stochastic corrections reproducing the correct fluctuation behavior on the central limit and large deviations scale, see \cite{DFG20}. Since for large particle number the fluctuations decay, we see the small factor $\varepsilon^{\frac{1}{2}}$ in front of the noise. For example, a concrete example of an interacting particle process is given by the zero range process, see \cite{FG20,FG19}, leading to nonlinear, non-degenerate diffusion $\alpha$ in \eqref{eq:fluct_SPDE} and noise coefficients corresponding to $g(u)=\alpha^\frac{1}{2}(u)$. We note that with this choice \eqref{eq:fluct_SPDE} is in line with the GENERIC framework \eqref{eq:fluct_SPDE_Generic} when considering\begin{equation}\label{eq:fluct_SPDE_2}
  {\partial_t} u=\Delta\alpha(u)
\end{equation}
as a gradient flow on the space of measures with energy given by the Boltzmann entropy. The corresponding stochastic PDE \eqref{eq:fluct_SPDE} is super-critical and, therefore, lacks a well-posedness theory. Instead, one considers joint scaling limits $\varepsilon\to0,N\to\infty$ of 
\begin{equation}\label{eq:fluct_SPDE_3}
  du=\Delta\alpha(u)\d t+\varepsilon^{\frac{1}{2}}\nabla\cdot(g(u)dW^N),
\end{equation}
where $W^N$ is a regularized noise, see \cite{FG19,FG20}. In the case $\alpha' \ge c >0$ and $g$ Lipschitz continuous, this class of stochastic PDE is included in the results of the present work. 

The third class of equations covered by the present work arises in the continuum scaling limit of the empirical mass of interacting branching diffusions with localized interaction, which, informally, converges to the solution of a stochastic PDE
\begin{equation}\label{eq:branching_IPS}
d\uu=\Delta\uu^{2}\d t+(\uu c(\uu))^{\frac{1}{2}}\d W,
\end{equation}
where $\d W$ denotes space-time white noise, see \cite{DGG,Meleard}. The results of the present work apply to the particular case of $c(\uu)=\uu$ and $W$ being a trace class Wiener process in {$\Hh^{\frac{d+2}{2}}$}. 
}

It is common to these stochastic PDE that, due to the irregularity of the random perturbation, solutions are expected to be of low regularity. In fact, in many cases solution take values in spaces of distributions only, causing severe difficulties in even giving meaning to the nonlinear terms appearing in the stochastic PDE.

The lack of regularity of solutions is one of the decisive differences distinguishing the numerical analysis of stochastic PDE from deterministic PDE. While, if the noise and thus the solutions are regular enough, the numerical analysis can proceed similarly to the deterministic case, this ceases to be true in more rough situations. Indeed, if one considers \eqref{spme0} with regular enough noise, the solutions will take values in spaces of functions ($L^{p}$ spaces), and, therefore, standard finite element basis can be used, such as piecewise constant or piecewise linear functions. The proof of their convergence still requires adaptation from the deterministic arguments, e.g. replacing compactness arguments by a combination of tightness arguments and Skorohod's representation theorem (cf.~e.g.~\cite{gg_2019}), but the numerical method is close to the deterministic case. In contrast, when the noise is not as regular, one cannot expect to close $L^{p}$-based estimates, but one has to work in spaces of distributions. Concretely, this means to move from $L^{p}$-based estimates for \eqref{spme0} to $\mathbb{H}^{-1}$-based estimates. 

While the modification of finite element methods from $L^{2}$-based to $\mathbb{H}^{-1}$-based thus is necessary and natural in the context of stochastic PDE, this causes obstacles in their numerical realization: Precisely, while in an $L^{2}$-based approach, the choice of piecewise constant (or piecewise linear) finite elements $\phi_{i}$ leads to a sparse mass matrix
\[
(\tilde{\bf M}_h)_{i,j}=(\phi_{i},\phi_{j})_{\Ltwo},
\]
this is not true in the $\mathbb{H}^{-1}$-based approach which leads to a mass matrix
\begin{equation}\label{eq:intro-mass}
 ({\bf M}_h)_{i,j}=(\phi_{i},\phi_{j})_{\Hmone}=(\phi_{i},\mDeltaD^{-1}\phi_{j})_{\Ltwo}.
\end{equation}
Note that \eqref{eq:intro-mass} is not a sparse matrix, since $\mDeltaD^{-1}\phi_{j}$ has global support. Consequently, the resulting numerical scheme is inefficient. 

Interestingly, in one spatial dimension this difficulty was addressed in the contribution \cite{EmmrichSiska}, where an $\mathbb{H}^{-1}$ -based finite element scheme was suggested in the context of a \textit{deterministic} porous medium equation, motivated by the aim to treat irregular initial data and forcing. In \cite{EmmrichSiska} it was noticed, that in one spatial dimension a modified finite element basis $\tilde{\phi}_{i}$ can be constructed, leading to a sparse mass-matrix \eqref{eq:intro-mass}. In view of \eqref{eq:intro-mass} this requires to choose a basis so that $\mDeltaD^{-1}\phi_{j}$ has small support. While, in one spatial dimension, this can relatively easily be enforced by choosing $\phi_{i}$ of the form $$-a_{i-1}1_{[x_{i-1},x_{i})}+a_{i}1_{[x_{i},x_{i+1})}-a_{i+1}1_{[x_{i+1},x_{i+2})},$$ for $d\ge2$ this construction becomes less obvious. In addition, in higher dimension, the proof of the $L^{p}$-density of the resulting finite element spaces proves much more challenging. 

In the light of this exposition, the contribution of the present work is two-fold: Firstly, motivated by the intrinsic irregularity of stochastic PDE, we provide an $\mathbb{H}^{-1}$ based analysis of a fully discrete finite element scheme for \eqref{spme0} and prove its convergence. Secondly, we construct a finite element basis in dimension $d\ge2$, which allows for an efficient implementation
of the proposed numerical approximation in the $\mathbb{H}^{-1}$-setting, and analyze its approximation properties in $L^{p}$. 
More precisely, motivated by the deterministic numerical approximation \cite{EmmrichSiska} we propose a fully discrete finite element based numerical approximation
of \eqref{spme0} based on its very weak formulation.
We show that the proposed numerical approximation converges for {$p\in (1, \infty)$.
Furthermore,  we generalize the finite element spatial discretization of the very weak formulation, which was restricted to $d=1$ in \cite{EmmrichSiska},
to higher dimensions. 
Moreover, we present numerical simulations to demonstrate the efficiency and convergence behavior of the proposed numerical scheme.

The paper is organized as follows. In Section~\ref{sec_not} we state
the notation and assumptions along with the definition and basic properties of very weak solutions of \eqref{spme0}.
We introduce the fully discrete numerical approximation of \eqref{spme0} in Section~\ref{sec_num}
and show well-posedness of the proposed discrete approximation along with a priori estimates for the numerical solution.
The convergence of the numerical approximation towards the very weak solution of \eqref{spme0} is shown
in Section~\ref{sec_conv}.
In Section~\ref{sec_fem} we propose and analyze a non-standard finite element scheme for
the spatial discretization of the very weak solution which enables an efficient implementation
of the resulting fully discrete numerical approximation.
Numerical simulations which demonstrate the practicability of the proposed numerical scheme are presented in Section~\ref{sec_num_exp}.

\subsection*{Comments on the literature}
There exists a rich literature on the numerical approximation of deterministic degenerate parabolic equations, i.e. \eqref{spme0} with  $\sigma(u)\equiv 0$,
where the earlier results include \cite{MNV87}, \cite{JK91}. For more recent results we refer to \cite{EL08}, \cite{EmmrichSiska}, \cite{DEJ19}, \cite{DroniouFastDiff} and the references therein.
As far as we are aware, the only result on the numerical approximation of \eqref{spme0} so far is  \cite{gg_2019},
where the convergence of the proposed numerical approximation towards a martingale solution has been shown in dimension $d=1$ {for regular noise}
and a limited range of the exponent $p\in(2,3)$, not including the case of the stochastic fast diffusion equation.

In the deterministic setting, the analysis of the equation \eqref{spme0} is well understood, see, e.g.~\cite{Vazquez}.
In the stochastic setting, the well-posedness of \eqref{spme0} in the variational framework goes back to \cite{KR_SEE,P75} with many details given in \cite{LR15};
for a generalization of the variational approach to the case of the stochastic fast diffusion equation we refer to \cite{RRW07}. 
Generalizations to maximal monotone nonlinearities and Cauchy problems can be found in \cite{BDPR}, based on monotonicity techniques.
Martingale solutions for diffusion coefficients given as Nemytskii operators have been constructed in \cite{GRZ09}. In \cite{K06} the well-posedness for \eqref{spme0} with additive noise was shown based on a weak convergence approach. An $L^1$-based alternative approach to well-posedness has been developed based on entropy solutions in \cite{BVW15,DGG,Da.Ge20} and based on kinetic solutions in \cite{GS16-2,De.Ho.Vo2016,GH18,FG18,FG19}. Solutions to \eqref{spme0} with space time white multiplicative noise have been constructed in \cite{Da.Ge.Ge2020}.

Besides well-posedness, also the long-time behavior of solutions has been analyzed, see, for example, \cite{FG18} for the existence of random dynamical systems, \cite{BGLR10,G13-2} for the existence of random attractors, and \cite{BDPR,DaGeTs,W15-2} for ergodicity. For regularity of solutions we refer to \cite{G12,DdMH15,DG17,BR15} and the references therein. Results on finite speed of propagation and waiting times were derived in \cite{G13-1,BDPR,FG15}. Extensions to parabolic-hyperbolic SPDE may be found in \cite{BR18,BVW15}, and to doubly nonlinear SPDE in \cite{Sc.St2019} and the references therein.

\section{Notation and preliminaries} \label{sec_not}
Let  $\D\subset \mathbb{R}^d$  be a bounded open domain {with $\mathcal{C}^{1,1}$-smooth boundary}  $\partial\D$
or a {rectangular} domain.
For $1\leq p \leq \infty  $, we denote the conjugate exponent as $p' = \frac{p}{p-1}$.
We use the notation $(\mathbb{L}^p,\|\cdot\|_{\mathbb{L}^p})$ for the standard Lebesgue spaces of $p$-th order integrable functions on $\D$
and $(\Wkp{k}{p},\|\cdot\|_{\Wkp{k}{p}})$ for the standard Sobolev spaces on $\D$,
where $(\Wkpz{k}{p},\|\cdot\|_{\Wkpz{k}{p}})$ stands for the $\Wkp{k}{p}$ space with zero trace on $\partial \D$; 
for $p=2$ we denote the corresponding Sobolev spaces as $(\Hk,\|\cdot\|_{\Hk})$ and $(\Honez,\|\cdot\|_{\Honez})$.
We note that the dual space of  $\Honez$,  denoted by $(\Hmone,\|\cdot\|_{\Hmone})$, is a Hilbert space with the scalar product
$(v,w)_{\Hmone}:=( v, \mDeltaD^{-1}w )_{\Ltwo}= (\nabla \mDeltaD^{-1}v, \nabla \mDeltaD^{-1}w )_{\Ltwo}$
where $\mDeltaD^{-1}$ is the inverse Dirichlet Laplace operator $\mDeltaD^{-1}: \Hmone\rightarrow \Honez$. 

{Throughout the paper we denote $\Vv:=(\Lp\cap \Hmone)$, and $\Hh:=\Hmone$
and note that $\Vv\hookrightarrow \Hh \equiv \dual{\Hh}\hookrightarrow \dual{\Vv}$ constitutes a Gelfand triple
for the considered range of the exponent $p$ in $d\geq 1$ (for $p \geq 2$ one may take $\Vv\equiv \Lp$), cf., \cite{book_lions}.}

For $v\in \Hmone$ we define the inverse Laplace operator $\tilde{v}=:\mDeltaD^{-1}v$ as
the unique weak solution  of the problem
\begin{equation}
\begin{aligned}
-\Delta \tilde{v} = v &\quad \textnormal{ in } \D,\\
\tilde{v} = 0 & \quad\textnormal{ on } \partial \D\,.
\end{aligned}
\end{equation}
We note that {the above assumption on $\D$} guarantees that $\mDeltaD^{-1}v\in \Wkp{2}{p}\cap\Wkpz{1}{p}$
for $v\in \Vv\hookrightarrow\Hh= \Hmone$ and that $\tilde{v}$ depends continuously on $v$.

We consider $W$ to be a cylindrical Wiener process on a real separable Hilbert space $\Kk$, that is, for an orthonormal basis
$\{\tilde{e}_i\}_{i\in \mathbb{N}}$ of $\Kk$, we (formally) have  $W(t)=\sum_{i\in\mathbb{N}}\tilde{e}_i \beta_i(t)$ with $\{\beta_i(t)\}_{i\in \mathbb{N}}$ independent Brownian motions 
on a filtered probability space $(\Omega, {\mathcal F}, \{ {\mathcal F}_t\}_t, {\mathbb P})$.
Let $\HS{\Kk}{\Hh}$ denote the space of real Hilbert-Schmidt linear operators from $\Kk$ to $\Hh$. We note that ($\HS{\Kk}{\Hh}$,  $\|\,\cdot\,\|_{\HS{\Kk}{\Hh}}$, $(\,\cdot\,,\,\cdot\,)_{\HS{\Kk}{\Hh}}$) 
is a real separable Hilbert space with inner product
$$
(\sigma_1,\sigma_2)_{\HS{\Kk}{\Hh}}=\sum_{i=1}^\infty (\sigma_1\tilde{e}_i,\sigma_2\tilde{e}_i)_\Hh\,,
$$
and the corresponding norm $\|\sigma\|_{\HS{\Kk}{\Hh}}^2=\sum_{i=1}^\infty \|\sigma \tilde{e}_i\|_\Hh^2$.


We consider a slight generalization of the equation \eqref{spme0}:
\begin{subequations}\label{spme}
\begin{alignat}{2}
\d u &= [\Delta \alpha(u) + f]\d t +\sigma(u)\d \WW \quad &&\textnormal{ in }  (0,T)\times\D,\label{stochPMEa}\\
\alpha(u) &= g &&\textnormal{ on } (0,T)\times\partial \D,\\
u(0) &= u_0, &&\textnormal{ in } \D,
\end{alignat}
\end{subequations}
where $\alpha:\mathbb{R}\rightarrow \mathbb{R}$, and $\sigma:\Vv \rightarrow \HS{\Kk}{\Hh}$;
the initial condition $u_0\in L^2(\Omega,\Hh)$ is assumed to be $\mathcal{F}_0$-measurable.

To simplify the presentation we consider (progressively measurable)
{$f\in L^{\infty}(\Omega\times(0,T)\times\D)$ and  $g\in L^{\infty}(\Omega\times (0,T)\times \partial\D)$},
a generalization to less regular data is straightforward, cf. \cite{EmmrichSiska}.
Furthermore, we assume that the function $\alpha :\mathbb{R} \rightarrow \mathbb{R}$ is continuous, monotonically increasing,
and satisfies a coercivity and growth condition, i.e.,
\begin{equation}\label{ass_alpha}
\alpha(z)z\ge \mu |z|^p-\lambda \quad \mathrm{and}\quad  |\alpha(z)|\le c(|z|+1)^{p-1}\,,\quad \forall z\in \mathbb{R}\,,
\end{equation}
for some $p>1$ and $c$, $\mu>0$, $\lambda\ge 0$, respectively.

Clearly, $\alpha(z)\equiv |z|^{p-2}z$ yields the stochastic porous medium/fast diffusion equation \eqref{spme0}
and satisfies the above assumptions for $p>1$.

{We note that $\mDeltaD^{-1}v\in \Wkp{2}{p}$ for $v\in \Vv$
by standard elliptic regularity theory, cf. \cite[Ch. 9]{Gilbarg}, \cite[Ch. 4]{book_grisvard}.
Furthermore, for $v\in \Vv$ the normal trace of $\mDeltaD^{-1}v$ satisfies $\partial_{\vec{n}}\left(\mDeltaD^{-1}v\right)\in W^{1/p',p}(\partial\D)$ 
for domains with $\mathcal{C}^{1,1}$-smooth boundary or rectangular domains, cf.~\cite[Thm.~5.4-5.5 p.~97-99]{Necas}. 
Hence, it follows that $b\in L^{p'}(\Omega\times (0,T);\dual{\Vv})$. 
In the particular case $g\equiv 0$ the following also generalizes to convex domains with piecewise smooth boundary.}

\begin{comment}
\begin{rem}[Regularity of $\D$]
In $2d$  we have that $\Honez \subset \Lp$ (i.e. $\|(-\Delta_D^{-1}) v\|_{\Lp} \leq C\|(-\Delta_D^{-1}) v\|_{\Honez} \leq C$) for $p\in (0,\infty)$.
Then since $(-\Delta_D^{-1}) v \in \Honez$ on \blue{any bounded connected domain $\D\subset \mathbb{R}^2$}.
it follows that $(-\Delta_D^{-1}) v \in \Lp$  for $p\in (1,\infty)$
and consequently \eqref{def_ab} is well defined for any $v \in \mathbb{V}:= \Lp\cap \Hmone$.
(need to check the regularity of $\partial_{\vec{n}}\mDeltaD^{-1}v)_{L^2(\partial\D)}$ or just set $g=0$).

3d still needs to be checked, in 3d we have that $\Honez \subset \mathbb{L}^{q}$ up to $q \leq 6$, check check check.
But on convex domains we have also $\mathbb{H}^2$ regularity, so I think it will also work.

(note: holds for $f$, $g$ with sufficient regularity, $L^\infty$ should be more than enough)
\end{rem}
\end{comment}


Throughout the paper we assume that the following conditions are satisfied.
\begin{assumption}\label{AssA}
\begin{enumerate}[i)]
\item Hemi-continuity of $A$: the function 
\begin{align*}
\epsilon \mapsto \dpairing{A(w + \epsilon z)}{v}{\Vv}: [0,1]\rightarrow \mathbb{R}
\end{align*}
is continuous for all $v,w,z\in \Vv$.
\item Monotonicity of $A$: there exists $\lambda_B\ge 0$, such that  for all $v,w\in \Lp$
\begin{equation}\label{lambdab} 
2\dpairing{Av-Aw}{v-w}{\Vv}+\lambda_B\|v-w\|_{\Hh}^2\ge \|\sigma(v)-\sigma(w)\|_{\HS{\Kk}{\Hh}}^2\, .
\end{equation}

\item  Coercivity of $A$: for $\mu>0$ and $\lambda,\lambda_A,\kappa_\sigma \ge 0$ it holds
\begin{align}\label{Aux23}
\dpairing{Av}{v}{\Vv}+\lambda_A\|v\|_{\Hh}^2\ge\,& \mu \|v\|_{\Vv}^p -\lambda |\D| +\frac{1}{2}\|\sigma(v)\|_{\HS{\Kk}{\Hh}}^2-\kappa_\sigma.
\end{align}
\item Boundedness of $A$: there exists a $C> 0$ such that
$$\|Av\|_{\dual{\Vv}}\le C(\|v\|_{\Vv}+1)^{p-1} \quad \forall v \in \Vv
.$$
\end{enumerate}
\end{assumption}

We next generalize the concept of very weak solutions for the deterministic version of \eqref{spme} with $\sigma(u)\equiv 0$  from \cite{EmmrichSiska}
to the stochastic problem.
We consider the integral form of \eqref{spme} as
\begin{align*}
u(t) = u_0 + \int_0^t [\Delta \alpha(u(s)) + f(s)] \d s + \int_0^t \sigma(u(s))\d \WW(s).
\end{align*}
We multiply the above equation by $\tilde{v}=\mDeltaD^{-1}v$, integrate over $\D$, and integrate twice by parts in the second order term
to obtain, using the boundary condition,
\begin{align*}
(u(t),\mDeltaD^{-1}v)_{\Ltwo} &= (u_0,\mDeltaD^{-1}v)_{\Ltwo}-\int_0^t(\alpha(u(s)),v)_{\Ltwo}\d s\nonumber\\
&\quad -\int_0^t(g(s),\partial_{\vec{n}}\mDeltaD^{-1}v)_{L^2(\partial\D)}\d s\nonumber\\
&\quad +\int_0^t( f(s),\mDeltaD^{-1}v)_{\Ltwo}\d s\nonumber\\
&\quad +\int_0^t( \sigma(u(s))\d \WW(s),\mDeltaD^{-1}v)_{\Ltwo}\,.
\end{align*}

The above formal construction motivates the following definition of very weak solutions of the stochastic problem \eqref{spme}.
\begin{definition}\label{def_veryweak}
Let $u_0 \in L^2(\Omega, \mathcal{F}_0, \mathbb{P}; \Hh)$.
Then a $\mathcal{F}_t$-adapted {process}
{$u\in L^p(\Omega, \{ \mathcal{F}_t\}_t, \mathbb{P}; L^p((0,T);\Vv))\cap L^2(\Omega, \{ \mathcal{F}_t\}_t, \mathbb{P};C([0,T];\Hh))$}
is a \textbf{very weak solution} of \eqref{spme}
if it satisfies $\mathbb{P}$-a.s. for all $v\in \Vv$ and all $t\in [0,T]$:
\begin{align}
(u(t),v)_{\Hh}=&\; (u_0,v)_{\Hh}- \int_0^t \dpairing{A u(s)}{v}{\Vv} \d s\nonumber\\
&\quad + \int_0^t \dpairing{b(s)}{v}{\Vv} \d s + \int_0^t (\sigma(u(s))\d \WW(s),v)_{\Hh},\label{oSPME}
\end{align}
with
\begin{align}\label{def_ab}
\dpairing{A u(s)}{v}{\Vv}&= ( \alpha(u(s)),v)_{\Ltwo}\,,
\nonumber 
\\
\dpairing{b(s)}{v}{\Vv}&=( f(s),\mDeltaD^{-1}v)_{\Ltwo} - (g(s),\partial_{\vec{n}}\mDeltaD^{-1}v)_{L^2(\partial\D)}\,.
\end{align}
\end{definition}

\begin{rem}
Owing to the Assumption~\ref{AssA} 
we may interpret the very weak formulation of \eqref{spme} from Definition~\ref{def_veryweak}
as a monotone stochastic evolution equation  posed on the Gelfand triple 
$\Vv\hookrightarrow \Hh \equiv \dual{\Hh}\hookrightarrow \dual{\Vv}$,
cf. \cite[Th\'eor\`eme 3.1]{book_lions}, \cite{EmmrichSiska}.
Hence, the existence and uniqueness of the very weak solution in Definition~\ref{def_veryweak}
follows by the standard theory of monotone stochastic evolution equations \cite{KR_SEE}, {\cite{RRW07}}.
\end{rem}

{
Below we state examples of SPDE problems covered by the framework of Assumption \ref{AssA}; these include all of the problems mentioned in the introduction, in particular.
We let $\{e_k\}_{k\in\mathbb{N}}$ be an orthonormal basis of $\mathbb{L}^2$ consisting of eigenvectors of the Laplacian $-\Delta$ with Dirichlet boundary conditions and corresponding eigenvalues $\{\lambda_k\}_{k\in\mathbb{N}}$. 
We note that 
\begin{equation}\label{eq:growth_ef}
   \|e_k\|_{\mathbb{L}^\infty}\lesssim {\lambda_k}^{d/4}, \text{ and } \|\nabla e_k\|_{\mathbb{L}^\infty}\lesssim {\lambda_k}^{(d+2)/4}.
\end{equation}

\begin{example}[GENERIC framework for the $\mathbb{H}^{-1}$-gradient flow]
  We consider \eqref{eq:fluct_SPDE_7} with $p > 1$ and $W=\sum_{i=1}^\infty \eta_i e_i\beta_i$ a trace-class Wiener process in $\mathbb{L}^2$, i.e., $\sum_{i=1}^\infty\eta_i^2 < \infty$. 
Then, $ \mathrm{div}(W)$ is a trace-class Wiener process in $\mathbb{H}^{-1}$ 
and we choose $\Vv = \mathbb{L}^p\cap \mathbb{H}^{-1}$, $\Hh=\mathbb{H}^{-1}$, $\Kk=\Ltwo$,
$A(u)=-\Delta (|u|^{p-2}u)$ extended to $\Vv \to \Vv'$ and $\sigma(u)w \equiv\sigma w :=  \sum_{i=1}^\infty \eta_i (e_i,w)_{\mathbb{L}^2} \mathrm{div} (e_i)$. Then, Assumption \ref{AssA} can be verified analogously to \cite{LR15}. 
\end{example}

\begin{example}[Fluctuations in non-equilibrium systems]
We consider \eqref{eq:fluct_SPDE_3} so that $\alpha\in\mathcal{C}^1(\mathbb{R})$ satisfies $c^* < \alpha^\prime < C^*$ for some $c^*,C^*>0$, $g$ is Lipschitz continuous and $W=(\beta_1,\dots,\beta_N)$ is a $\R^N$-valued Brownian motion, that is, 
\begin{equation*}
  du=\Delta\alpha(u)dt+\varepsilon^{\frac{1}{2}}\nabla\cdot(g(u)\d W),
\end{equation*}
for $\varepsilon \le  \frac{c^*}{2 C(N)}$, where $C(N)= \left(\sum_{i=1}^N\|e_i\|_{\mathbb{L}^\infty}^2\right)$. We choose $\Vv = \Ltwo$, $\Hh=\mathbb{H}^{-1}$, $\Kk=\R^N$, $A(v)=-\Delta \alpha(v)$ extended to $\Vv \to \Vv'$, and
    $$\sigma(\uu)w:= \varepsilon^{\frac{1}{2}}\sum_{i=1}^N \nabla\cdot\big(g(\uu)e_i(w,\tilde e_i)_{{\mathbb{R}^N}}\big).$$
  We then have  
  \begin{equation*}\begin{split}
  &-2\dpairing{Av-Aw}{v-w}{\Vv}+\|\sigma(v)-\sigma(w)\|_{\HS{\Kk}{\Hh}}^2 \\
  &=-2\dpairing{Av-Aw}{v-w}{\Vv}+\sum_{j=1}^N\|\sigma(v)\tilde e_j-\sigma(w)\tilde e_j\|_{\Hh}^2 \\
  &= -(\alpha(v)-\alpha(w),v-w)_{\Ltwo} + \varepsilon  \sum_{j=1}^N \|\nabla\cdot (g(v)e_i) -\nabla\cdot (g(w)e_i)\|_{\mathbb{H}^{-1}}^2 \\
    &\le -c^* \|v-w\|_{\Ltwo}^2 + \varepsilon \left(\sum_{i=1}^N\|e_i\|_{\Lp[\infty]}^2\right) \|g(v)-g(w)\|_{\Ltwo}^2 \\
    &\le -c^* \|v-w\|_{\Ltwo}^2 + C(N)\varepsilon \|g\|_{Lip} \|v-w\|_{\Ltwo}^2 \le -\frac{c^*}{2}  \|v-w\|_{\Ltwo}^2.
  \end{split}\end{equation*}
  The remaining assumptions can be verified similarly. We note that the scaling relation $\varepsilon \le  \frac{c^*}{2 C(N)}$ implicitly depends on the dimension $d$, since the number of frequency modes $\le N$ depends on the dimension, cf. \cite{DFG20}.
\end{example}

\begin{example}[Branching interacting particle systems]
  We consider \eqref{eq:branching_IPS} with $c(\uu)=\uu$ and $\tilde{W}$ is a trace-class Wiener process in $\Hh^{1}$, 
   that is, 
  \begin{equation}\label{eq:branching_IPS_2}
     d\uu=\Delta\uu^{[2]}dt+\uu \d \tilde{W},
  \end{equation}
  with $\uu^{[2]} := |\uu|\uu$ and non-negative initial condition $\uu_0$. In order to fit this example in the abstract setup of Assumption~\ref{AssA} we choose $\Vv = \mathbb{L}^3$, $\Hh=\mathbb{H}^{-1}$, $\Kk=\ell^2$.
  Let $W$ be a cylindrical Wiener process on $\Kk$, $A(v)=-\Delta\uu^{[2]}$ extended to $\Vv \to \Vv'$, and
    $$\sigma(\uu)w:=\uu \sum_{i=1}^\infty e_i \eta_i(w,{\tilde e}_i)_{\ell^2},$$  
where $\eta_i>0$, $i\in \mathbb{N}$ satisfy {$\left(\sum_{i=1}^\infty \eta_i^2 \lambda_i^{\frac{d+2}{2}} \right)<\infty$. Note that then $\tilde W := \sum_{i=1}^\infty \eta_i e_i \beta_i$ defines a trace class Wiener process in $\Hh^\frac{d+2}{2}$.}  
  We then have, {by \eqref{eq:growth_ef}},
  \begin{equation*}\begin{split}
  &-2\dpairing{Av-Aw}{v-w}{\Vv}+ \|\sigma(v)-\sigma(w)\|_{\HS{\Kk}{\Hh}}^2 \\
  &= -(v^{[2]}-w^{[2]},v-w)_{L^2} +  \sum_{i=1}^\infty \|\sigma(v){\tilde e_i} -\sigma(w){\tilde e_i}\|_{\mathbb{H}^{-1}}^2 \\
  &\le  \sum_{i=1}^\infty \|(v-w) (e_i \eta_i)\|_{\mathbb{H}^{-1}}^2   
 {\le   \left(\sum_{i=1}^\infty \eta_i^2 \|e_i\|_{\mathbb{W}^{1,\infty}}^2 \right)  \|v-w\|_{\mathbb{H}^{-1}}^2}\\
&\le   \left(\sum_{i=1}^\infty \eta_i^2 \lambda_i^{\frac{d+2}{2} } \right)  \|v-w\|_{\mathbb{H}^{-1}}^2  \le C  \|v-w\|_{\mathbb{H}^{-1}}^2.
  \end{split}\end{equation*}
    The remaining assumptions can be verified similarly.
\end{example}
}

\section{Fully discrete Numerical Approximation} \label{sec_num}

We introduce a uniform partition of the time interval $[0,T]$ with a constant time-step size $\tau=T/N$, where $N\in \mathbb{N}$,
as $0=t_0< t_1<\ldots<t_N=T$ with $ t_n:=n\tau$.
For a mesh size $h \in (0, 1]$ we 
consider a  family of finite dimensional subspaces $(\Vh)_{h>0}\subset \Vv$
with the approximation property
\begin{equation}\label{vh_conv}
\inf_{v_h\in \Vh}\|v-v_h\|_{\Vv}\rightarrow 0 \quad\text{for }h\rightarrow 0, \quad \forall v\in \Vv\,,
\end{equation}
and let $\tilde{J}\equiv \tilde{J}_h=\mathrm{dim}(\Vh)$ for any $h>0$.
{We define a family of mappings $R_h:\Vv \rightarrow \Vh$ via the best approximation property, i.e.,  
$R_h v = \displaystyle \underset{v_h\in \Vh}{\mathrm{arg\ inf}}\|v-v_h\|_{\Vv}$ for $v\in \Vv$.
Furthermore, we denote by $P_h:\Hh\rightarrow \Vh$ the family of projection operators which satisfy}
$$
\lim_{h\rightarrow 0} \|w-P_hw\|_{\Hh}= 0\quad\forall w\in \Hh\,.
$$ 
An explicit construction of the discrete finite element spaces $\Vh$ and the operators $R_h$ and $P_h$ will be provided in
Section~\ref{sec_fem} below (see Lemma~\ref{lem_rh_conv}, Corollary~\ref{cor_rh} and Remark~\ref{rem_conv_ph}).

We define the discrete Brownian increments for $i= 1,2,\dots$ as
\begin{equation}\label{betainc}
\Delta_n \beta_i:=
\begin{cases} 0&\text{if }n = 1,\\
\beta_i(t_n)-\beta_i(t_{n-1})&\text{if }n = 2,\ldots, N\, ,
\end{cases}
\end{equation}
and for $r\in\mathbb{N}$ 
we define the truncated Hilbert-Schmidt operator $\sigma^r:\Vv\rightarrow \HS{\Kk}{\Hh}$ as 
$$
\sigma^r(u)w=\sum_{i=1}^r\sigma(u)\tilde{e}_i(w,\tilde{e}_i)_\Kk \qquad\text{for } w\in \Kk,
$$ 
where $\{\tilde{e}_i\}_{i\in\mathbb{N}}$ is the orthonormal basis of $\Kk$ and $u\in \Vv$.

The time-discrete approximation of the right-hand side $b$ (given in Definition~\ref{def_veryweak}) is obtained as
$$
b^n := \frac{1}{\tau} \int_{t_{n-1}}^{t_n} b(t) \d t \approx b(t_n)\,.
$$

Given $N\in \mathbb{N}$, $\tau=\frac{T}{N}$, $h>0$ and $r\geq 1$, the fully discrete approximation of \eqref{spme}
is obtained as follows: set $u_{h}^{0} = P_h u_0\in \Vh$, and  for $n = 1,\dots,N$ determine $u_h^n\in \Vh$ as the solution of the problem
\begin{align}\label{num_spme}
\left(u_h^n-u_h^{n-1},v_h\right)_{\Hh}+\tau \dpairing{A u_h^n}{v_h}{\Vv} = 
\tau  \dpairing{b^n}{v_h}{\Vv} + \left(\sigma^r(u_h^{n-1})\Delta_n \WW, v_h\right)_{\Hh}\,.
\end{align}
for all  $v_h\in \Vh$.
We note that the above scheme can be equivalently rewritten as
\begin{align}\label{num_spme_int}
\left(u_h^n,v_h\right)_{\Hh}+\tau \sum_{k=1}^n \dpairing{A u_h^k}{v_h}{\Vv} = & \left(u_h^0,v_h\right)_{\Hh}
\tau \sum_{k=1}^n \dpairing{b^k}{v_h}{\Vv}
\\ \nonumber
&
 + \sum_{k=1}^n\left(\sigma^r(u_h^{k-1})\Delta_k \WW, v_h\right)_{\Hh}\,.
\end{align}

\begin{rem}
{We note that the choice $\Delta_1 \beta_i \equiv 0$, $i\in\mathbb{N}$ in \eqref{betainc} 
is not strictly required but is convenient since it slightly
simplifies the notation and convergence analysis in Section~\ref{sec_conv} for $u_0\in \Hh$.
In particular, this choice enables to restate the numerical scheme \eqref{num_spme_int} 
in the form \eqref{NumScheme_t} with the ''shifted'' interpolant $\utau^-$ defined in \eqref{def_u_htau_minus}
which satisfies the estimate in Corollary~\ref{corollary_boundedness_spme}.

An alternative is to show the convergence by a density argument.
For $u_0\in \Hh$ one can consider a sufficiently regular sequence $u^k_0\rightarrow u_0$, $k\rightarrow \infty$,
set $\Delta_1 \beta_i \equiv \beta_i(t_1)-\beta_i(t_0)$ and define $\utau(t) = u^k_0$ for $t\in [0,\tau)$.
Then the stochastic integral $\int_{\tau}^{\theta_\tau^+(t)}$ in \eqref{NumScheme_t} is replaced by $\int_{0}^{\theta_\tau^+(t)}$
and Corollary~\ref{corollary_boundedness_spme} holds for each $k < \infty$.}
\end{rem}

The measurability of the fully discrete solution is a consequence of the following lemma, c.f. \cite[Lemma 3.2]{EmmrichSiska}, \cite[Lemma 3.8]{GyongyMillet}.
\begin{lemma} \label{lem_measur}
Let $(S,\Sigma)$ be a measure space. Let ${\bf f}:S\times \Vh\rightarrow \Vh$ be a function that is continuous in its first argument for every (fixed) $\alpha\in S$
and is $\Sigma$-measurable in its second argument for every (fixed) $X\in \Vh$.
If for every $\alpha\in S$ the equation ${\bf f}(\alpha, X)=0_{\Vh}$ has a unique solution $X={\bf g}(\alpha)$ then ${\bf g}:S\rightarrow \Vh$ is $\Sigma$-measurable.
\end{lemma}

The next lemma guarantees the existence, uniqueness and measurability of the fully discrete numerical approximation \eqref{num_spme}.
\begin{lemma}\label{ThmExUnFullDiscSol}
For any $h>0$, $u_h^0\in L^2(\Omega, \mathcal{F}_0,\mathbb{P}; \Hh)$, and {$\tau \le \frac{1}{\lambda_B}$}
there exists a unique solution $\left\{u_h^n\right\}_{n=1}^N$ of the  numerical scheme \eqref{num_spme}.
Furthermore, the $\Vh$-valued random variables $u_h^n$ are $\mathcal{F}_{t_n}$-measurable, $n = 1,\ldots, N$. 
\end{lemma}
\begin{proof}

We assume that for $u_h^0\in L^2(\Omega, \mathcal{F}_0,\mathbb{P}; \Hh)$ there exist  $\Vh$-valued random variables $\big\{u^j_h \big\}_{j=1}^{n-1}$ that satisfy \eqref{num_spme} and that $u_h^j$ are $\mathcal{F}_{t_j}$-measurable for $j = 1,\ldots, n-1$. 
We show the existence of $\Vh$-valued $u_h^n$, that satisfies \eqref{num_spme} and is $\mathcal{F}_{t_n}$-measurable.


For each $\omega\in \Omega$ the scheme \eqref{num_spme} defines a canonical mapping ${\bf h}_\omega:\mathbb{V}_h\rightarrow \mathbb{V}_h$ 
for which it holds ${\bf h}_\omega(u_h^{n}(\omega))\equiv 0$.
Consequently for $U\in\Vh$ we write
\begin{align*}
\langle{\bf h}_\omega(U), U\rangle_{\mathbb{V}_h} :=\,& \frac{1}{\tau} (U-u_h^{n-1}(\omega),U)_{\Hh} + \dpairing{A(U)}{U}{\Vv} \\
&\quad -\dpairing{b^n(\omega)}{U}{\Vv}-\left(\sigma^r(u_h^{n-1}(\omega))\frac{\Delta_n \WW(\omega)}{\tau},U\right)_{\Hh}\,.
\end{align*}
We note that
\begin{align*}
(U-u_h^{n-1}(\omega),U)_{\Hh} 
\ge \|U\|_{\Hh}^2- C\|u_h^{n-1}(\omega)\|_{\Hh} \|U\|_{\Vv}\,.
\end{align*}
Hence, using the coercivity Assumption~\ref{AssA}~iii) along with the embedding $\Vv\hookrightarrow {\Hh}$  we obtain
\begin{align*}
\langle {\bf h}_\omega(U), U \rangle_{\mathbb{V}_h}
&\ge \|U\|_{\Vv} \Bigg( \mu \|U\|_{\Vv}^{p-1}- \frac{C}{\tau}\|u_h^{n-1}(\omega)\|_{\Hh} 
   -C\left\|\sigma(u_h^{n-1}(\omega))\frac{\Delta_n \WW(\omega)}{\tau}\right\|_{\Hh}\Bigg)
\\
&\quad +\left(\frac{1}{\tau}-\lambda_A\right) \|U\|_{\Hh}^2+\frac{1}{2}\|\sigma(U)\|_{\HS{\Kk}{\Hh}}^2 - C(\lambda_A, \D, b^n)\,.
\end{align*}

We choose $R_\omega\geq C(\lambda_B, \D, b^n) > 0$ such that
\begin{align*}
 \mu R_\omega^{p-1}&- \frac{C}{\tau}\|u_h^{n-1}(\omega)\|_{\Hh} 
 -C\left\|\sigma(u_h^{n-1}(\omega))\frac{\Delta_n \WW(\omega)}{\tau}\right\|_{\Hh}\ge 1\,.
\end{align*}
Since $(1/\tau -\lambda_B) \ge 0$,
we get for $\|U\|_{\Vv} = R_\omega$ that
\begin{align*}
\langle {\bf h}_\omega(U), U\rangle_{\mathbb{V}_h} \geq 0\,.
\end{align*}
Consequently, for each $\omega\in \Omega$ the existence of $u_h^n(\omega)\in \Vh$ that satisfies \eqref{num_spme}
follows by the Brouwer's fixed point theorem \cite[Ch. II, Lemma 1.4]{Temam}.

To show uniqueness we consider $U$, $\tilde{U}\in \Vh$, such that ${\bf h}_\omega(U)={\bf h}_\omega(\tilde{U})\equiv 0$ and obtain by the monotonicity Assumption~\ref{AssA}~ii) that
\begin{align*}
0 &= \tau \langle {\bf h}_\omega(U)-{\bf h}_\omega(\tilde{U}), U - \tilde{U} \rangle_{\Vh}
=\| U-\tilde{U}\|_{\Hh}^2 +\tau \dpairing{A(U)-A(\tilde{U})}{U - \tilde{U}}{\Vv}\\
&\ge (1-\lambda_B\tau)\| U - \tilde{U}\|_{\Hh}^2 \ge 0,
\end{align*}
which yields the uniqueness of the discrete solution for $\tau \lambda_B < 1$.

Finally, the $\mathcal{F}_{t_n}$-measurability of the $u_h^n$ follows by Lemma~\ref{lem_measur}

\end{proof}

{Under a slightly stronger assumption on $\tau$ we obtain the following stability Lemma.}
\begin{lemma}\label{lemma_apriori_spme}
For $\tau\le \frac{1}{2(1+\lambda_B)}$
there exist constants $\mu >0$, $C\geq 0$ 
such that for $n=1,\ldots,N$ it holds
\begin{align}
&\E{\|u_h^n\|_{\Hh}^2 +\mu\tau\sum_{j=1}^n \|u_h^j\|_{\Vv}^p}\nonumber \le C\,,
\end{align}
and
$$
\E{\sum_{j=1}^n \tau \|Au_h^j\|_{\dual{\Vv}}^{p'} }\leq C\,.
$$
\end{lemma}
\begin{proof}
i) We set $v_h=u_h^j\in \Vh$ in \eqref{num_spme} with $n\equiv j$, use the identity $2(a-b,a)_{\Hh}=\|a\|_{\Hh}^2 - \|b\|_{\Hh}^2 + \|a-b\|_{\Hh}^2$ and by summing up the resulting equations for $j=1,\dots, n$ we get, that
\begin{align} \label{aux1_proof_apri}
&\|u_h^n\|_{\Hh}^2+\sum_{j=1}^n\|u_h^j-u_h^{j-1}\|_{\Hh}^2
+ 2\tau\sum_{j=1}^n\dpairing{A u_h^j}{u_h^j}{\Vv}
\nonumber
\\
& =\|u_h^0\|_{\Hh}^2+ 2\tau\sum_{j=1}^n\dpairing{b^j}{u_h^j}{\Vv}+ 2\sum_{j=1}^n\left(\sigma^r(u_h^{j-1})\Delta_j \WW ,u_h^j\right)_{\Hh}\, .
\end{align}
Using the Cauchy-Schwarz and Young's inequalities we estimate the stochastic term as
\begin{align*}
\left(\sigma^r(u_h^{j-1})\Delta_j \WW ,u_h^j\right)_{\Hh}
\le \left(\sigma^r(u_h^{j-1})\Delta_j \WW ,u_h^{j-1}\right)_{\Hh} + \frac{1}{2} \left\|\sigma^r(u_h^{j-1})\Delta_j \WW \right\|_{\Hh}^2 + {\frac{1}{2}} \| u_h^j- u_h^{j-1} \|_{\Hh}^2\,.
\end{align*}
On noting the independence of $\sigma^r(u_h^{j-1})$ and $\Delta_j \WW$ we estimate
\begin{align*}
\E{\left\|\sigma^r(u_h^{j-1})\Delta_j \WW \right\|_{\Hh}^2}
{=\tau \E{\|\sigma^r(u_h^{j-1})\|_{\HS{\Kk}{\Hh}}^2} \leq \tau \E{\|\sigma(u_h^{j-1})\|_{\HS{\Kk}{\Hh}}^2}}\, .
\end{align*}
Next, on recalling \eqref{def_ab}, using the boundedness of $f$, $g$ we deduce by the H\"older and Young inequalities that
\begin{align*}
&\dpairing{b^j}{u_h^j}{\Vv}\le C(p, f,g) + \frac{\mu}{2}\|u_h^j\|_{\Vv}^p\,.
\end{align*}
Hence, on recalling the coercivity Assumption~\ref{AssA}~$iii)$
and using the above inequalities we obtain after taking the
expectation in \eqref{aux1_proof_apri} that
\begin{align*}
&\E{\|u_h^n\|_{\Hh}^2  +\mu\tau\sum_{j=1}^n \|u_h^j\|_{\Vv}^p   } \\
&\le C + \E{\|u_h^0\|_{\Hh}^2}  +\tau(1+\lambda_B)\E{\sum_{j=1}^n\|u_h^j\|_{\Hh}^2 }\, .
\end{align*}
The first statement of the Lemma then follows after an application of the discrete Gronwall lemma for $\tau(1+\lambda_B)\leq \frac{1}{2}$.

ii) 
For the second estimate we use the boundedness Assumption~\ref{AssA}~$iv)$, $p' = \frac{p}{p-1}$ and obtain that
$$
\|Au_j^n\|_{\dual{\Vv}}^{\frac{p}{p-1}} \leq C_p(\|v\|_{\Vv}^p + 1)\,.
$$
Hence the second estimate follows by part~$i)$ of the proof.
\end{proof}

\begin{rem}
The assumption on the step-size $\tau$ in the above Lemma (which is required for the application of the discrete Gronwall lemma) is not too restrictive.
For instance, for the stochastic porous media equation \eqref{spme0}
with $\sigma(u)=u$ one may deduce for the constants in Assumption~\ref{AssA}, \eqref{lambdab}
that $\lambda_1=\lambda_2=0$, $\lambda_3=1$, $\lambda_B=2$. Consequently, we only require a mild condition $\tau\le \frac{1}{2(1+2)}=\frac{1}{6}$.
\end{rem}

\section{Convergence of the numerical approximation}\label{sec_conv}

Given the temporal partition $\{t_n\}_{n=0}^N$ with associated discrete random variables $\{u_h^n\}_ {n = 0}^N$
we define the piecewise constant time-interpolants for $t\in[0,T] $ as follows:
\begin{equation}\label{def_u_htau}
\utau(0)=u_h^1,\quad \utau(t)=u_h^n\qquad \text{ for }t\in (t_{n-1},t_n]
\end{equation}
and
\begin{align}\label{def_u_htau_minus}
 \utau^-(t)&= 0\quad\text{ for }t\in [0,t_1)=[0,\tau),\quad 
 \utau^-(t)=u_h^{n-1}\quad \text{ for }t\in [t_{n-1},t_n),\\
\utau^-(T)&=u_h^N.\nonumber
\end{align}
We note that the interpolant $\utau^-$ is $(\mathcal{F}_t)_{t\in [0,T]}$ adapted by Lemma~\ref{ThmExUnFullDiscSol}.

On recalling \eqref{num_spme_int}  we note that
the numerical scheme can be restated in terms of the above interpolants, i.e., it holds $\mathbb{P}$-a.s. that
\begin{align}
&\left(\utau(t),\vh\right)_{\Hh}+ \int_{0}^{\theta_\tau^+(t)}\dpairing{A \utau(s) -b_\tau(s)}{\vh}{\Vv} \d s\nonumber \\
&\hspace{1cm}= \left(u_h^0,\vh\right)_{\Hh} +  \int_{\tau}^{\theta_\tau^+(t)}\left(\sigma^r(\utau^-(s))\d \WW(s) ,\vh\right)_{\Hh}
\qquad \text{for all } t \in (0,T),\,\,\forall v \in \Vh\,,
\label{NumScheme_t}
\end{align}
where 
\begin{equation}
\theta_\tau^+(0):=0,\quad  \theta_\tau^+(t):=t_n \qquad\text{ for }t\in(t_{n-1},t_n],\quad n=1,\ldots,N\,.
\end{equation}

As a consequence of Lemma~\ref{lemma_apriori_spme} and Assumption~1 the time interpolants from \eqref{def_u_htau} and \eqref{def_u_htau_minus} satisfy the following a~priori estimates.
\begin{corollar}\label{corollary_boundedness_spme}
For any $h >0$ and (sufficiently small) $\tau>0$ it holds that
\begin{small}
\begin{align*}
i)\ & \sup_{t\in [0,T]} \E{\|\utau^-(t)\|_{\Hh}^2}\le C, &&ii)\ \sup_{t\in [0,T]} \E{\|\utau(t)\|_{\Hh}^2}\le C,\\
iii)\ &   \E{\int_0^T\|\utau^-(t)\|_{\Vv}^p \d t}\le C, && iv)\ \E{\int_0^T \|\utau(t)\|_{\Vv}^p \d t}\le C,\\
v)\ & \E{\int_0^T\|A\utau^-(t)\|_{\dual{\Vv}}^{p'}\d t}\le C, && vi)\ \E{\int_0^T\|A\utau(t)\|_{\dual{\Vv}}^{p'}\d t}\le C,
\end{align*}
and
\begin{align*}
vii)\ & \E{\int_0^T \|\sigma(\utau^-(t))\|_{\HS{\Kk}{\Hh}}^2\d t}\le C, \\
viii)\ & \E{\int_0^T \|\sigma(\utau(t))\|_{\HS{\Kk}{\Hh}}^2\d t}\le C,
\end{align*}
\end{small}
where $C> 0$ is a constant that only depends on the data of the problem.
\end{corollar}

\begin{comment}
The following Lemma can be found in \cite{GyongyMillet} and \cite[Lemma 4.2]{EmmrichSiska2} together with a proof.
\begin{lemma}\label{lemma_convergence_pw_const}
Let $X$ be a separable and reflexive Banach space and let $1< q< \infty$. Consider $\left( (x_\ell^n)_{n=0}^{N_\ell}\right)_{\ell\in\mathbb{N}}$ with $x_\ell^n\in L^q(\Omega;X)$ for all $n=0,1,\ldots,N_\ell$ and $\ell\in \mathbb{N}$. Let the piecewise-constant processes $\overline{x}_\ell$ and $\overline{x}_{\ell}^-$ be given by $\overline{x}_\ell(t_n)=\overline{x}_\ell^-(t_n)=x_\ell^n$ and
$$\overline{x}_\ell(t) = x_\ell^n \text{ if } t\in (t_{n-1},t_n)\quad \text{ and }\quad \overline{x}_\ell^-(t)= x_\ell^{n-1} \text{ if } t\in (t_{n-1},t_n),$$
for $n = 1\ldots,N_\ell$, $\ell\in \mathbb{N}$.\\
Assume that $\left( \overline{x}_\ell \right)_{\ell\in\mathbb{N}}$ and $\left( \overline{x}_\ell^- \right)_{\ell\in\mathbb{N}}$ are bounded in $L^q(\Omega\times (0,T);X)$, then there is a subsequence (denoted by $\ell'$) and $x$, $x^-\in$ $L^q(\Omega\times (0,T);X)$  such that
$$\overline{x}_{\ell'}\rightharpoonup x \text{ and } \overline{x}_{\ell'}^-\rightharpoonup x^- \text{ in } L^q(\Omega\times (0,T);X)$$
as $\ell'\rightarrow \infty$. Further $x= x^-$.
\end{lemma}
\end{comment}

From the a priori bounds in Corollary~\ref{corollary_boundedness_spme} we can directly deduce the following sub-convergence result.
\begin{lemma}\label{lem_weak_limits}
Let the Assumptions~\ref{AssA} hold and let $u_0\in L^2(\Omega, \mathcal{F}_0, \mathbb{P};\Hh)$.
Then there exists a subsequence $h,\tau,r$ (not relabeled) such that for $h,\tau\rightarrow 0$, $r\rightarrow \infty$ the following holds:
\begin{enumerate}[i)]
\item there is a progressively measurable $u\in L^p(\Omega\times (0,T); \Vv)$ such that 
$$
\utau^-\rightharpoonup u  \text{ and }\utau\rightharpoonup u \quad \text{ in }L^p(\Omega\times (0,T); \Vv).
$$ 
There is a $u_T \in L^2(\Omega;\Hh)$ such that 
$$
\utau^-(T)= \utau(T)\rightharpoonup u_T\quad\mathrm{in}\,\, L^2(\Omega,\Hh)\,.
$$
\item There exists a progressively measurable $a\in L^{p'}(\Omega\times(0,T);\dual{\Vv})$ such that $A\utau\rightharpoonup a$ in $L^{p'}(\Omega\times (0,T);\dual{\Vv})$.
There is a progressively measurable $\overline{\sigma} \in L^2(\Omega\times (0,T);\HS{\Kk}{\Hh})$ such that {$\sigma^{r}(\utau^-)$, $\sigma^r(\utau)$} and $\sigma(\utau)$ 
weakly converge to $\overline{\sigma}$ in $L^2(\Omega\times (0,T); \HS{\Kk}{\Hh})$.
\item for $(\d \mathbb{P} \times \d t)$-almost all $(\omega,t)\in \Omega\times(0,T)$ the following equation holds in $\dual{\Vv}$
\begin{equation}\label{eq_lim0}
u(t) = u_0 +\int_0^t b(s)-a(s) \d s +\int_0^t \overline{\sigma}(s) \d W(s),
\end{equation}
\item there is an $\Hh$-valued continuous version of $u$ (sill denoted by $u$) which satisfies \eqref{eq_lim0} and
\begin{align}\label{eq_lim_ito}
\|u(t)\|_{\Hh}^2&=\|u_0\|_{\Hh}^2 + \int_0^t \Big(2\dpairing{b(s)-a(s)}{u(s)}{\Vv} + \|\overline{\sigma}(s)\|_{\HS{\Kk}{\Hh}}^2\Big) \d s\\
&\qquad  +2\int_0^t (u(s), \overline{\sigma}(s) \d W(s))_{\Hh}.\nonumber
\end{align}
\item $u_T = u(T)$, i.e.~$\utau(T)\rightharpoonup u(T)$ in $L^2(\Omega;\Hh)$.
\end{enumerate} 
\end{lemma}
\begin{proof}
i) We deduce from Corollary~\ref{corollary_boundedness_spme}~iii),~iv) that $\utau^-\rightharpoonup u^-$ and $\utau\rightharpoonup u$
in $L^p(\Omega\times (0,T); \Vv)$. The limit are the same according to \cite[Lemma 4.2]{EmmrichSiska2} see also \cite[proof of Prop.~3.3]{Gyongy}.

{Item~$ii)$ of the Lemma follows from Corollary~\ref{corollary_boundedness_spme}~$vii)$ and $viii)$, the limits again coincide in $L^p(\Omega\times (0,T); \Vv)$ by the arguments from $i)$.}

To show part $iii)$ 
we consider $v = \psi \phi\in L^{\infty}(\Omega\times (0,T);\Vv)$ for {$\psi\in L^\infty(\Omega\times (0,T);\mathbb{R})$}, $\phi \in \Vv$.
We set $\vh=\psi\phi_h\in \Vh$ with $\phi_h=R_h\phi\in \Vh$ in \eqref{NumScheme_t},
integrate w.r.t. $t$ over $[0,T]$ and take the expectation to get\forben{as above suggest not to state $\omega$-dependence of $\utau(t)$ explicitly.}
\begin{align}
& \E{\int_0^T (\utau(t),v(t))_{\Hh}+\dpairing{\int_0^{t} A \utau(s)\d s}{v(t)}{\Vv}\d t}\nonumber 
\\
\label{eqh}
&= \mathbb{E}\left[\int_0^T (u_h^0,v(t))_{\Hh} +\dpairing{\int_0^{t} b_\tau(s)\d s}{v(t)}{\Vv}\right.
\\
&\quad + \left.  \left(\int_{0}^{t}\sigma^{r}(\utau^-(s))\d W(s),v(t)\right)_{\Hh} \d t  \right] \nonumber 
\\
\nonumber
&\quad  + \mathcal{R}_{1,\tauh}+\mathcal{R}_{2,\tauh}-\mathcal{R}_{3,\tauh}- \mathcal{R}_{4,\tauh}- \mathcal{R}_{5,\tauh}+\mathcal{R}_{6,\tauh}+\mathcal{R}_{7,\tauh}+\mathcal{R}_{8,\tauh}\,,
\end{align}
where 
\begin{align*}
\mathcal{R}_{1,\tauh}&:=\E{\int_0^T \dpairing{\int_t^{\theta_{\tau}^+(t)} b_\tau(s)-A\utau(s)\d s}{\vh(t)}{\Vv}\d t},\\
\mathcal{R}_{2,\tauh}&:=\E{\int_0^T\left(\int_{0}^{\tau}\sigma^r(\utau^-(s))\d W(s),\vh(t)\right)_{\Hh} \d t},\\
\mathcal{R}_{3,\tauh}&:=\E{\int_0^T\left(\int_{t}^{\theta_{\tau}^+(t)}\sigma^r(\utau^-(s))\d W(s),\vh(t)\right)_{\Hh} \d t},\\
\mathcal{R}_{4,\tauh}&:=(\utau,\vh-v)_{L^2(\Omega\times (0,T); {\Hh})},\\
\mathcal{R}_{5,\tauh}&: = \left\langle \int_0^{\cdot} A \utau(s)\d s,\vh-v \right\rangle_{L^{p'}(\Omega\times (0,T); \dual{\Vv}) \times L^{p}(\Omega\times (0,T); \Vv)},\\
\mathcal{R}_{6,\tauh}&:=(u_h^0,\vh-v)_{L^2(\Omega\times (0,T); {\Hh})},\\
\mathcal{R}_{7,\tauh}&:= \left\langle \int_0^{\cdot} b_\tau(s)\d s,\vh-v \right\rangle_{L^{p'}(\Omega\times (0,T);\dual{\Vv}) \times L^{p}(\Omega\times (0,T);\Vv)},\\
\mathcal{R}_{8,\tauh}&: = \left( \int_0^{\cdot}\sigma^r(\utau^-(s)) \d W(s),\vh-v \right)_{L^2(\Omega\times (0,T); {\Hh})}.
\end{align*}
By the boundedness of $b_\tau$ and $A\utau$ in $L^{p'}(\Omega\times (0,T);\dual{\Vv})$ and 
$\sigma(\utau^-)$ in $L^2(\Omega\times (0,T); \HS{\Kk,\Hh})$ and an application of It\^{o}'s isometry we get
$\mathcal{R}_{1,\tauh}$, $\mathcal{R}_{2,\tauh}$, $\mathcal{R}_{3,\tauh} \rightarrow 0$ for $\tauh \rightarrow 0$.

Further, the boundedness of $\utau$ in $L^2(\Omega\times (0,T);{\Hh})$ and $u_h^0$ in $L^2(\Omega;{\Hh})$ yields for $k=4,\ldots,8$ that
\begin{align*}
|\mathcal{R}_{k,\tauh}|&\le C\|v-\vh\|_{L^p(\Omega\times (0,T); \Vv)}\, .
\end{align*}
On recalling $v = \psi \phi$ and $\vh=\psi\phi_h\in \Vh$, $\phi_h=R_h\phi\in \Vh$ we deduce
by \eqref{vh_conv} that
\begin{align*}
\left\| \vh-v \right\|_{L^{p}(\Omega\times (0,T); \Vv)}&= \|\psi\|_{L^{p}(\Omega\times (0,T), \mathbb{R})}\|\phi-\phi_h\|_{\Vv} 
\rightarrow 0\qquad \mathrm{for}\,\, h\rightarrow 0.
\end{align*}
Hence, on noting Corollary~\ref{corollary_boundedness_spme} we conclude that $\mathcal{R}_{k,\tauh}\rightarrow 0$, $k=4,\dots,8$ for $h\rightarrow 0$.

Next, the weak convergence $A \utau \rightharpoonup  a$, $\sigma(\utau)\rightharpoonup  \overline{\sigma}$ implies for $h,\tau\rightarrow 0$, $r\rightarrow \infty$
\begin{align*}
&\E{\int_0^T \dpairing{\int_0^{t} A \utau(s)\d s}{v(t)}{\Vv}\d t}  \rightarrow \E{\int_0^T \dpairing{\int_0^{t} a(s) \d s}{v(t)}{\Vv}\d t},\\
& \mathbb{E}\left[\int_0^T \left(\int_{0}^{t}\sigma^{r}(\utau^-(s))\d W(s),v(t)\right)_{\Hh} \d t  \right] \rightarrow  \mathbb{E}\left[\int_0^T \left(\int_{0}^{t}\overline{\sigma}(s)\d W(s),v(t)\right)_{\Hh} \d t  \right].
\end{align*}
From the weak convergence of $\utau \rightharpoonup u$ in $L^2(\Omega\times (0,T);{\Hh})$ and the strong convergence of $u_h^0 \rightarrow u_0$ in $L^2(\Omega;{\Hh})$ 
we deduce that
\begin{align*}
\E{\int_0^T (\utau(t),v(t))_{\Hh} \d t} \rightarrow \E{\int_0^T (u(t),v(t))_{\Hh} \d t},
\intertext{and}
\E{\int_0^T (u_h^0,v(t))_{\Hh}  \d t  }\rightarrow \E{ \int_0^T(u_0,v(t))_{\Hh} \d t   } .
\end{align*}
Finally, since $b_\tau\rightarrow b$ in $L^{p'}(\Omega\times (0,T);\dual{\Vv})$ it follows that
\begin{align*}
&\E{\int_0^T \dpairing{\int_0^{t} b_\tau(s)\d s}{v(t)}{\Vv}\d t  } \rightarrow \E{\int_0^T \dpairing{\int_0^{t} b(s)\d s}{v(t)}{\Vv}\d t }.
\end{align*}
From the above convergence results we conclude, by taking $h,\tau \rightarrow 0$, $r\rightarrow\infty$ in \eqref{eqh} that
\begin{align*}
&\E{\int_0^T (u(t),v(t))_{\Hh}+\dpairing{\int_0^{t} a(s)\d s}{v(t)}{\Vv}\d t} \\
&= \E{\int_0^T (u_0,v(t))_{\Hh} +\dpairing{\int_0^{t} b(s)\d s}{v(t)}{\Vv}
+   \left(\int_{0}^{t}\overline{\sigma}(s)\d W(s),v(t)\right)_{\Hh} \d t },
\end{align*}
for all $v = \psi\phi$, $\phi\in \Vv$, which implies \eqref{eq_lim0}.

By the standard theory of monotone SPDEs, cf. \cite{KR_SEE} (or \cite{RRW07}),
part $iv)$ follows from $iii)$ by
the It\^{o} formula for the square of the $\Hh$-norm, which also implies that $u$ has an $\Hh$-valued
continuous modification (which we again denote by $u$) that satisfies \eqref{eq_lim0}.

Finally, to show $v)$ we note that $\utau(T)\rightharpoonup u_T$ by part $i)$ which together with $iii)$ implies
$$
u_T+\int_0^T a(s)\d s = u_0 +\int_0^T b(s)\d s+ \int_0^T \overline{\sigma}(s)\d W(s)\quad\mathrm{in}\,\,\mathbb{L}^{p'}\,.
$$
Since the continuous $\Hh$-valued modification of $u$ (cf. $iv)$) satisfies \eqref{eq_lim0} we may conclude that $u_T=u(T)$.
\end{proof}

The following variant of the Gronwall lemma, cf. \cite[Lemma 5.1]{EmmrichSiska2},
will be useful for the proof of the subsequent theorem.
\begin{lemma}\label{lemma_ES_51}
Let $a$ and $b$ be real-valued integrable functions such that for all $t\in [0,T]$
\begin{equation}\label{eq_ES_51}
a(t)\le a(0)+\int_0^t b(s)\d s,
\end{equation}
then for all $\lambda_B\ge 0$ and for all $t\in[0,T]$
\begin{equation}\label{eq_ES_52}
e^{-\lambda_B t} a(t)+\lambda_B\int_0^t e^{-\lambda_B s} a(s)\d s \le a(0) + \int_0^t e^{-\lambda_B s}b(s) \d s.
\end{equation}
Moreover, if equality holds in \eqref{eq_ES_51}, then equality holds in \eqref{eq_ES_52}.
\end{lemma}

In the next theorem we conclude that the weak limit of the numerical approximation from Lemma~\ref{lem_weak_limits}
is the very weak solution of the equation \eqref{spme}.
\begin{theorem}[Convergence of the numerical approximation]\label{thm_convergence_to_cont_spme}
Let the Assumption~\ref{AssA} hold and let $u_0\in L^2(\Omega, \mathcal{F}_0, \mathbb{P};\Hh)$. 
Then, for $h,\tau\rightarrow0$, $r\rightarrow \infty$ 
the fully discrete solution of scheme \eqref{NumScheme_t} converges to the unique very weak solution 
$u\in L^p(\Omega\times(0,T); \Vv)\cap L^2(\Omega;C([0,T];\Hh))$ of \eqref{spme}
in the sense of Definition~\ref{def_veryweak}.
\end{theorem}
\begin{proof}
We have shown in Lemma~\ref{lem_weak_limits} that every weak limit $u$ of the numerical approximation satisfies for $t\in [0,T]$
$$
u(t) = u_0 +\int_0^t b(s)-a(s) \d s +\int_0^t \overline{\sigma}(s) \d W(s)\,.
$$
Hence, it remains to show that $a = Au$, $\overline{\sigma} = \sigma(u)$. 

Throughout the proof we use the shorthand notation $\ell := (h,\tau,r)$ and $\ell \rightarrow \infty$
stands for $h,\tau \rightarrow 0$, $r\rightarrow \infty$.
We define 
$$ \Xi_\ell(t):=\begin{cases}\|\utau(t)\|_{L^2(\Omega;\Hh)}^2&\text{if }t\in(0,T],\\
\|u_h^0\|_{L^2(\Omega;\Hh)}^2&\text{if }t=0 \,.
\end{cases}
$$
Analogously to the proof of Lemma \ref{lemma_apriori_spme}
we deduce from \eqref{aux1_proof_apri} on noting the definition of the time interpolants \eqref{NumScheme_t} that
 for any $t\in(0,T]$ it holds
\begin{align*}
\Xi_\ell(t) & \leq  \Xi_\ell(0)
\\ 
& \qquad + \mathbb{E}\bigg[\int_{0}^{t}2\dpairing{b_\tau(s)-A \utau(s)}{\utau(s)}{\Vv}  + \|\sigma^r(\utau(s))\|_{\HS{\Kk}{\Hh}}^2 \d s\bigg]  +\mathcal{R}_\ell(t)\,,
\end{align*}
with $\mathcal{R}_\ell(t)=\E{\int_{t}^{\theta_\tau^+(t)}2\dpairing{b_\tau(s)-A \utau(s)}{\utau(s)}{\Vv} + \|\sigma^r(\utau(s))\|_{\HS{\Kk}{\Hh}}^2 \d s}$.

We use Lemma \ref{lemma_ES_51} and obtain from the above inequality that
\begin{align}
&e^{-\lambda_B T}\Xi_\ell(T) \le \Xi_\ell(0)-\lambda_B \int_0^T e^{-\lambda_B s} \Xi_\ell(s)\d s \nonumber\\
&\qquad + \E{\int_{0}^{T}e^{-\lambda_B s}\left(2\dpairing{b_\tau(s)-A \utau(s)}{\utau(s)}{\Vv} + \|\sigma^r(\utau(s))\|_{\HS{\Kk}{\Hh}}^2\right) \d s}\label{eq_ES_53} \\
& \qquad +\lambda_B \int_0^T e^{-\lambda_B s} | \mathcal{R}_\ell(s)| \d s\nonumber.
\end{align}
Note that by the monotonicity property \eqref{lambdab} it holds for arbitrary $w\in L^{p}(\Omega\times (0,T);\Vv)$ that
\begin{align*}
&-2\E{\int_0^T e^{-\lambda_B s}\dpairing{A \utau(s)}{\utau(s)}{\Vv}\d s}\\
 &\le \mathbb{E}\Bigg[\int_0^T e^{-\lambda_B s} \Big(- \|\sigma(\utau(s)) -\sigma(w(s))\|_{L_2(\Kk,\Hh)}^2  +\lambda_B \|\utau(s)-w(s) \|_{\Hh}^2 \Big)\d s\Bigg] \\
&\quad -2 \mathbb{E}\Bigg[\int_0^T e^{-\lambda_B s}\Big(\dpairing{A w(s)}{\utau(s)-w(s)}{\Vv} +\dpairing{A \utau(s)}{w(s)}{\Vv}\Big) \d s\Bigg].
\end{align*}
We substitute the above inequality into \eqref{eq_ES_53} and obtain
\begin{align}
&e^{-\lambda_B T}\|\utau(T)\|_{L^2(\Omega;{\Hh})}^2 \nonumber\\
& \le \|u_{h}^0\|_{L^2(\Omega;{\Hh})}^2 + 2\E{\int_{0}^{T}e^{-\lambda_B s}\dpairing{b_\tau(s)}{\utau(s)}{\Vv} \d s}\nonumber\\
&\quad  + \mathbb{E}\Bigg[\int_0^T e^{-\lambda_B s} \Big(- \|\sigma(w(s))\|_{L_2(\Kk,\Hh)}^2 +2\left(\sigma(\utau(s)),\sigma(w(s))\right)_{L_2(\Kk,\Hh)} \nonumber\\
&\qquad \qquad + \lambda_B \|w(s) \|_{\Hh}^2 - 2\lambda_B \left(\utau(s),w(s) \right)_{\Hh} \Big)\d s\Bigg] \label{eq_ES_54}\\
&\quad -2\mathbb{E}\Bigg[\int_0^T e^{-\lambda_B s}\Big(\dpairing{A w(s)}{\utau(s)-w(s)}{\Vv}+\dpairing{A \utau(s)}{w(s)}{\Vv}\Big) \d s\Bigg]\nonumber\\
&\qquad\qquad +\lambda_B \int_0^T e^{-\lambda_B s} | \mathcal{R}_\ell(s)| \d s\,. \nonumber
\end{align}
Next, we observe that, by Corollary~\ref{corollary_boundedness_spme},
\begin{align*}
&\lambda_B \int_0^T e^{-\lambda_B } | \mathcal{R}_\ell(t)| \d t  \\
&\le \tau\lambda_B\Big(2 \left( \|b_\tau\|_{L^{p'}(\Omega\times (0,T);\dual{\Vv})}+\|A \utau\|_{L^{p'}(\Omega\times (0,T);\dual{\Vv})}\right)\|\utau\|_{L^{p}(\Omega\times (0,T); \Vv)}  \\
& \hspace{8cm}+ \|\sigma(\utau)\|_{L^2(\Omega\times (0,T);L_2(\Kk,\Hh))}^2\Big)\\
&\le C \tau\rightarrow 0\quad \text{for } \ell \rightarrow \infty \,.
\end{align*}
Hence, using the weak convergence of Lemma~\ref{lem_weak_limits}~i),~ii) we deduce from \eqref{eq_ES_54} by the lower-semicontinuity of norms that
\begin{align}
&e^{-\lambda_B T}\|u(T)\|_{L^2(\Omega;{\Hh})}^2 \le \liminf_{\ell\rightarrow\infty} e^{-\lambda_B T}\|\utau(T)\|_{L^2(\Omega;{\Hh})}^2\nonumber \\
&\le \|u_0\|_{L^2(\Omega;{\Hh})}^2 + 2\E{\int_{0}^{T}e^{-\lambda_B s}\dpairing{b(s)}{u(s)}{\Vv} \d s}\nonumber \\
&\quad  + \mathbb{E}\Bigg[\int_0^T e^{-\lambda_B s} \Big(- \|\sigma(w(s))\|_{L_2(\Kk,\Hh)}^2 +2(\overline{\sigma}(s),\sigma(w(s)))_{L_2(\Kk,\Hh)} \label{eq_ES_55} \\
&\qquad \qquad + \lambda_B \|w(s) \|_{\Hh}^2 - 2\lambda_B (u(s),w(s) )_{\Hh} \Big)\d s\Bigg]\nonumber \\
&\quad -2\mathbb{E}\Bigg[\int_0^T e^{-\lambda_B s}\Big(\dpairing{A w(s)}{u(s)-w(s)}{\Vv} +\dpairing{a(s)}{w(s)}{\Vv}\Big) \d s\Bigg].\nonumber
\end{align}

After a standard stopping argument and taking the expectation in \eqref{eq_lim_ito} we get for all $t\in [0,T]$
\begin{align*}
&\|u(t)\|_{L^2(\Omega;{\Hh})}^2= \|u_0\|_{L^2(\Omega;{\Hh})}^2 + \E{\int_{0}^{t}2\dpairing{b(s)-a(s)}{u(s)}{\Vv}+ \|\overline{\sigma}(s)\|_{\HS{\Kk}{\Hh}}^2 \d s}\nonumber.
\end{align*}
Using Lemma \ref{lemma_ES_51} we obtain from the above equality that
\begin{align}
&e^{-\lambda_B T}\|u(T)\|_{L^2(\Omega;{\Hh})}^2 =\|u_0\|_{L^2(\Omega;{\Hh})}^2- \lambda_B\E{\int_0^T e^{-\lambda_B s}\|u(s)\|_{{\Hh}}^2 \d s} \nonumber\\
&\quad +\E{\int_{0}^{T}e^{-\lambda_B s}\left(2\dpairing{b(s)-a(s)}{u(s)}{\Vv}+ \|\overline{\sigma}(s)\|_{L_2(\Kk,\Hh)}^2 \right)\d s}.\label{eq_ES_56}
\end{align}
Next, we subtract \eqref{eq_ES_56} from \eqref{eq_ES_55} and get
\begin{align*}
0
&\le \mathbb{E}\Bigg[\int_0^T e^{-\lambda_B s} \Big(- \|\sigma(w(s))-\overline{\sigma}(s)\|_{L_2(\Kk,\Hh)}^2  + \lambda_B \|w(s) - u(s)\|_{{\Hh}}^2 \Big)\d s\Bigg]\\
&\quad -2\mathbb{E}\Bigg[\int_0^T e^{-\lambda_B s}\Big(\dpairing{A w(s)}{u(s)-w(s)}{\Vv}-\dpairing{a(s)}{u(s)- w(s)}{\Vv}\Big) \d s\Bigg]\,.
\end{align*}
Consequently, it holds that
\begin{align}
&2\E{\int_0^T e^{-\lambda_B s}\dpairing{A w(s)}{u(s)-w(s)}{\Vv} \d s}\nonumber \\
&\le \mathbb{E}\Bigg[\int_0^T e^{-\lambda_B s} \Big( \lambda_B \|w(s) - u(s)\|_{{\Hh}}^2+2\dpairing{a(s)}{u(s)- w(s)}{\Vv} \Big)\d s\Bigg]\,.\label{eq_ES_58}
\end{align}
On taking  $w= u$ in \eqref{eq_ES_58} we get
$$
0\le - \mathbb{E}\Bigg[\int_0^T e^{-\lambda_B s}\|\sigma(u(s))-\overline{\sigma}(s)\|_{L_2(\Kk,\Hh)}^2\d s\Bigg]\le 0\,,
$$
which implies $\sigma(u(s))=\overline{\sigma}(s)$ in $L^2(\Omega\times (0,T);L_2(\Kk,\Hh))$.

Next we choose  $w = u -\varepsilon z$ \eqref{eq_ES_58} with $z\in L^p(\Omega\times (0,T); \Vv)$, $\varepsilon\in(0,1)$
\begin{align*}
& \E{\int_0^T e^{-\lambda_B s}\dpairing{A (u(s) -\varepsilon z(s))}{z(s)}{\Vv} \d s}\\
&\le \mathbb{E}\Bigg[\int_0^T e^{-\lambda_B s} \Big( \frac{1}{2}\varepsilon\lambda_B \| z(s)\|_{{\Hh}}^2+\dpairing{a(s)}{z(s)}{\Vv} \Big)\d s\Bigg]\, ,
\end{align*}
and obtainusing Assumption~\ref{AssA}~i) by the Lebesgue dominated convergence for $\varepsilon\rightarrow 0$ that
$$
\E{\int_0^T e^{-\lambda_B s}\dpairing{A u(s)}{z(s)}{\Vv} \d s}\le  \mathbb{E}\Bigg[\int_0^T e^{-\lambda_B s} \dpairing{a(s)}{z(s)}{\Vv} \d s\Bigg]\,.
$$
This implies that $a = Au$, since $z\in L^p(\Omega\times (0,T);\Vv)$ is arbitrary.


Finally, we conclude by the uniqueness of the very weak solution, that the whole sequence converges to the same limit $u$.

\end{proof}

\section{Practical finite element approximation in $\Lp$} \label{sec_fem}

A natural approach is to construct the numerical solution $u_{h}^n\in \mathbb{V}_h\subset \mathbb{L}^p$, $n=0,\dots,N$
using a finite element space $\mathbb{V}_h$ consisting of piecewise constant functions on a given partition of the domain $\D$ with a given mesh size $h$.
However, the piecewise constant finite element approximation of the very weak formulation is impractical since
the resulting finite element matrix associated with the $\Hh$-scalar product 
$(\cdot,\cdot)_{\Hh} = (\cdot, \mDeltaD^{-1} \cdot)_{\Ltwo} = (\nabla\mDeltaD^{-1}\cdot, \nabla\mDeltaD^{-1} \cdot)_{\Ltwo}$ in the discrete very weak formulation 
\eqref{num_spme} will be dense. Furthermore, the evaluation of the $\Hh$-inner product requires the evaluation of the inverse Laplace operator $\mDeltaD^{-1}$,
which does not have an explicit formula in general.
This is a consequence of the fact that the inverse Laplacian of the characteristic function $\chi_{\mathcal{T}}$ for some subset $\mathcal{T}\subset \D$
does not have compact support in $\D$, i.e., in general $\mathrm{supp}\{\mDeltaD^{-1}\chi_{\mathcal{T}}\} \equiv \D$. A further complication lies in the fact
that there is no explicit formula available for $\mDeltaD^{-1}\chi_{T}$, in general.

Below, we discuss the construction of a finite element basis $\{\phi_i \}_{i=1}^{\tilde{J}}$ of $\mathbb{V}_h$ for $d\geq 1$ on rectangular domains
with the property that $\psi_i := \mDeltaD^{-1}\phi_i$ can be computed explicitly and has local support in $\D$ for $i=1,\dots,\tilde{J}$.

\subsection{Finite-element basis in $d=1$}
We summarize the  finite element method proposed in \cite{EmmrichSiska} for $\D \subset \mathbb{R}^1$.
For the domain $\D = (-L, L)$, where $L>0$ we introduce a partition into disjoint open intervals $\{(\bfx_{i-1},\bfx_i)\}_{i=1}^J$, $\bfx_0=-L$, $\bfx_J=L$
such that $\overline{\D} = \cup_{i=1}^J[\bfx_{i-1},\bfx_i]$
and denote $\chi_{I}$ to be the characteristic function of the interval $I$.
We then set $\Vh= \mathrm{span} \{\phi_i,\,\,i=1,\dots,J\}\subset \Lp$
where $\phi_i:[-L,L]\rightarrow\mathbb{R}$ are defined as 
\begin{align}\label{phi_1}
\phi_1(x)&=\frac{3}{2}\chi_{[\bfx_0,\bfx_1]}(x)-\frac{1}{2}\chi_{(\bfx_1,\bfx_2]}(x)\,,
\\ \label{phi_i}
\phi_i(x)&=-\frac{1}{2}\chi_{(\bfx_{i-2},\bfx_{i-1}]}(x)+\chi_{(\bfx_{i-1},\bfx_{i}]}(x)-\frac{1}{2}\chi_{(\bfx_{i},\bfx_{i+1}]}(x)\,,
\\ \label{phi_J}
\phi_J(x)&=-\frac{1}{2}\chi_{(\bfx_{J-2},\bfx_{J-1}]}(x)+\frac{3}{2}\chi_{(\bfx_{J-1},\bfx_{J}]}(x)\,.
\end{align}
for any $x\in (-L,L)$.

Note that the proposed approximation is equivalent to
a piecewise constant approximation, i.e., $\Vh\equiv\mathrm{span}\{\phi_i\}=\mathrm{span}\{\chi_{(\bfx_{i-1},\bfx_{i}]},\,\, i=1,\dots,J\}$. 
The proposed basis has the useful property that $\psi_i:=\mDeltaD^{-1}\phi_i$  (with $\mDeltaD^{-1}$ defined on $(-L,L)$) admits an explicit representation for all $i=1,\ldots,J$
which has a small support in $\D$. It can be verified by direct calculation that
\begin{equation}\label{psi_1}
\psi_{1}(x)=\begin{cases}
-\frac{3}{4}(x-\bfx_0)^2+h(x-\bfx_0)&\text{if }x\in [\bfx_0,\bfx_1],\\
\frac{1}{4}(x-\bfx_1)^2-\frac{h}{2}(x-\bfx_1)+\frac{h^2}{4}&\text{if }x\in (\bfx_1,\bfx_2],\\
0&\text{otherwise},
\end{cases}
\end{equation}
further
\begin{equation}\label{psi_i}
\psi_{i}(x)=\begin{cases}
\frac{1}{4}(x-\bfx_{i-2})^2&\text{if }x\in (\bfx_{i-2},\bfx_{i-1}],\\
-\frac{1}{2}(x-\bfx_{i-1}-\frac{h}{2})^2+\frac{3h^2}{8}&\text{if }x\in (\bfx_{i-1},\bfx_i],\\
\frac{1}{4}(\bfx_{i+1}-x)^2&\text{if }x\in (\bfx_{i},\bfx_{i+1}]\\
0&\text{otherwise,}
\end{cases}
\end{equation}
for $i=2,\ldots,J-1$, and
\begin{equation}\label{psi_J}
\psi_{J}(x)=
\begin{cases}
\frac{1}{4}(\bfx_{J-1}-x)^2-\frac{h}{2}(\bfx_{J-1}-x)+\frac{h^2}{4}& \text{if }x\in (\bfx_{J-2},\bfx_{J-1}]\\ 
-\frac{3}{4}(\bfx_J-x)^2+h(\bfx_J-x)&\text{if }x\in (\bfx_{J-1},\bfx_J]\\ 
0&\text{otherwise,}
\end{cases}
\end{equation}
We note that both basis have a small support in $\D$, i.e., $\mathrm{supp}(\phi_j) = \mathrm{supp}(\psi_j)$, $j=1,\dots,J$ with 
$$
\mathrm{supp}(\phi_j)=\begin{cases} [\bfx_0,\bfx_2]&\text{if }i=1\,, 
\\ 
[\bfx_{j-2},\bfx_{j+1}]& \text{for }j=2,\ldots,J-1\,,
\\
 [\bfx_{J-2},\bfx_J]&\text{if }j=J \,.
\end{cases}
$$
Consequently, the ''mass'' matrix $$\mathbf{M}_h = \{m_{ij}\}_{i,j=1}^J := \{(\phi_j, \mDeltaD^{-1} \phi_i)_{\Ltwo}\}_{i,j=1}^J\equiv \{(\phi_j, \psi_i)_{\Ltwo}\}_{i,j=1}^J$$ which corresponds to the $\Hh$-inner product in the numerical scheme \eqref{num_spme} will be sparse. 

\subsection{Spatial discretization in higher dimensions}

We consider $\D =(-L,L)^d$ for some $L>0$, $d=1,2,\dots$, and denote $x=(x_1,\dots,x_d)^T\in \D$. 
Given $m\in\mathbb{N}$ we set $J:=2^m$ and consider a uniform partition of $\D$ 
with mesh size $h=\frac{2L}{J}$ into $\tilde{J}:=J^d$ rectangular
subdomains $\D_\ui:=(\bfx_{i_1-1},\bfx_{i_1}]\times(\bfx_{i_2-1},\bfx_{i_2}]\times \cdots \times (\bfx_{i_d-1},\bfx_{i_d}]$
for a multiindex $\ui\in \{1,\ldots,J\}^d$, where $\ui:=(i_1,\ldots,i_d)$, $i_k=1,\ldots,J$ for $k=1,\ldots,d$,
and $\bfx_{i_k}:=-L+i_kh$.
We denote the above partition of the domain $\D$ as $\mathcal{T}_h = \{\D_\ui,\,\,\ui\in \{1,\ldots,J\}^d\}$.

We consider $\phi_{i_k}$, $\psi_{i_k}= \mDeltaD^{-1}\phi_{i_k}$, $i_k=1,\dots,J$ to be the one dimensional basis functions defined in the previous section and construct the basis functions $\{\bfphi_\ui\}$, $\ui\in \{1,\ldots,J\}^d$ of $\mathbb{V}_h$ in $\mathbb{R}^d$ as follows: for $\ui\in \{1,\ldots,J\}^d$ we set
\begin{align}
\bfphi_\ui(x) =\,&\left(\frac{3}{d^{1/(d-1)}}\frac{1}{h^2}\right)^{d-1}\sum_{k=1}^d \phi_{i_k}(x_k)\prod_{\substack{l=1\\ l \ne k}}^d \psi_{i_l}(x_l)\label{def_phi_i_d}\\
=\,&\sum_{k=1}^d \phi_{i_k}(x_k)\prod_{\substack{l=1\\ l \ne k}}^d \left(\frac{3}{d^{1/(d-1)}}\frac{1}{h^2}\psi_{i_l}(x_l)\right)\qquad x\in \D\,. \nonumber
\end{align}
On noting $\psi_{i_k}= \mDeltaD^{-1}\phi_{i_k}$ it can be deduced from \eqref{def_phi_i_d} by a direct calculation that $\bfpsi_i=\mDeltaD^{-1}\bfphi_\ui$ can be expressed explicitly as
\begin{align}
\bfpsi_\ui(x)&= \left(\frac{3}{d^{1/(d-1)}}\frac{1}{h^2}\right)^{d-1} \prod_{k=1}^d \psi_{i_k}(x_k)\label{def_psi_i_d}\\
&= \left(\frac{d^{1/(d-1)}}{3} h^2\right)\prod_{k=1}^d \left(\frac{3}{d^{1/(d-1)}}\frac{1}{h^2}\psi_{i_k}(x_k)\right) \quad x\in \D\,. \nonumber
\end{align}
Equivalently the basis functions $\bfpsi_\ui(x)$, $\ui\in \{1,\ldots,J \}^d$  are the solutions of the Poisson problem
\begin{alignat*}{2}
-\Delta \bfpsi_\ui&=\bfphi_\ui&&\quad\text{ in }\D=(-L,L)^d,\\
\bfpsi_\ui&=0 &&\quad\text{ on }\partial\D\,.
\end{alignat*}

An example of a basis function for $2\leq i_k\leq J-1$ for $d=2$ is given in Figure~\ref{fig_basis2d}.
\begin{figure}
\centering
\includegraphics[trim=100 0 50 150,clip, scale=0.36]{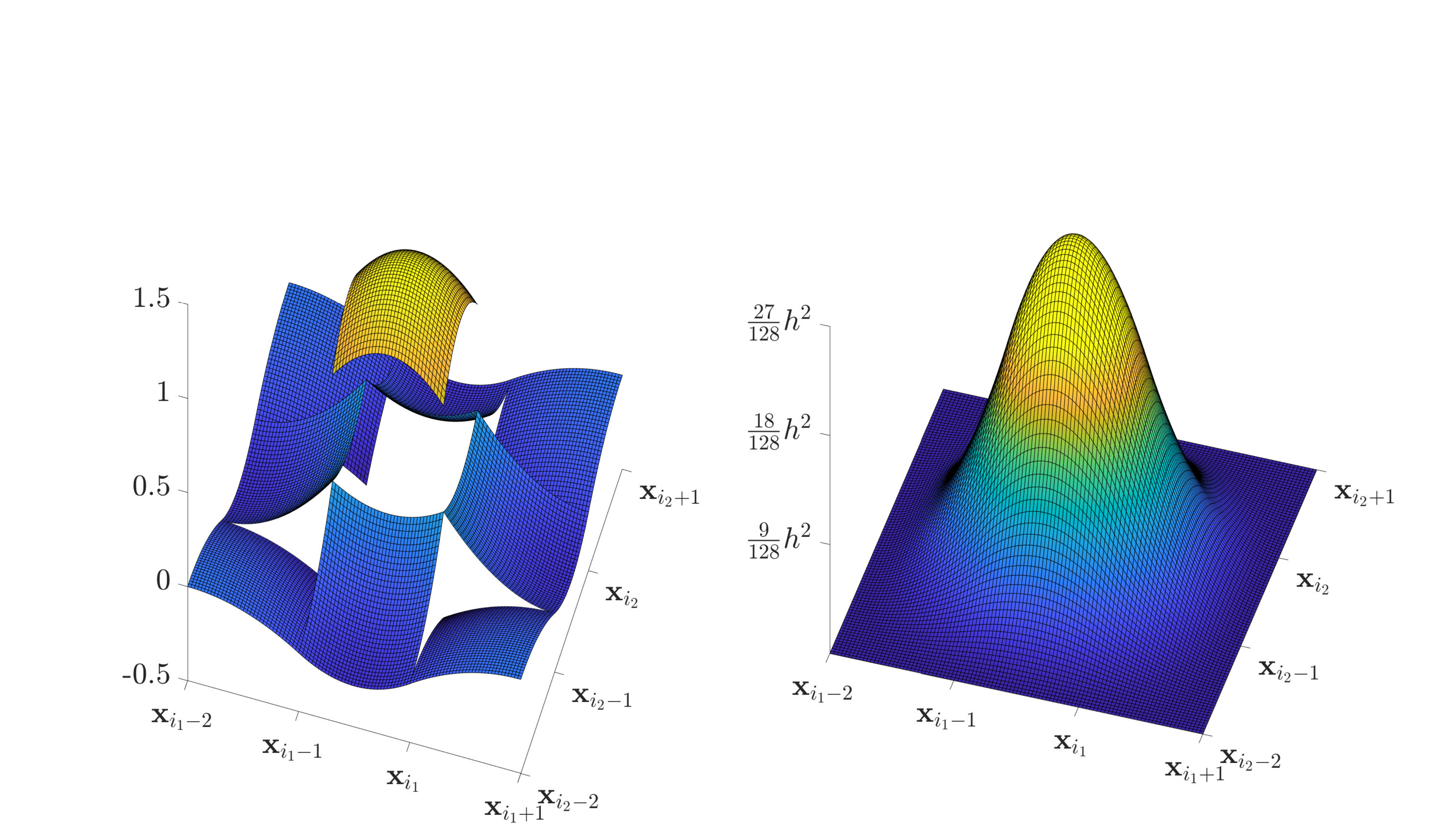}
\caption{$\bfphi_{(i_1,i_2)}$ and $\bfpsi_{(i_1,i_2)}$ for $d=2$}
\label{fig_basis2d}
\end{figure}

Clearly $\bfpsi_\ui\in \mathcal{C}^1(\bar\D)$ since $\psi_{i_k}\in \mathcal{C}^1([-L,L])$ for all $k=1,\ldots,d$. 
In addition, since 
$\psi_{i_k}$, $\phi_{i_k}\subset \mathbb{R}$ have a small local support in $[-L,L]$, 
also $\mathrm{supp}(\bfpsi_\ui)= \bigtimes_{k=1}^d \mathrm{supp}(\psi_{i_k})$ and 
$\mathrm{supp}(\bfphi_\ui)= \bigcup_{k=1}^d \mathrm{supp}(\phi_{i_k}) \times \left(\bigtimes_{\substack{l=1\\ l \ne k}}^d \mathrm{supp}(\psi_{i_l}) \right)\subset\mathbb{R}^d$
remain ''small''. {Consequently, the "mass" matrix for $d\ge 1$ 
$$
\mathbf{M}_h = \{m_{ij}\}_{i,j=1}^{\tilde{J}} := \{(\bfphi_j, \mDeltaD^{-1} \bfphi_i)_{\Ltwo}\}_{i,j=1}^{\tilde{J}}\equiv \{(\bfphi_j, \bfpsi_i)_{\Ltwo}\}_{i,j=1}^{\tilde{J}}
$$ 
is sparse; more precisely, there are only $5^d$ non-zero elements in each row of $\mathbf{M}_h$.}

By construction, the finite element space $\mathbb{V}_h$ 
consists of (discontinuous) piecewise polynomial functions on the rectangular partition $\mathcal{T}_h$ of the domain $\D$.
In order to analyze the approximation properties of $\mathbb{V}_h$ in $\mathbb{L}^p$ it is convenient to
consider the space of piecewise constant functions on $\mathcal{T}_h$ which is denoted as $\overline{\mathbb{V}}_h=\mathrm{span}\{\bfchi_{\ui}\}$ where $\bfchi_\ui:=\mathbbm{1}_{\D_\ui}$.

We define the restriction operator $\overline{R}_h:\Lp\rightarrow\overline{\mathbb{V}}_h$ as
\begin{equation}\label{def_orh}
\overline{R}_h v(x):=\sum_{\ui\in\{1,\ldots,J\}^d} \overline{v}_\ui \bfchi_{\ui}(x)\,,
\end{equation}
where $\DD \overline{v}_\ui = \left(\frac{1}{|\D_\ui|}\int_{\D_\ui} v(y)\d y\right)$.

Next we analyze the properties of the operator $\overline{R}_h$.
\begin{lemma}\label{lem_orh}
For any $p\geq 1$ the operator $\overline{R}_h$ is $\mathbb{L}^p$-stable, i.e.,  $\|\overline{R}_h v\|_{\Lp}\le \|v\|_{\Lp}$ for all $v\in \Lp$, and for all $v\in \Wkp{1}{p}$ it holds that
$$
\|v- \overline{R}_h v\|_{\Lp}\le C h\|\nabla v\|_{\Lp}\,.
$$
\end{lemma}
\begin{proof}
The $\mathbb{L}^p$-stability follows from the definition of $\overline{R}_h$ by the H\"older inequality as
\begin{align*}
\|\overline{R}_h v\|_{\Lp}^p
\le \sum_{\ui\in\{1,\ldots,J\}^d}|\D_\ui| \int_{\D_\ui}|v(y)|^p\d y\left( \int_{\D_\ui}\frac{1}{|\D_\ui|^{p/(p-1)}}\d y \right)^{p-1} 
 = \|v\|_{\Lp}^p\,.
\end{align*}

%
Next, we assume that $v$ is smooth, the result for $v\in \Wkp{1}{p}$ follows by density.
By the fundamental theorem of calculus 
and the H\"older inequality we get that
\begin{align*}
&\|v-\overline{R}_h v\|_{\Lp}^p\le \sum_{\ui\in\{1,\ldots,J\}^d}\frac{1}{|\D_\ui|}\int_{\D_\ui} \int_{\D_\ui}|v(x)-v(y)|^p\d y \d x\\
&\leq d^{p-1} h^p \sum_{k=1}^d\sum_{\ui\in\{1,\ldots,J\}^d}\int_{\D_\ui} |\partial_{x_k}v(x)|^p \d x =C(p,d)h^p\|\nabla v \|_{\Lp}^p.
\end{align*}
\end{proof}

\begin{lemma}
$\{\overline{\mathbb{V}}_h\}_{h>0}$ is a Galerkin scheme for $\Lp$, {$p\geq 1$}. I.e., for every $v\in \Lp$ it holds that
$$
\inf_{\overline{v}_h\in \overline{\mathbb{V}}_h}\|v-\overline{v}_h\|_{\Lp}\rightarrow 0 \text{ for } h\rightarrow 0\,.
$$
\end{lemma}
\begin{proof}
By density of $\Wkp{1}{p}\hookrightarrow \Lp$ we deduce from Lemma~\ref{lem_orh} that
\begin{align}\label{orh_conv}
\|v-\overline{R}_h v \|_{\Lp}
& \rightarrow 0\text{ for } h\rightarrow 0\quad \forall v\in \Lp.
\end{align}
Since $\overline{R}_hv\in \overline{\mathbb{V}}_h$ we get from the above that
$$
\inf_{\overline{v}_h\in \overline{\mathbb{V}}_h}\|v-\overline{v}_h\|_{\Lp}\le \|v-\overline{R}_hv \|_{\Lp}\rightarrow 0 \text{ for } h\rightarrow 0\,.
$$
\end{proof}

For the (piecewise polynomial) basis functions $\bfphi_\ui$ defined in \eqref{def_phi_i_d}
we denote $\obfphi_\ui:=\overline{R}_h\bfphi_\ui\in \overline{\mathbb{V}}_h$
and observe that $\overline{\mathbb{V}}_h=\mathrm{span}\{\bfchi_\ui\}=\mathrm{span}\{\obfphi_\ui\}$.


In order to show the approximation property of
the finite element space $\Vh:=\mathrm{span}\{\bfphi_\ui\}\subset \Lp$ 
we define the restriction operator $R_h:\Lp\rightarrow \Vh$ as
\begin{equation}\label{def_rh}
R_h v(x):=\sum_{\ui\in\{1,\ldots,J\}^d} v_\ui \bfphi_\ui(x) \,,
\end{equation}
where $v_\ui = \DD \frac{8}{3h^2}\frac{1}{|\D_\ui|}\int_{\D_\ui}\mDeltaD^{-1} v(y) \d y$.

For simplicity we restrict the proof of the convergence of the above restriction operator to $d=2$ and assume that $\D$ is a rectangle;
we expect an analogous proof to hold for $d\geq 3$ and more general domains as well. 
For $n\in \mathbb{N}$ we denote by $\mathbb{V}_n := \mathrm{span}\{e_k,\,\, k=0,\dots,n\}$ the finite-dimensional space spanned by the
the first $n$ eigenfunctions of the homogeneous Dirichlet Laplace operator on the rectangular domain $\D=(-L,L)\times(-L,L)$
\begin{equation}
e_k(x_1,x_2)=\sin\left(2\pi k  \frac{x_1+L}{2L}\right)\sin\left(2\pi k \frac{x_2+L}{2L}\right),\quad k\in\mathbb{N}.\label{Fourier_eb}
\end{equation}
{By the density of $\cup_{n\in\mathbb{N}}\mathbb{V}_n$ in $\Lp$ it suffices to show the
convergence of the restriction operator \eqref{def_rh} for $v\in\mathbb{V}_n$.} 
\begin{lemma}\label{lem_rh_conv}
Let $n\in \mathbb{N}$ be fixed. For any $p\geq 1$ and $v \in \mathbb{V}_n$ it holds that
$$
\|v- R_h v\|_{\Lp}\longrightarrow 0 \quad \text{for }h\rightarrow 0\,.
$$
\end{lemma}
\begin{proof}
{It is enough to show that the statement holds for $v \equiv e_k$, $k\in\mathbb{N}$.}

For $x=(x_1, x_2)\in \D$ we consider the following discrete Laplace operator
\begin{align}
\nonumber
(-\Delta_h^9) u(x_1,x_2) := \frac{8}{3h^2}&\left[  u(x_1,x_2)-\frac18 u(x_1+h,x_2+h)-\frac18 u(x_1,x_2+h)\right.\\
\nonumber
&\quad-\frac18 u(x_1-h,x_2+h)-\frac18 u(x_1+h,x_2)\\
\label{deltah}
&\quad-\frac18 u(x_1-h,x_2)-\frac18 u(x_1+h,x_2-h)\\
\nonumber
&\quad\left.-\frac18 u(x_1,x_2-h)-\frac18 u(x_1-h,x_2-h)\right].
\end{align}
The discrete Laplace operator $-\Delta_h^9$ corresponds 
to the 9-point finite difference approximation of the Laplace operator,
cf. \cite[p. 190, Example 4]{book_fd}; see also Figure \ref{fig_Stencil9}.
\begin{figure}[htp!]
\centering
\includegraphics[width=0.3\textwidth]{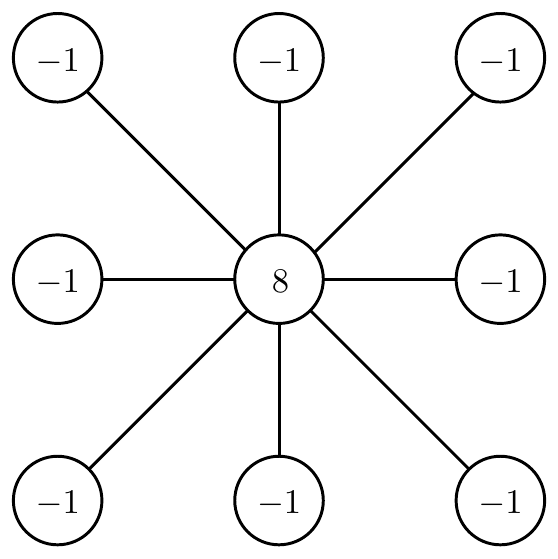}
\caption{Finite difference stencil related to the Discrete Laplace operator $-\Delta_h^9$.}
\label{fig_Stencil9}
\end{figure}
We note that for $u\in \mathcal{C}^4(\overline{\D})$ the discrete Laplace operator (\ref{deltah}) satisfies
the consistency property
\begin{equation}\label{consist}
(-\Delta_h^9)u(x)+ \Delta u(x)=\mathcal{O}(h^2) \qquad \forall x \in \D.
\end{equation}

{With each element $\D_{\ui}\in\mathcal{T}_h$ we associate the corresponding basis functions $\bfphi_{\ui}$, $\bfpsi_{\ui}$.
To deal with the complication that the basis functions associated
with the elements of the partition $\mathcal{T}_h$ along the boundary of the domain $\D$ have a different shape (c.f., \eqref{def_phi_i_d} for $i_1,i_2=1,J$ and \eqref{phi_1}, \eqref{phi_J}),
we introduce a layer of $4(J+1)$ ''ghost'' cells $\D_{(0,i_2)}^*$, $\D_{(J+1,i_2)}^*$, $\D_{(i_1,0)}^*$, $\D_{(i_1,J+1)}^*$, $i_1,i_2=0,\dots,J+1$
(the dimensions of the cells will be specified below)  along the outer side of the boundary of $\D$. We then denote the resulting extended partition with $(J+2)^2$ cells as 
$\mathcal{T}_h^*= \mathcal{T}_h\cup \{\D_{(i_1,i_2)}^*\}$, i.e., $\mathcal{T}_h^*$ includes the elements of $\mathcal{T}_h$ and the ''ghost'' cells.
}

{Recall the following trivial symmetry properties of the eigenfunctions $e_k$ from \eqref{Fourier_eb} 
(as well as for $(-\Delta^{-1}_D) e_k$, since $(-\Delta^{-1}_D) e_k = \lambda_k e_k$) which hold along the boundary of $\D$:
$e_k(-L-x_1,x_2)=-e_k(-L+x_1,x_2)$, $e_k(L+x_1,x_2)=-e_k(L-x_1,x_2)$, 
and $e_k(-L-x_1,-L-x_2)=e_k(-L+x_1,-L+x_2)$, $e_k(L+x_1,-L-x_2)=e_k(L-x_1,-L+x_2)$.
We note that (for ghost cells $\D_{\ui}^*$ with dimensions given implicitly via the definition \eqref{phi_phis}) the symmetry also transfers to the piecewise constant approximation of
$e_k$ over $\mathcal{T}_h^*$, i.e., for $\overline{R}_h e_k$ naturally extended on $\mathcal{T}_h^*$.
We will use this fact to construct an "extension" of $R_h$ from \eqref{def_rh} on $\mathcal{T}_h^*$ (see \eqref{def_rh_ext} below).
}

{ 
We consider a (modified) finite element basis associated with the elements of the extended partition $\mathcal{T}_h^*$ 
with $(J+2)^2$ basis functions which are defined
as \eqref{def_phi_i_d} with the exception that we only use the (suitably shifted) "interior" basis functions \eqref{phi_i}, \eqref{psi_i}.
Namely, we use \eqref{def_phi_i_d} where for $i_1 = i$, $i_2=i$ we set for $i = 0,\dots, J+1$
\begin{align}
\label{phis_i}
\phi_i^*(x)&=-\frac{1}{2}\chi_{(\bfx_{i-2},\bfx_{i-1}]}(x)+\chi_{(\bfx_{i-1},\bfx_{i}]}(x)-\frac{1}{2}\chi_{(\bfx_{i},\bfx_{i+1}]}(x)\,,\quad 
\end{align}
where we define $\bfx_{-1}= -L - (\bfx_{1} - \bfx_{0})$, $\bfx_{J+1}= L + (\bfx_{J} - \bfx_{J-1})$
(i.e., we replace the basis functions \eqref{phi_1}, \eqref{phi_J} and \eqref{psi_1}, \eqref{psi_J} by their "interior" counterparts);
we proceed analogously for the basis functions $\psi_1$, $\psi_J$, i.e., replace \eqref{psi_1}, \eqref{psi_J} by a suitably shifted analogues
$\psi_1^*$, $\psi_J^*$ of \eqref{psi_i}.

We note that the ''boundary'' basis functions satisfy $\phi_1(x)|_{(\bfx_{0}, \bfx_{1})} = (\phi_1^*(x) - \phi_{0}^*(x))|_{(\bfx_{0}, \bfx_{1})}$,
$\phi_J(x)|_{(\bfx_{J-1}, \bfx_{J})} = (\phi_J^*(x) - \phi_{J+1}^*(x))|_{(\bfx_{J-1}, \bfx_{J})}$ (and similarly for $\psi_1$, $\psi_J$).
We deduce from \eqref{def_phi_i_d} that analogous relations also hold for $\bfphi^*_\ui$ and $\bfphi_\ui$ (as well as for $\bfpsi^*_\ui$ and $\bfpsi_\ui$)
for instance it holds at the bottom boundary (analogically for the top, left and right boundaries)
\begin{equation}\label{phi_phis}
\bfphi_{(i_1,1)}|_{\D_{(i_1,1)}}= \big(\bfphi_{(i_1,1)}^* - \bfphi_{(i_1,0)}^*\big)|_{\D_{(i_1,1)}}\,,
\end{equation}
and similarly for $\bfphi_{(i_1,1)}|_{\D_{(i_1+1,1)}}$, $\bfphi_{(i_1,1)}|_{\D_{(i_1-1,1)}}$.
Slightly modified relations hold for the basis functions associated with the corner elements
$\D_{(1,1)}$, $\D_{(1,J)}$, $\D_{(J,1)}$, $\D_{(J,J)}$ of $\mathcal{T}_h$; for instance for $\D_{(1,1)}$ we deduce
\begin{align}\label{phi_phis_corner}
\bfphi_{(1,1)}|_{\D_{(1,1)}} &= \big(\bfphi_{(1,1)}^* - \bfphi_{(1,0)}^*- \bfphi_{(0,1)}^* + \bfphi_{(1,1)}^*\big)|_{\D_{(1,1)}}\,,
\\ \nonumber
\bfphi_{(1,1)}|_{\D_{(2,1)}} & = \big(\bfphi_{(1,1)}^* - \bfphi_{(1,0)}^*\big)|_{\D_{(2,1)}}\,,
\\ \nonumber
\bfphi_{(1,1)}|_{\D_{(1,2)}} & = \big(\bfphi_{(1,1)}^* - \bfphi_{(0,1)}^*\big)|_{\D_{(1,2)}}\,,
\end{align}
and similarly for basis functions at $\D_{(1,J)}$, $\D_{(J,1)}$, $\D_{(J,J)}$.

On noting the aforementioned symmetry properties of eigenfunctions $e_k$ and the relations \eqref{phi_phis}, \eqref{phi_phis_corner} (along with their counterparts covering the remaining situations)
we observe that \eqref{def_rh} for $v \equiv e_k$ is equivalent to 
\begin{equation}\label{def_rh_ext}
R_h v(x)|_{\D} \equiv \sum_{\ui\in\{0,1,\ldots,J,J+1\}^d} v_\ui \bfphi^*_\ui(x)\,,
\end{equation}
where $\{\bfphi^*_\ui\}$ is the previously constructed extended basis of "interior" basis functions associated with elements of $\mathcal{T}_h^*$.

The equivalent representation \eqref{def_rh_ext} of the restriction operator \eqref{def_rh} simplifies the subsequent considerations, 
since it only involves one type of (interior) basis functions.
For the rest of the proof we will work with the basis functions $\bfphi^*_\ui$ but
drop the superscript ''$^*$'' to simplify the notation (also note $\bfphi^*_{(i_1,i_2)}\equiv \bfphi_{(i_1,i_2)}$ for $1< i_1,i_2 < J$, i.e., the modification is only required at the boundary).
}

We consider an element $\D_\ui \subset \D$.
By a direct calculation of the elementwise mean of the basis functions \eqref{def_phi_i_d} for $d=2$ (i.e., evaluating $\overline{\bfphi}_\uj\equiv \overline{R}_h \bfphi_\uj$),
we note that for $x\in \D_\ui$, fixed $\ui = (i_1, i_2)$ it holds that $\overline{\bfphi}_\ui(x) \equiv 1$ 
and $\overline{\bfphi}_\uj(x) \equiv -\frac{1}{8}$ 
for $\uj\in\mathcal{N}(\ui):=\{\uj\in \{1,,\ldots,J\}^2;\,\,\overline{\D}_\uj \cap \overline{\D}_\ui \neq \emptyset\}\equiv\{\uj = (i_1+k_1,i_2+k_2);\, k_1,k_2=-1,0,1\}$, 
$\uj\neq \ui$, cf. Figure~\ref{fig_Stencil9}; below we denote $\uk=(k_1,k_2)\in\{-1,0,1\}^2$ the local index of $\uj$ with respect to $\ui$ and write $\uj\equiv\mathrm{glob}_{\ui}(\uk)$.
Consequently, we observe that the coefficients in the definition of the discrete Laplace operator \eqref{deltah} for $x\in \D_\ui$
correspond to the values $\overline{\bfphi}_\uj|_{\D_\ui}$, $\uj\in \mathcal{N}(\ui)$, scaled by the factor $\frac{8}{3h^2}$.

Hence, from the above observation, noting the definitions \eqref{def_rh}, \eqref{def_orh} and recalling \eqref{deltah} we deduce for $x\in \D_\ui$  that
\begin{align}\nonumber
\overline{R}_h  [R_h v](x) & = \sum_{\uj\in\mathcal{N}(\ui)} v_\uj \overline{\bfphi}_\uj(x)  = \frac{8}{3h^2}\sum_{\uj\in\mathcal{N}(\ui)}\frac{1}{|\D_\uj|}\int_{\D_\uj} \mDeltaD^{-1} v(y)  \overline{\bfphi}_\uj(x) \d y
\\ \nonumber
& \equiv \frac{8}{3h^2} \sum_{k_1,k_2 = -1}^1 \frac{1}{|\D_{\mathrm{glob}_{\ui}(\uk)}|}\int_{\D_{\mathrm{glob}_{\ui}(\uk)}} \mDeltaD^{-1} v(y)  \overline{\bfphi}_{\mathrm{glob}_{\ui}(\uk)}(x) \d y
\\ \label{orr_delta9}
& \equiv \frac{1}{|\D_{\ui}|}\int_{\D_{\ui}} \sum_{k_1,k_2 = -1}^1 \mDeltaD^{-1} v(y_1+k_1h,y_2+k_2 h)  \frac{8}{3h^2} \overline{\bfphi}_{\mathrm{glob}_{\ui}(\uk)}(x) \d y
\\ \nonumber
& = \frac{1}{|\D_\ui|}\int_{\D_\ui} (-\Delta^9_h)\left(\mDeltaD^{-1} v(y)\right) \d y \,,
\end{align}
where we 
employed the integral transformation $\D_\uj \rightarrow \D_\ui$ for $\uj\neq \ui$ 
(i.e., $y=(y_1,y_2)\in \D_\uj \rightarrow (y_1+ k_1 h,y_2+k_2 h)\in \D_\ui$)
along with the fact that $|\D_\uj|=|\D_\ui|$.

By the consistency property \eqref{consist} we get from \eqref{orr_delta9} for $x\in \D_\ui$ that
\begin{align*}
\overline{R}_h [R_h v](x) &= \frac{1}{|\D_\ui|}\int_{\D_\ui} -\Delta\mDeltaD^{-1} v(y) \d y + \mathcal{O}(h^2) = 
\frac{1}{|\D_\ui|}\int_{\D_\ui} v(y) \d y + \mathcal{O}(h^2)\\ &\equiv \overline{R}_h v(x) + \mathcal{O}(h^2)\,.
\end{align*}
Consequently, on recalling Lemma~\ref{lem_orh} we conclude for $h \rightarrow 0$ that
\begin{equation}\label{rh_p1}
\left\|v- \overline{R}_h[R_h v]\right\|_{\Lp} \leq \|v- \overline{R}_h v \|_{\Lp} + \mathcal{O}(h^2) \rightarrow 0\,.
\end{equation}

Next, we estimate the difference $ \overline{R}_h[R_h v]-R_h v$. 
Due to the local support of the basis functions for $x\in \D_\ui$ we may express
\begin{align}\label{difx}
\left(\overline{R}_h[R_h v]-R_h v\right)(x) & = \frac{8}{3h^2}\sum_{\uj\in\mathcal{N}(i)}\frac{1}{|\D_\uj|}\int_{\D_\uj}\mDeltaD^{-1} v(y)\d y \left(\overline{\bfphi}_\uj(x)-\bfphi_\uj(x)\right)\,.
\end{align}
As in \eqref{orr_delta9} we employ the transformation $\D_\uj \rightarrow \D_\ui$ for $\uj\neq \ui$ and rewrite the above expression as 
\begin{align*}
&\left(\overline{R}_h[R_h v]-R_h v\right)(x) \\
& = \frac{8}{3h^2}\frac{1}{|\D_\ui|}\int_{\D_\ui} \sum_{k_1,k_2=-1}^{1} \Big(\mDeltaD^{-1} v(y_1+k_1h, y_2+k_2 h)   \left(\overline{\bfphi}_{\mathrm{glob}_{\ui}(\uk)}(x)-\bfphi_{\mathrm{glob}_{\ui}(\uk)}(x)\right)\Big)\d y\,.
\end{align*}
Hence, after expressing the basis functions \eqref{def_phi_i_d} explicitly (recall $\ui = (i_1, i_2)$, $x=(x_1,x_2)\in \D_\ui = (\mathbf{x}_{i_1-1},\mathbf{x}_{i_1})\times(\mathbf{x}_{i_2-1},\mathbf{x}_{i_2})$),
for  each $y=(y_1,y_2)$ we restate
\begin{align}\label{expand1} \nonumber
\sum_{k_1,k_2=-1}^{1} & \mDeltaD^{-1} v(y_1+k_1h, y_2+k_2 h)   \left(\overline{\bfphi}_{\mathrm{glob}_{\ui}(\uk)}(x)-\bfphi_{\mathrm{glob}_{\ui}(\uk)}(x)\right)
\\ \nonumber
=
\bigg[& \mDeltaD^{-1}v(y_{(1,1)})\left(-\frac{1}{2}a_{i_1,1}(x_1)-\frac{1}{2}a_{i_2,1}(x_2)+\frac{1}{8}\right)\\ \nonumber
&+\mDeltaD^{-1}v(y_{(0,1)})\left(-\frac{1}{2}a_{i_1,2}(x_1)+a_{i_2,1}(x_2)+\frac{1}{8}\right)\\ \nonumber
&+\mDeltaD^{-1}v(y_{(-1,1)})\left(-\frac{1}{2}a_{i_1,3}(x_1)-\frac{1}{2}a_{i_2,1}(x_2)+\frac{1}{8}\right)\\ \nonumber
&+\mDeltaD^{-1}v(y_{(1,0)})\left(a_{i_1,1}(x_1)-\frac{1}{2}a_{i_2,2}(x_2)+\frac{1}{8}\right)\\ 
&+\mDeltaD^{-1}v(y_{(0,0)})\Big(a_{i_1,2}(x_1)+a_{i_2,2}(x_2)-1\Big)\\ \nonumber
&+\mDeltaD^{-1}v(y_{(-1,0)})\left(a_{i_1,3}(x_1)-\frac{1}{2}a_{i_2,2}(x_2)+\frac{1}{8}\right)\\ \nonumber
&+\mDeltaD^{-1}v(y_{(1,-1)})\left(-\frac{1}{2}a_{i_1,1}(x_1)-\frac{1}{2}a_{i_2,3}(x_2)+\frac{1}{8}\right)\\ \nonumber
&+\mDeltaD^{-1}v(y_{(0,-1)})\left(-\frac{1}{2}a_{i_1,2}(x_1)+a_{i_2,3}(x_2)+\frac{1}{8}\right)\\ \nonumber
&+\mDeltaD^{-1}v(y_{(-1,-1)})\left(-\frac{1}{2}a_{i_1,3}(x_1)-\frac{1}{2}a_{i_2,3}(x_2)+\frac{1}{8}\right)\bigg]\,,
\end{align}
where we employ a shorthand notation $y_{(k_1 ,k_2)} = (y_1+k_1h, y_2+k_2 h)$
and for $n=1,2$ we denote (cf. \eqref{def_phi_i_d})
\begin{align*}
a_{i_n,1}(x_n)& := \frac{3}{2h^2}\psi_{i_n+1}(x_n) =\frac{3}{2h^2}\frac{1}{4}\left(x_n-{\bf x}_{i_n-1}\right)^2 & \text{for } x_n\in({\bf x}_{i_n-1},{\bf x}_{i_n}),\\
a_{i_n,2}(x_n)& := \frac{3}{2h^2}\psi_{i_n}(x_n) =\frac{3}{2h^2}\left[-\frac{1}{2}\left(x_n-{\bf x}_{i_n-1}-\frac{h}{2}\right)^2+\frac{3h^2}{8}\right]  & \text{for } x_n\in({\bf x}_{i_n-1},{\bf x}_{i_n}),\\
a_{i_n,3}(x_n)& := \frac{3}{2h^2}\psi_{i_n-1}(x_n) =\frac{3}{2h^2}\frac{1}{4}\left({\bf x}_{i_n}-x_n\right)^2  &\text{for } x_n\in({\bf x}_{i_n-1},{\bf x}_{i_n}).
\end{align*}
The following property, which follows from \eqref{psi_i} by direct calculation, will be essential in the sequel
\begin{equation}\label{prop_cancel}
a_{i_n,1}(x)+a_{i_n,2}(x)+a_{i_n,3}(x)= \frac{3}{4}\,,  
\end{equation}
for $i_n=2,\ldots,J-1$, and $x\in({\bf x}_{i_n-1},{\bf x}_{i_n})$.

Next, we expand the terms $\tilde{v}(y_{(k_1,k_2)}) := \mDeltaD^{-1}v(y_1+k_1h, y_2+k_2 h)$ in \eqref{expand1} at $y \equiv y_{(0,0)}$ using Taylor series as
\begin{align*}
& \sum_{k_1,k_2=-1}^{1}  \tilde{v}(y_{(k_1,k_2)})   \left(\overline{\bfphi}_{\mathrm{glob}_{\ui}(\uk)}(x)-\bfphi_{\mathrm{glob}_{\ui}(\uk)}(x)\right)
= I+\dots+IV\,,
\end{align*}
where
\begin{align*}
 I = & \bigg[\tilde{v}(y)+\left(\partial_{x_1}\tilde{v}(y)+\partial_{x_2}\tilde{v}(y)\right)h + \left(\frac{1}{2}\partial_{x_1}^2\tilde{v}(y)+\partial_{x_1}\partial_{x_2}\tilde{v}(y)+ \frac{1}{2}\partial_{x_2}^2\tilde{v}(y)\right)h^2\\
&\quad +\mathcal{O}(h^3)\bigg]\left(-\frac{1}{2}a_{i_1,1}(x_1)-\frac{1}{2}a_{i_2,1}(x_2)+\frac{1}{8}\right)\\
&+\bigg[\tilde{v}(y)+\partial_{x_2}\tilde{v}(y)h +  \frac{1}{2}\partial_{x_2}^2\tilde{v}(y)h^2+\mathcal{O}(h^3)\bigg]\left(-\frac{1}{2}a_{i_1,2}(x_1)+a_{i_2,1}(x_2)+\frac{1}{8}\right)\,,
\end{align*}
\begin{align*}
II=& \bigg[\tilde{v}(y)+\left(-\partial_{x_1}\tilde{v}(y)+\partial_{x_2}\tilde{v}(y)\right)h + \left(\frac{1}{2}\partial_{x_1}^2\tilde{v}(y)-\partial_{x_1}\partial_{x_2}\tilde{v}(y)+ \frac{1}{2}\partial_{x_2}^2\tilde{v}(y)\right)h^2\\
&\quad +\mathcal{O}(h^3)\bigg]\left(-\frac{1}{2}a_{i_1,3}(x_1)-\frac{1}{2}a_{i_2,1}(x_2)+\frac{1}{8}\right)\\
&+\bigg[\tilde{v}(y)+\partial_{x_1}\tilde{v}(y)h + \frac{1}{2}\partial_{x_1}^2\tilde{v}(y)h^2+\mathcal{O}(h^3)\bigg]\left(a_{i_1,1}(x_1)-\frac{1}{2}a_{i_2,2}(x_2)+\frac{1}{8}\right)\\
&+\tilde{v}(y)\left(a_{i_1,2}(x_1)+a_{i_2,2}(x_2)-1\right)\,,
\end{align*}
\begin{align*}
III=& \bigg[\tilde{v}(y)-\partial_{x_1}\tilde{v}(y)h + \frac{1}{2}\partial_{x_1}^2\tilde{v}(y)h^2+\mathcal{O}(h^3)\bigg]\left(a_{i_1,3}(x_1)-\frac{1}{2}a_{i_2,2}(x_2)+\frac{1}{8}\right)\\
&+\bigg[\tilde{v}(y)+\left(\partial_{x_1}\tilde{v}(y)-\partial_{x_2}\tilde{v}(y)\right)h + \left(\frac{1}{2}\partial_{x_1}^2\tilde{v}(y)-\partial_{x_1}\partial_{x_2}\tilde{v}(y)+ \frac{1}{2}\partial_{x_2}^2\tilde{v}(y)\right)h^2\\
&\quad +\mathcal{O}(h^3)\bigg]\left(-\frac{1}{2}a_{i_1,1}(x_1)-\frac{1}{2}a_{i_2,3}(x_2)+\frac{1}{8}\right)\,,
\end{align*}
\begin{align*}
IV = &\bigg[\tilde{v}(y)-\partial_{x_2}\tilde{v}(y)h + \frac{1}{2}\partial_{x_2}^2\tilde{v}(y)h^2+\mathcal{O}(h^3)\bigg]\left(-\frac{1}{2}a_{i_1,2}(x_1)+a_{i_2,3}(x_2)+\frac{1}{8}\right)\\
& +\bigg[\tilde{v}(y)-\left(\partial_{x_1}\tilde{v}(y)+\partial_{x_2}\tilde{v}(y)\right)h + \left(\frac{1}{2}\partial_{x_1}^2\tilde{v}(y)+\partial_{x_1}\partial_{x_2}\tilde{v}(y)+ \frac{1}{2}\partial_{x_2}^2\tilde{v}(y)\right)h^2\\
&\quad +\mathcal{O}(h^3)\bigg]\left(-\frac{1}{2}a_{i_1,3}(x_1)-\frac{1}{2}a_{i_2,3}(x_2)+\frac{1}{8}\right)\, .
\end{align*}

We rearrange the above terms $I-IV$, use the identity \eqref{prop_cancel} and obtain
\begin{align}\label{taylor1}
\nonumber
 \sum_{k_1,k_2=-1}^{1}  & \tilde{v}(y_{(k_1,k_2)})   \left(\overline{\bfphi}_{\mathrm{glob}_{\ui}(\uk)}(x)-\bfphi_{\mathrm{glob}_{\ui}(\uk)}(x)\right)
\\
&=0 \cdot \Big[\tilde{v}(y) + \left(\partial_{x_1}\tilde{v}(y)  + \partial_{x_2}\tilde{v}(y)\right)h + \partial_{x_1}\partial_{x_2}\tilde{v}(y) h^2 \Big]
\\ \nonumber
&\quad + \frac{1}{2}\partial_{x_1}^2\tilde{v}(y)\bigg[-a_{i_2,1}(x_2) -a_{i_2,2}(x_2) -a_{i_2,3}(x_2)+\frac{3}{4}\bigg]h^2
\\ \nonumber
&\quad +\frac{1}{2}\partial_{x_2}^2\tilde{v}(y)\bigg[-a_{i_1,1}(x_1) -a_{i_1,2}(x_1) -a_{i_1,3}(x_1)+\frac{3}{4}\bigg]h^2
+\mathcal{O}(h^3)
\\ \nonumber
&=\mathcal{O}(h^3)\,.  
\end{align}
Hence, we substitute \eqref{taylor1} into \eqref{difx} to conclude that
\begin{align}\label{rh_p2} 
\DD \left\|\overline{R}_h[R_h v]-R_h v\right\|_{\Lp}
 & =\Big\|\frac{8}{3h^2}\sum_{\uj\in\{1,\ldots,J\}^d}\frac{1}{|\D_\uj|}\int_{\D_\uj}\mDeltaD^{-1} v(y) \left(\overline{\bfphi}_\uj(\cdot)-\bfphi_\uj(\cdot)\right)\d y \Big\|_{\Lp}
\\
\nonumber
& = Ch\,.
\end{align}

Finally, by the triangle inequality we estimate
\begin{align}\label{Proof_tildeRh_Est_1}
\|v- R_h v\|_{\Lp}&\le \|v- \overline{R}_h[R_h v]\|_{\Lp}+\|\overline{R}_h[R_h v]- R_h v\|_{\Lp}\,,
\end{align}
and the statement follows by \eqref{rh_p1} and \eqref{rh_p2}.
\end{proof}


\begin{comment}
\begin{proof}
On noting the property $\overline{R}_hv = \overline{R}_h(R_hv)$ and $R_h v\in \Vh\subset \Lp$ we obtain for all $v\in \Lp$ that\footnote{second term is not converging?}
\begin{align*}
\|v- R_h v\|_{\Lp} \le \|v- \overline{R}_h(R_h v)\|_{\Lp} + \|\overline{R}_h(R_h v)  - R_h v\|_{\Lp}\,.
\end{align*}
The statement then follows from \eqref{orh_conv}.
\end{proof}
\end{comment}

The above lemma allows us to deduce the density of $\{\mathbb{V}_h\}_{h>0}$ in $\mathbb{L}^p$.
\begin{cor}[Approximation property of $\Vh$]\label{cor_rh} 
For every $v\in \Lp$, $p\geq1$ it holds that
$$
\inf_{v_h\in \Vh}\|v-v_h\|_{\Lp}\rightarrow 0 \quad \text{ for } h\rightarrow 0.
$$
\end{cor}
\begin{proof}
Consider $v_\varepsilon \in \mathbb{V}_n$ and note that $\lim_{h\rightarrow 0} \|v_\varepsilon-R_h v_\varepsilon \|_{\Lp} = 0$ by Lemma~\ref{lem_rh_conv}.
Since $R_hv_\varepsilon \in \Vh$ we get
$$
\inf_{v_h\in \Vh}\|v-v_h\|_{\Lp}\le \|v-R_h v_\varepsilon \|_{\Lp} \leq \|v- v_\varepsilon \|_{\Lp} + \|v_\varepsilon-R_h v_\varepsilon \|_{\Lp}.
$$
The statement then follows by the density of $\cup_{n\in\mathbb{N}}\mathbb{V}_n$ in $\mathbb{L}^p$.
\end{proof}

The restriction operator \eqref{def_rh} is not implementable since it requires the evaluation of the function $\mDeltaD^{-1} v$,
which is not available in general.
For practical purposes (e.g., to compute the discrete approximation of the initial condition) it is convenient to consider the
discrete $\Hmone$-projection $P_h:\Hmone\rightarrow \Vh$ which is defined for $v\in \Hmone$ as follows
\begin{equation}\label{Proj}
(P_hv,w_h)_{\Hmone}=(v,w_h)_{\Hmone}\quad\forall w_h\in \Vh.
\end{equation}

\begin{rem}\label{rem_conv_ph}
The  $\Hmone$-stability of the orthogonal projection, i.e., $\|P_hv\|_{\Hmone} \leq C \|v\|_{\Hmone}$ follows 
on taking $w_h = P_h v$ in \eqref{Proj} and using the Cauchy-Schwarz and Young's inequalities.
Furthermore, we note that \eqref{Proj} is equivalent
to $P_hv = \arg\min_{w_h\in\mathbb{V}_h} \|v-w_h\|^2_{\Hmone}$ which in particular implies that $P_h(R_hv)=R_hv$
for $v\in\Lp$.

Consequently, the $\Hmone$-stability of $P_h$ 
and the continuous embedding $\Lp[p] \hookrightarrow \Hmone$, for $p\geq 2$ yield for all $v\in \Lp[p]$, $v_\varepsilon \in \mathbb{V}_n$
\begin{align*}
\|v- P_h v\|_{\Hmone} & \leq \|v- v_\varepsilon\|_{\Hmone} + \|v_\varepsilon - P_h v_\varepsilon\|_{\Hmone} + \|P_h (v_\varepsilon -  v)\|_{\Hmone} 
\\
& \leq C\|v- v_\varepsilon\|_{\Hmone} + \|v_\varepsilon- R_h v_\varepsilon\|_{\Hmone} + \|P_h (v_\varepsilon- R_h v_\varepsilon)\|_{\Hmone}
\\
& \leq C\|v- v_\varepsilon\|_{\Hmone} + C\|v_\varepsilon- R_h v_\varepsilon\|_{\Lp}\,.
\end{align*}
Hence, by Lemma~\ref{lem_rh_conv}, the density of $\Lp[p]$, $p\geq 2$ in $\Hmone$ and the density of $\cup_{n\in\mathbb{N}}\mathbb{V}_n$ in $\Lp[p]$ 
we conclude the approximation property of the $\Hmone$-orthogonal projection:
$$
\lim_{h\rightarrow 0} \|v- P_h v\|_{\Hmone} = 0\qquad \forall v \in \Hmone\,.
$$
\end{rem}

\section{Numerical Experiments}\label{sec_num_exp}

\subsection{Convergence of the projection in $d=2$}\label{sec_ph}

We study the experimental $\mathbb{L}^p$-convergence
of the $\mathbb{H}^{-1}$-projection operator \eqref{Proj}
as well as of an implementable  counterpart $\widetilde{R}_h:\Lp\rightarrow \Vh$ 
of the restriction operator \eqref{def_rh} defined as
\begin{equation*}
\widetilde{R}_h v(x):=\sum_{\ui\in\{1,\ldots,J\}^d} \tilde{v}_\ui \bfphi_\ui(x) \,,
\end{equation*}
where $\tilde{v}_\ui = \DD \frac{8}{3h^2} [(-\Delta^9_h)^{-1} \overline{R}_h v]_\ui$.
I.e., the coefficients are the solutions of finite difference scheme
$$
-\Delta_h^9 \left(\frac{3h^2}{8} \tilde{v}_{\ui}\right) = \overline{R}_h v|_{\D_\ui}\,,
$$
for $\ui \in \{1,\dots,J\}^2$; we note that it holds by construction that $\overline{R}_h \widetilde{R}_h v = \overline{R}_h v$ and $\widetilde{R}_h \overline{R}_h v = \widetilde{R}_h v$.

In Figure~\ref{fig_ph_conv} we display the convergence plot of the $\mathbb{H}^{-1}$-projection
of the Barenblatt solution $P_h u_B(t,\cdot)$ at $t=0.1$ (see \eqref{barenblatt} below) along with the convergence plot of $P_h \chi_{(-0.5,0.5)^2}$ of the (non-smooth) 
indicator function of the $(-0.5,0.5)^2$-square; in both cases $\D=(-1.5,1.5)^2$.
The convergence plot implies convergence of the projection in $\mathbb{L}^p$ of order $h$ for the smooth Barenblatt function and
of order of $h^{2/3}$ in the non-smooth case.

In addition we display in Figure~\ref{fig_ph_conv}
the convergence plot of the restriction operator $\widetilde{R}_h$ for the indicator function $\chi_{(-0.5,0.5)^2}$
which is also of order $h^{2/3}$.

\begin{figure}[htp!]
\centering
\includegraphics[width=0.48\textwidth]{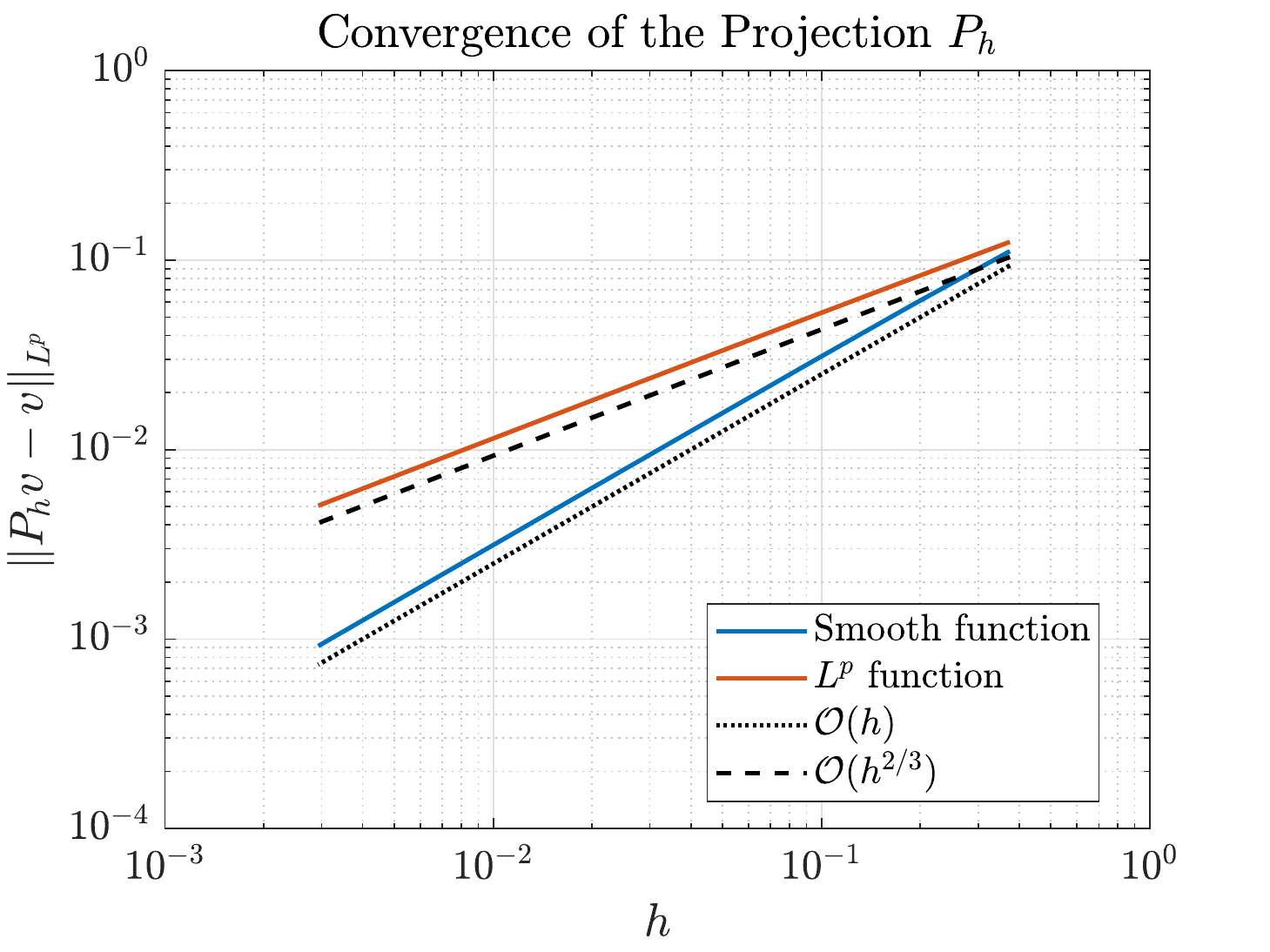}
\includegraphics[width=0.48\textwidth]{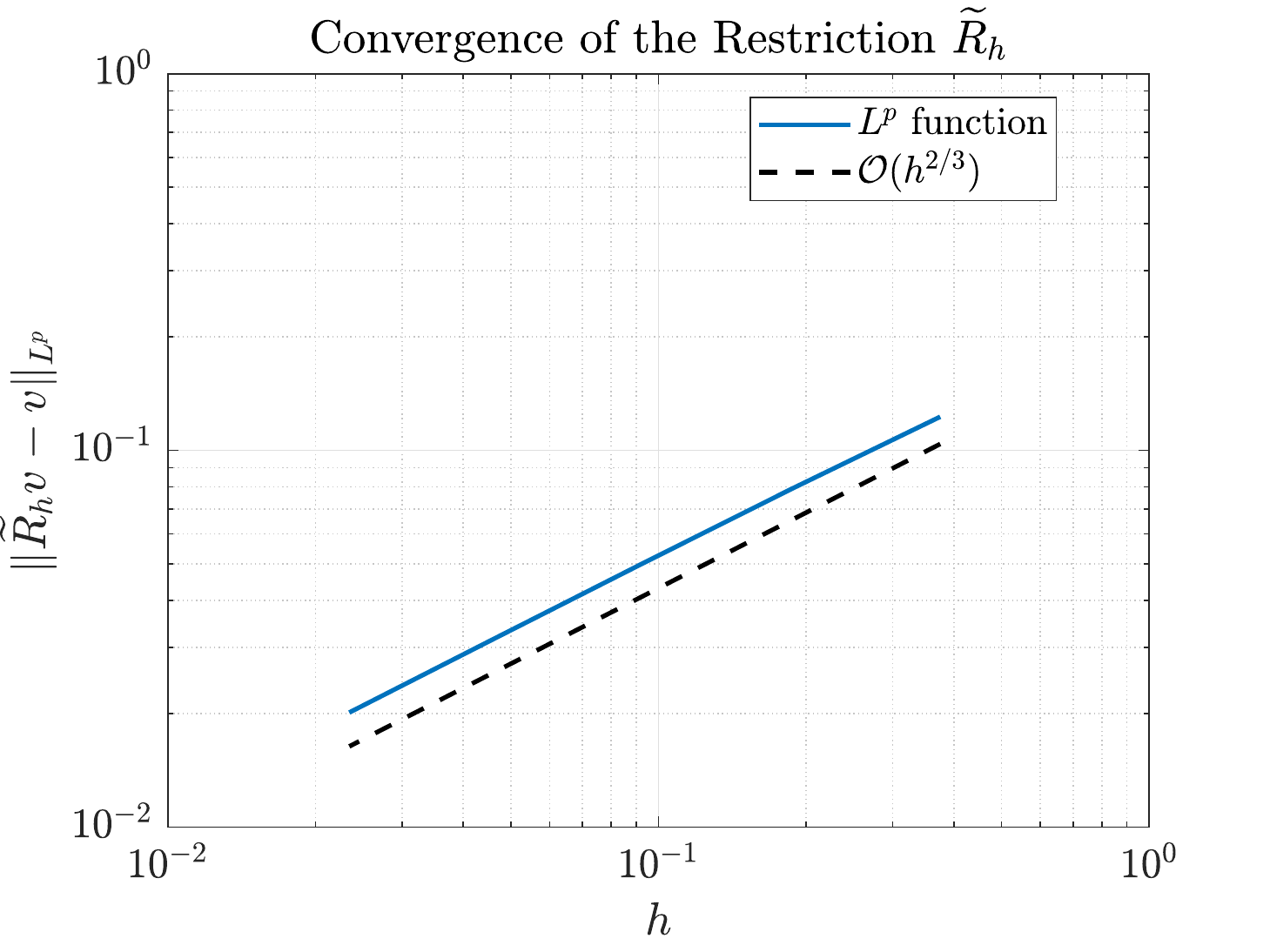}
\caption{Convergence of the $\mathbb{H}^{-1}$-projection (left) and of the restriction operator $\widetilde{R}_h$ (right).}
\label{fig_ph_conv}
\end{figure}

\subsection{Barenblatt solution for the deterministic PME}\label{subsec_barenblatt}

We consider the equation \eqref{spme0} with $\alpha(u)=|u|^{p-2}u$, $f\equiv 0$, $g\equiv 0$, $\sigma\equiv 0$
which corresponds to the deterministic porous medium equation
$$ 
\partial_t u = \Delta (|u|^{p-2}u).
$$ 
The exact solution of the porous media equation with initial condition $u_0=\delta_0$ (i.e., the $\delta$-distribution centered at $0$)
 the so-called Barenblatt solution 
\begin{equation}\label{barenblatt}
u_B(t,x)=t^{-a}\max\left\{0,C-k|x|^2t^{-2b}\right\}^{1/(p-2)},
\end{equation}
where $a\,,\,b,\,k,\,C$ are suitable constants that depend on $p$, $d$, c.f.~\cite[Ch. 17.5]{Vazquez}.

In the experiments below we choose $\D=(-1.5,1.5)^d$, $d=1,2$ and $T=0.1$, $p=3$.
We consider a regularized initial condition $u_0 = \delta_0 \approx \tilde{u}_{h,0} \in \overline{\mathbb{V}}_h$ with
$$
\tilde{u}_{h,0}(x)=\frac{1}{(2h)^d}\begin{cases}1 &\text{if } x\in\D_\ui,\ j\in \left\{\frac{J}{2},\frac{J}{2}+1\right\}^d \\0&\text{else}
\end{cases}
$$
and set $u_{h,0}=P_h(\tilde{u}_0)\in \Vh$.

We examine the convergence of the numerical approximation with respect to $\tau$, $h$ in the $L^p$-norm, i.e.,
we compute the error $\| u_B-\utau\|_{L^p([\underline{t},T]\times \D)}$ with
time-interval $[\underline{t},T]=[0.01,0.1]$ where we choose $\underline{t}>0$ to reduce the effect of the approximation of the initial condition.

In Table~\ref{tab_1d} we display the $L^p$-error for $\tau=1/N$, $h=2L/J$ in $d=1$.
The corresponding convergence plots in Figure~\ref{fig_conv1d} indicate that the convergence
order of the numerical approximation with respect to $\tau$ is slightly less than one  and around $\frac{3}{2}$ with respect to $h$.
\begin{table}[h]
\begin{tabular}{l c c c c c c}
$N$ $\setminus$ $J$ & $8$ & $16$ & $32$ & $64$ & $128$ & $256$\\
$8$ & 0.083032 & 0.02221 & 0.020881 & 0.024805 & 0.025878 & 0.025931\\
$16$ & 0.07524 & 0.016254 & 0.015419 & 0.015481 & 0.015893 & 0.016162\\
$32$ & 0.075711 & 0.017544 & 0.0070398 & 0.0088919 & 0.0094508 & 0.0094772\\
$64$ & 0.077172 & 0.021771 & 0.0054151 & 0.0051293 & 0.0056912 & 0.0059705\\
$128$ & 0.077702 & 0.022649 & 0.0060452 & 0.0028429 & 0.0033319 & 0.0035322\\
$256$ & 0.077591 & 0.022801 & 0.006569 & 0.002342 & 0.0017761 & 0.0019579\\
$512$ & 0.077532 & 0.022934 & 0.0069462 & 0.002467 & 0.0011016 & 0.0010656\\
$1024$ & 0.077593 & 0.023061 & 0.007187 & 0.0025917 & 0.00099377 & 0.00060724
\end{tabular}
\caption{$L^p((0.01,0.1)\times\D)$-error of the solution, $d=1$.}
\label{tab_1d}
\end{table}

\begin{figure}[htp!]
\centering
\includegraphics[width=0.48\textwidth]{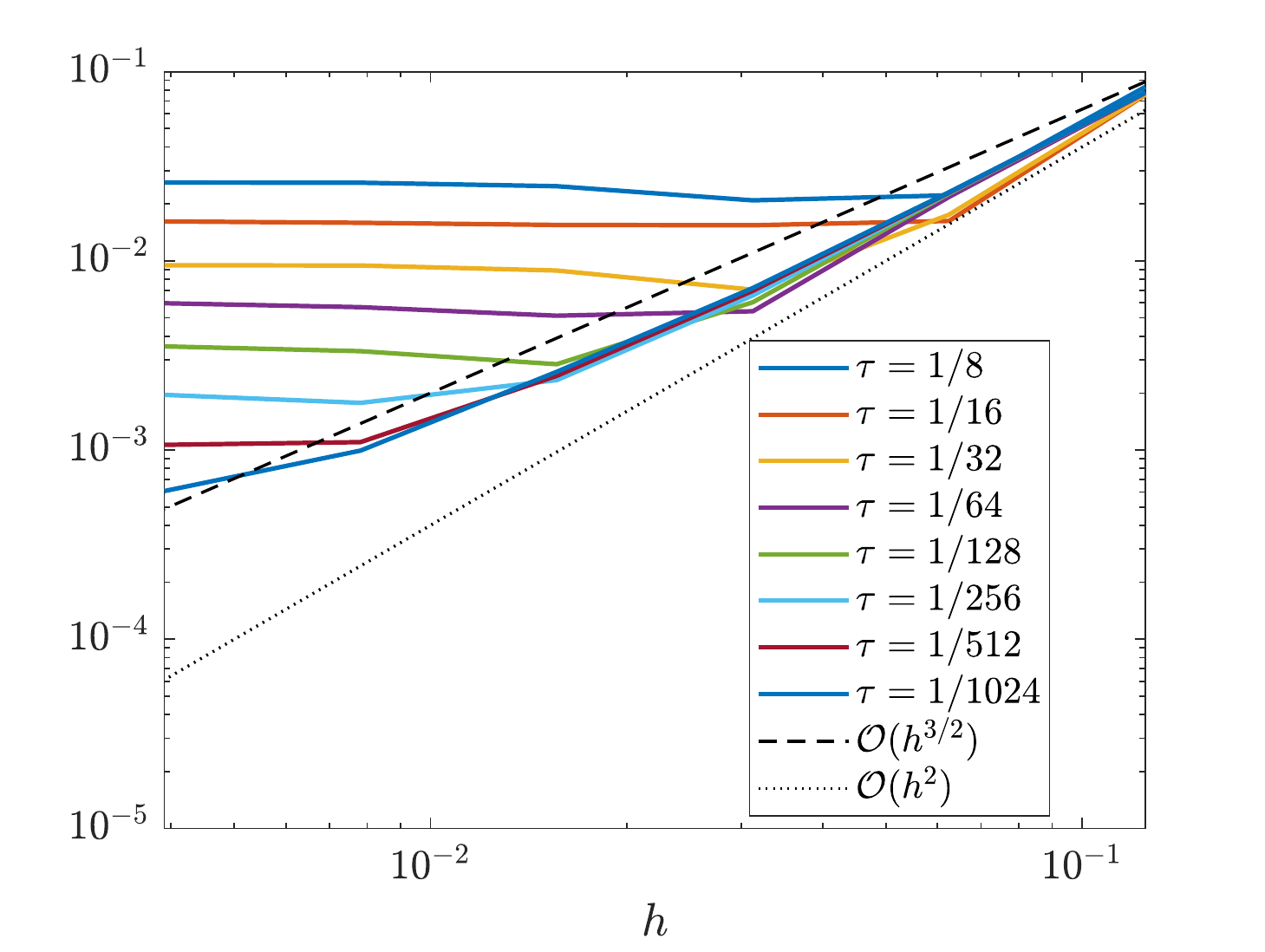}
\hfill
\includegraphics[width=0.48\textwidth]{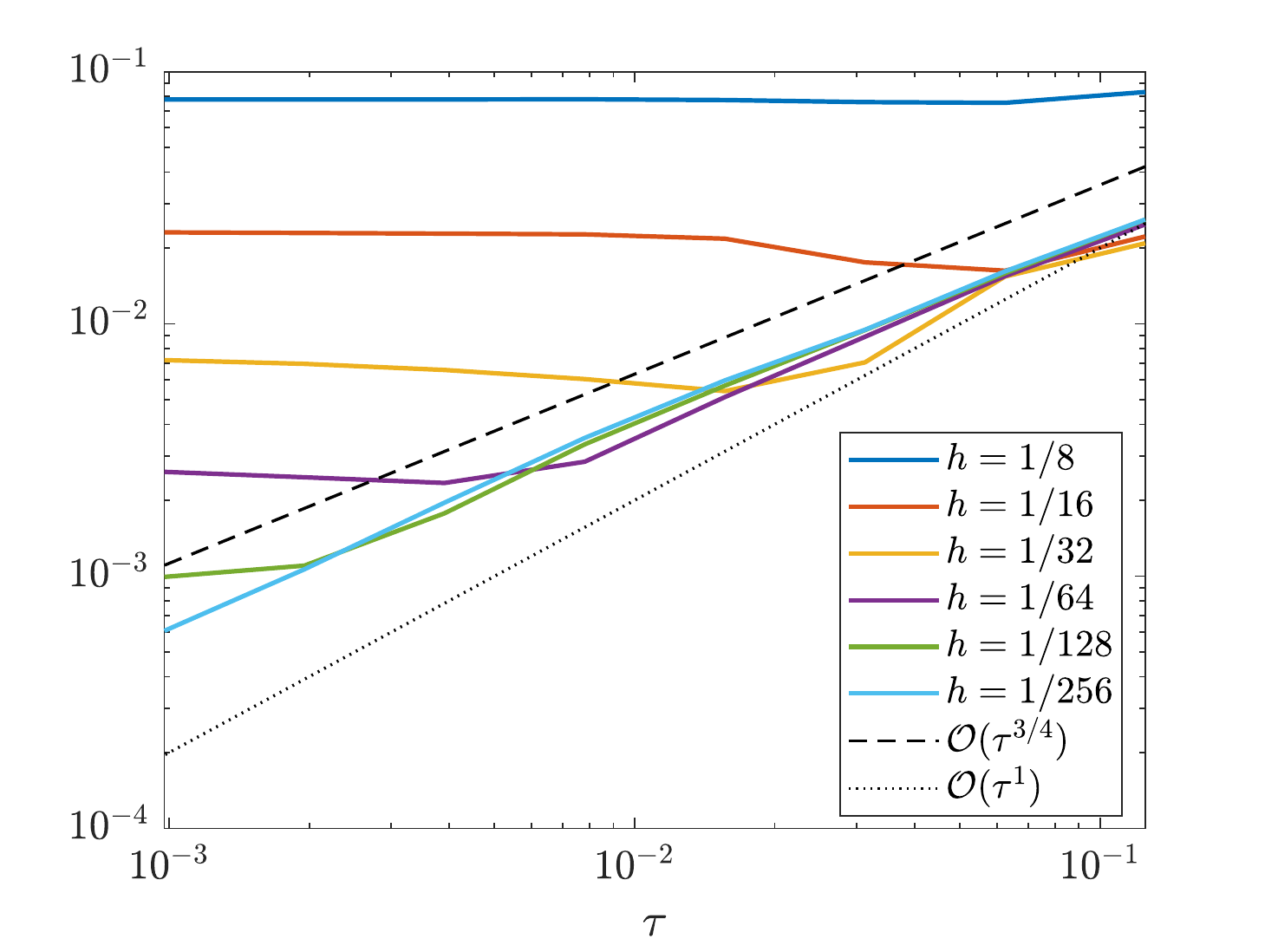}
\caption[Convergence for the deterministic equation in $1d$]{Convergence of the $L^p((0.01,0.1)\times \D)$-error of the solution for the deterministic equation in $1d$. Left: convergence of the spacial discretization for different time step-sizes, right: convergence of the time discretization for different spacial step-sizes}
\label{fig_conv1d}
\end{figure}

To highlight the finite speed of propagation property on the discrete level
we display the evolution of the support of the numerical approximation 
in Figure~\ref{fig_support_1d}.
\begin{figure}[htp!]
\begin{subfigure}[c]{1\textwidth}
\centering
\includegraphics[scale=0.35]{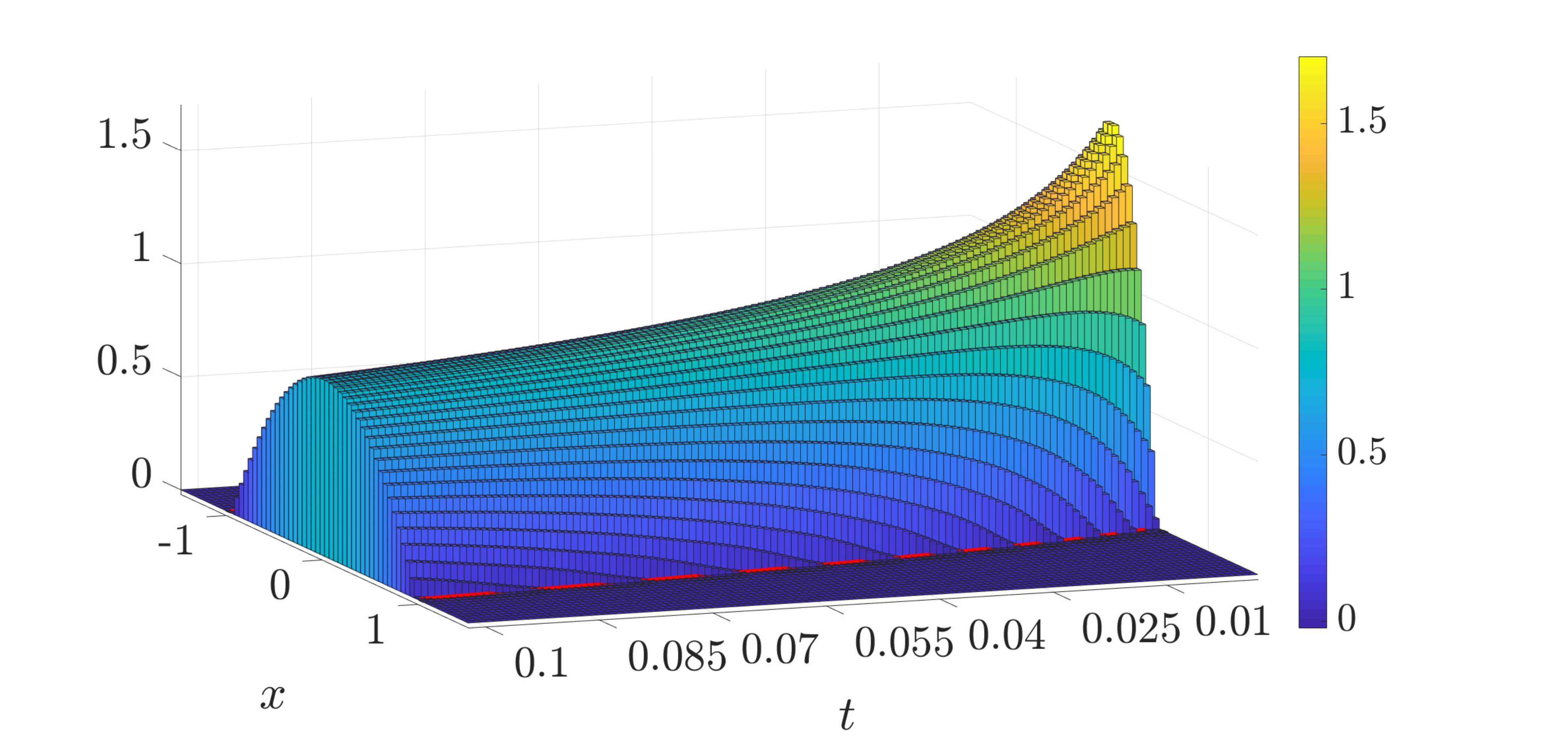}
\end{subfigure}
\begin{subfigure}[c]{1\textwidth}
\centering
\includegraphics[scale=0.35]{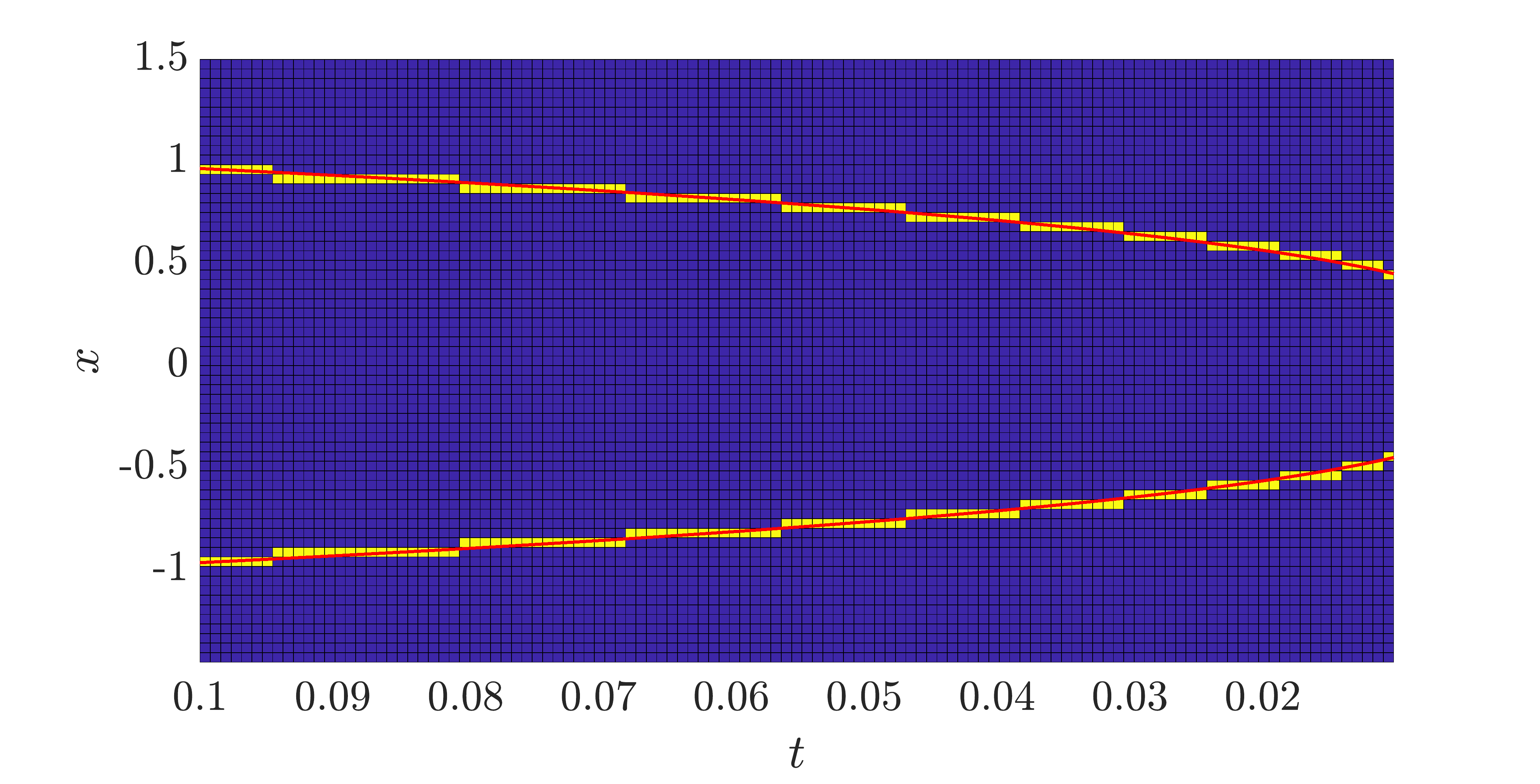}
\end{subfigure}
\caption{(top) Time evolution of the numerical solution for $J = 64$, $N = 128$ in $d=1$; (bottom)
the corresponding support of the numerical solution, in yellow, and the support of the analytical solution, in red.}
\label{fig_support_1d}
\end{figure}

Next we examine the convergence behaviour in $d=2$, we note that in this case $u_0\notin \mathbb{H}^{-1}$.
In Table~\ref{tab_2d} we display the $L^p$-error computed for $\tau=1/N$, $h=2L/J$.

The corresponding convergence plots in Figure~\ref{fig_conv2d} indicate that the convergence
order of the numerical approximation with respect to $\tau$ and $h$ are both close to one.
As expected, (due to the lower regularity of the initial condition in $d=2$) the observed convergence order of the spatial discretization
is slightly worse than the corresponding convergence order for $d=1$.
We display the time evolution of the numerical solution in Figure~\ref{fig_sol2d}
and a detail of the numerical solution at $T=0.1$, $d=2$ is displayed in Figure~\ref{fig_barsol2d}.

\begin{table}[h]
\begin{tabular}{l c c c c c c}
$N$ $\setminus$ $J$ & $8$ & $16$ & $32$ & $64$ & $128$ & $256$\\
$8$ & 0.092154 & 0.050998 & 0.045512 & 0.04534 & 0.04531 & 0.045288\\
$16$ & 0.094038 & 0.047956 & 0.032504 & 0.028832 & 0.027894 & 0.027644\\
$32$ & 0.095218 & 0.048104 & 0.026604 & 0.019108 & 0.016976 & 0.016402\\
$64$ & 0.098787 & 0.050731 & 0.026017 & 0.015466 & 0.011883 & 0.010852\\
$128$ & 0.10062 & 0.052414 & 0.026237 & 0.013883 & 0.0087862 & 0.0070965\\
$256$ & 0.10077 & 0.052773 & 0.026229 & 0.013247 & 0.0072316 & 0.0048007\\
$512$ & 0.10086 & 0.052981 & 0.026295 & 0.013092 & 0.0067107 & 0.0037732\\
$1024$ & 0.10109 & 0.05325 & 0.026433 & 0.013108 & 0.0065808 & 0.0034176
\end{tabular}
\caption{$L^p((0.01,0.1)\times\D)$-error of the solution, $d=2$.}
\label{tab_2d}
\end{table}

\begin{figure}[htp!]
\centering
\includegraphics[width=0.48\textwidth]{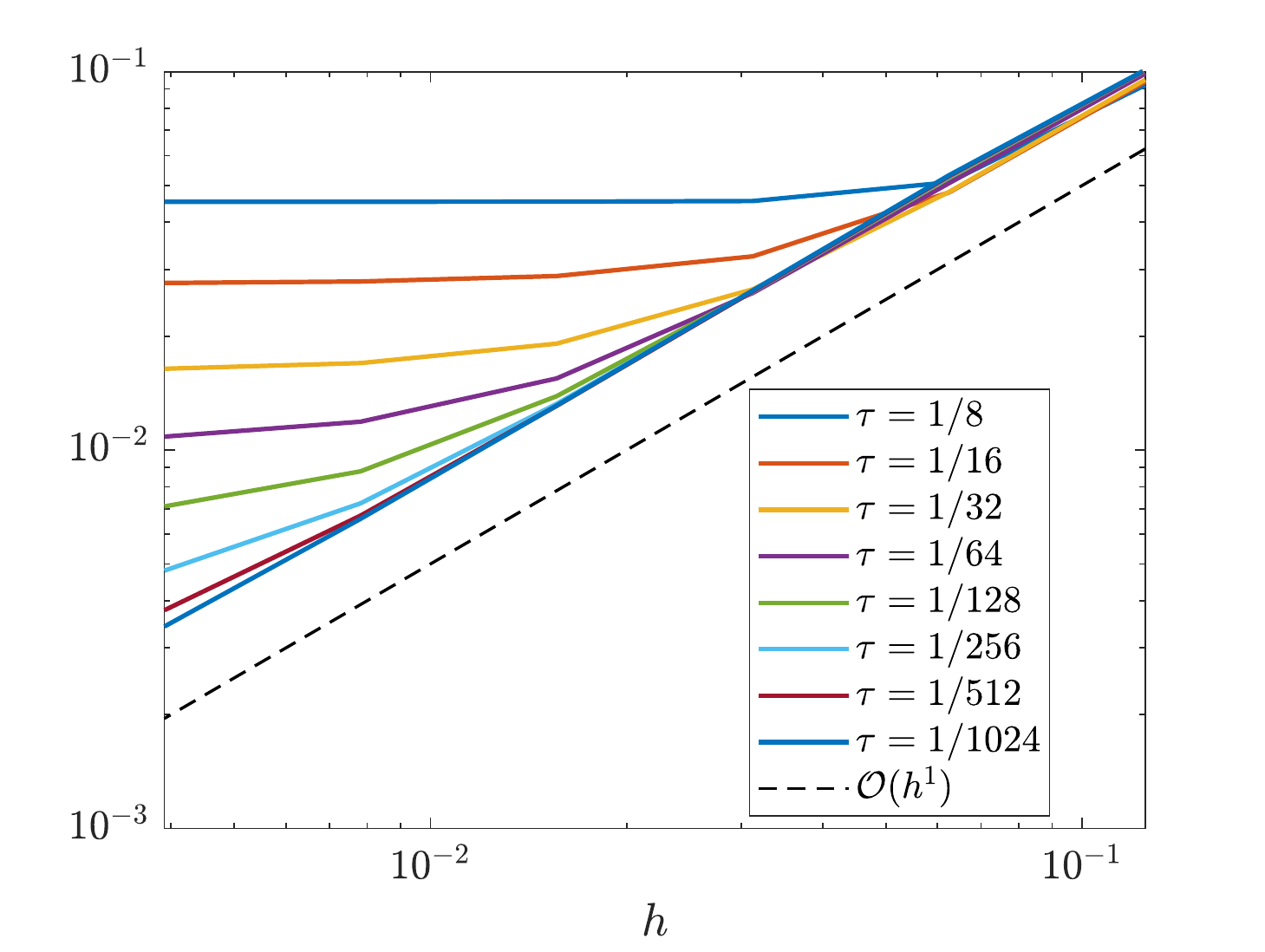}
\hfill
\includegraphics[width=0.48\textwidth]{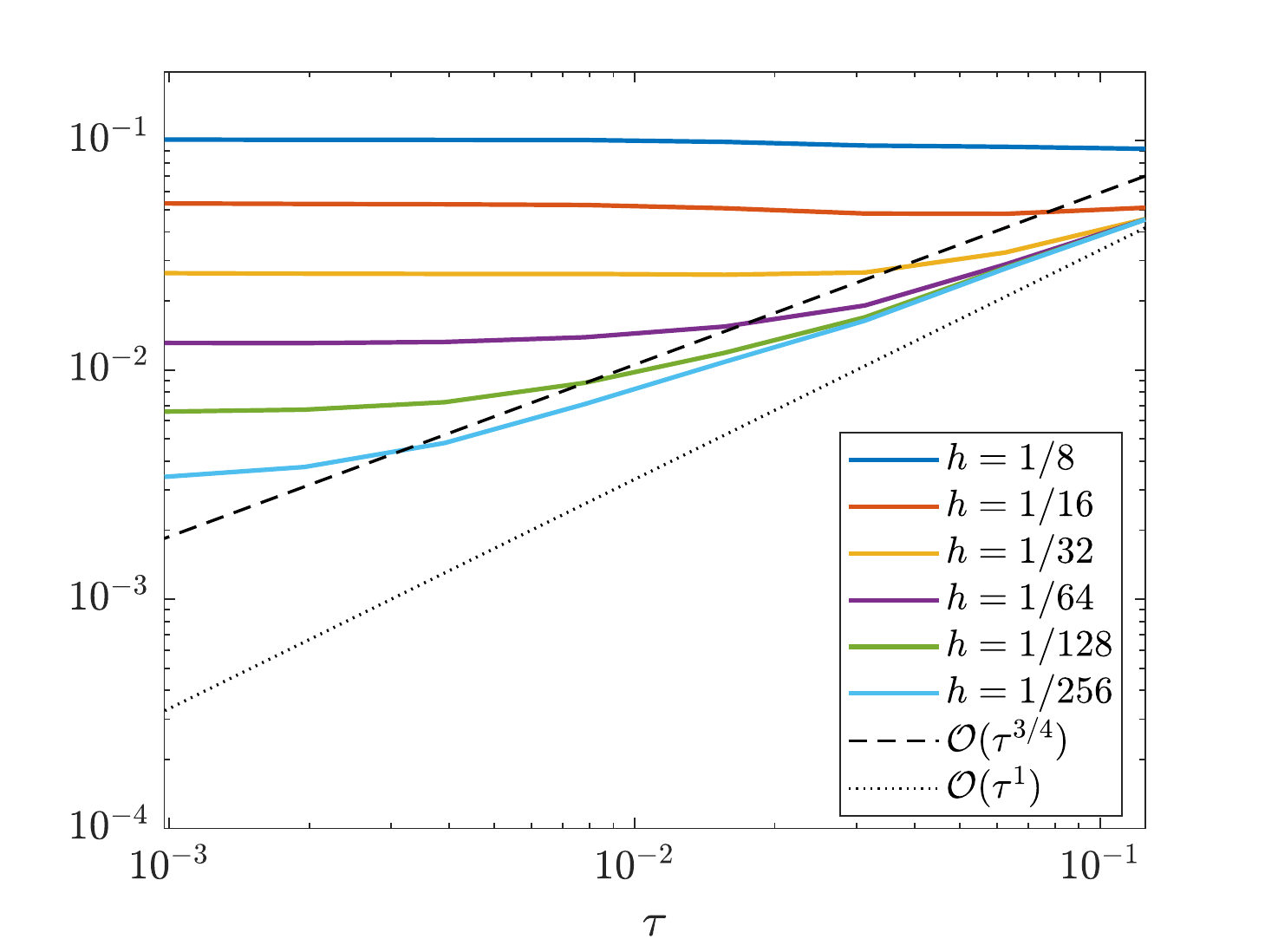}
\caption[Convergence for the deterministic equation in $2d$]{Convergence of the $L^p((0.01,0.1)\times \D)$-error of the solution for the deterministic equation in $2d$. 
Left: convergence of the discretization for different time step-sizes, right: convergence of the discretization for different mesh-sizes.}
\label{fig_conv2d}
\end{figure}

\begin{center}
\begin{figure}[htp!]
\centering
\begin{minipage}{0.49\textwidth}
\includegraphics[scale=0.5]{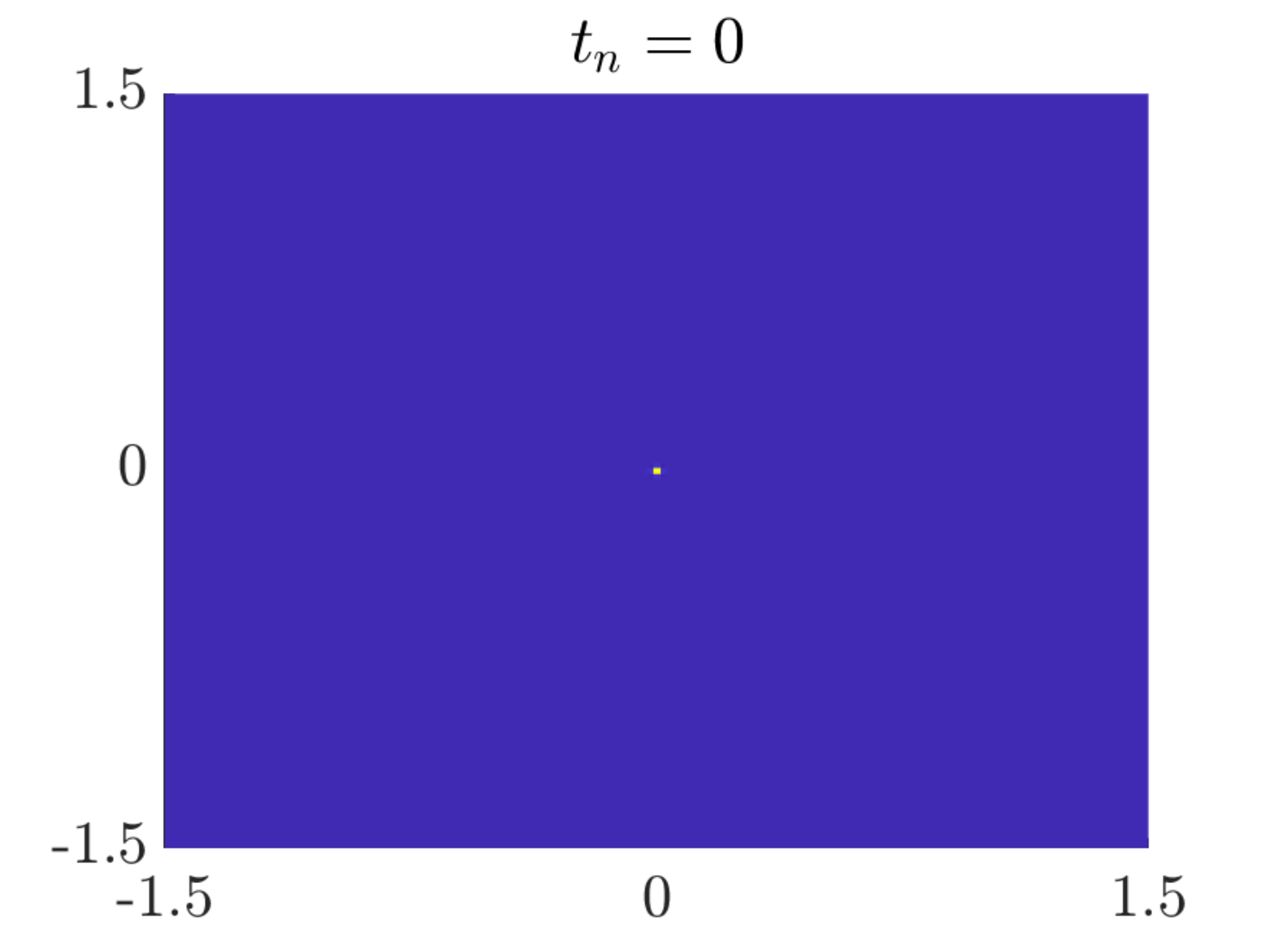}
\end{minipage}\\ \vspace{0.7cm}
\begin{minipage}{0.49\textwidth}
\includegraphics[scale=0.5]{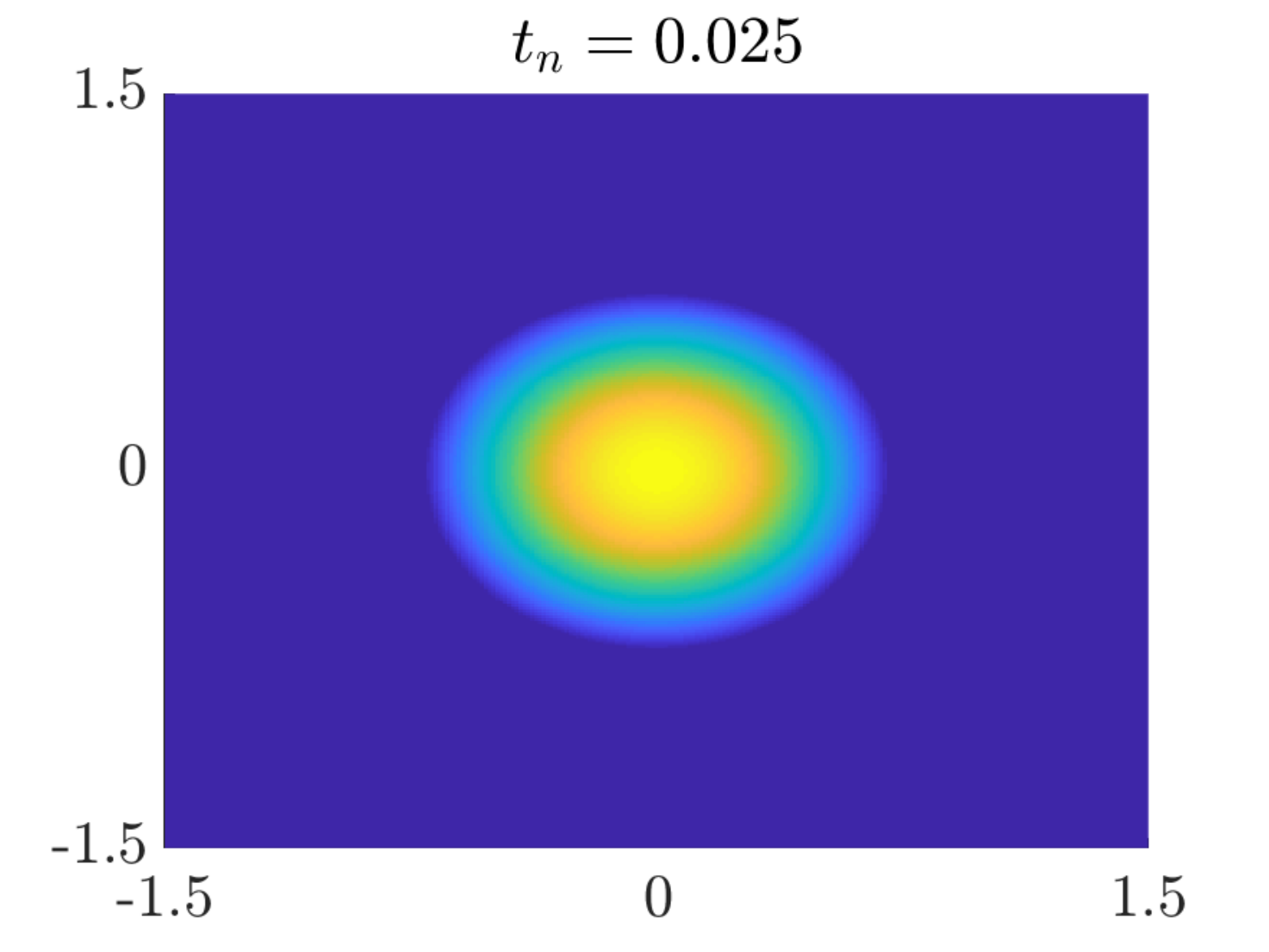}
\end{minipage}\begin{minipage}{0.49\textwidth}
\includegraphics[scale=0.5]{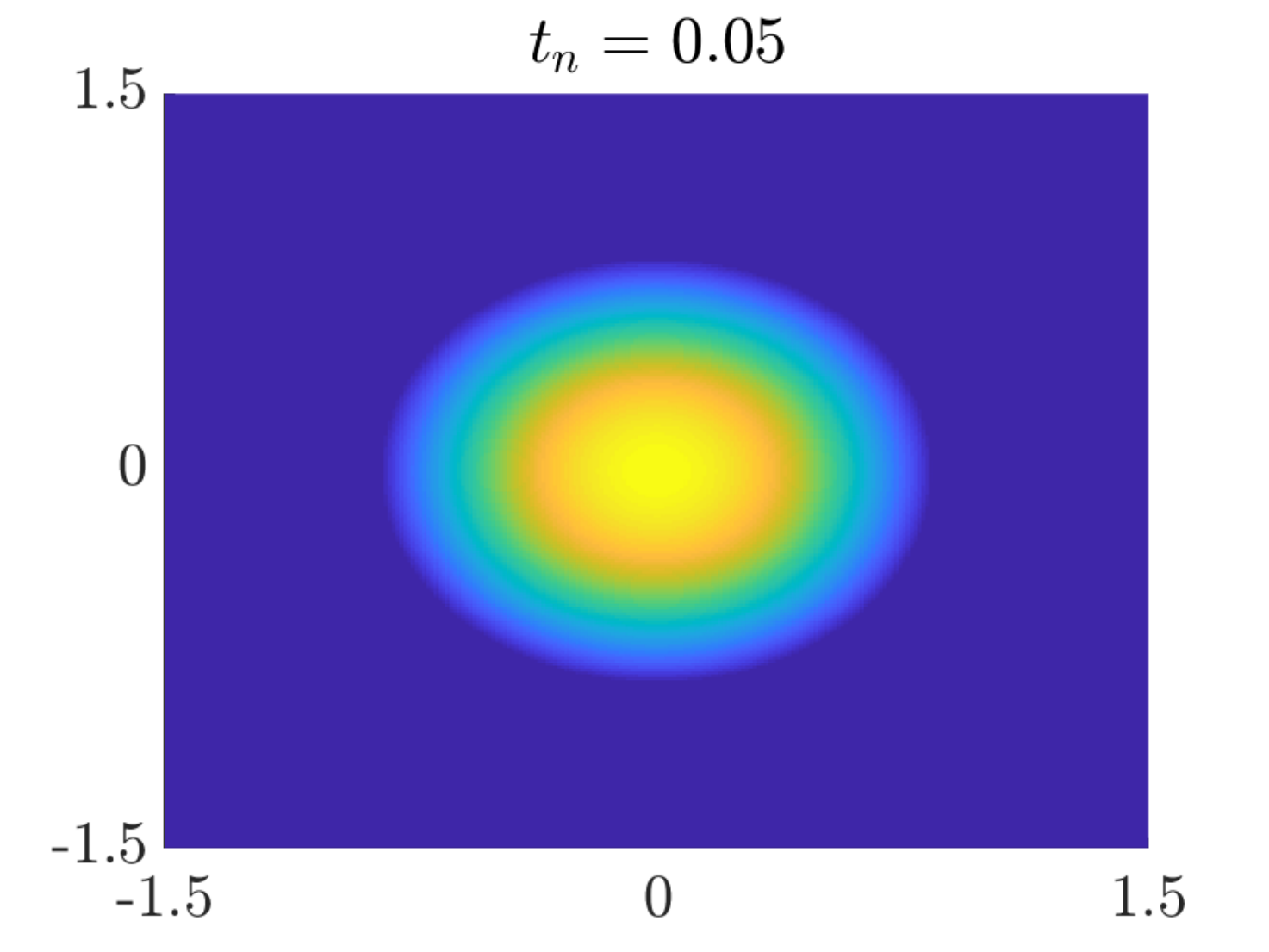}
\end{minipage}\\ \vspace{0.7cm}
\begin{minipage}{0.49\textwidth}
\includegraphics[scale=0.5]{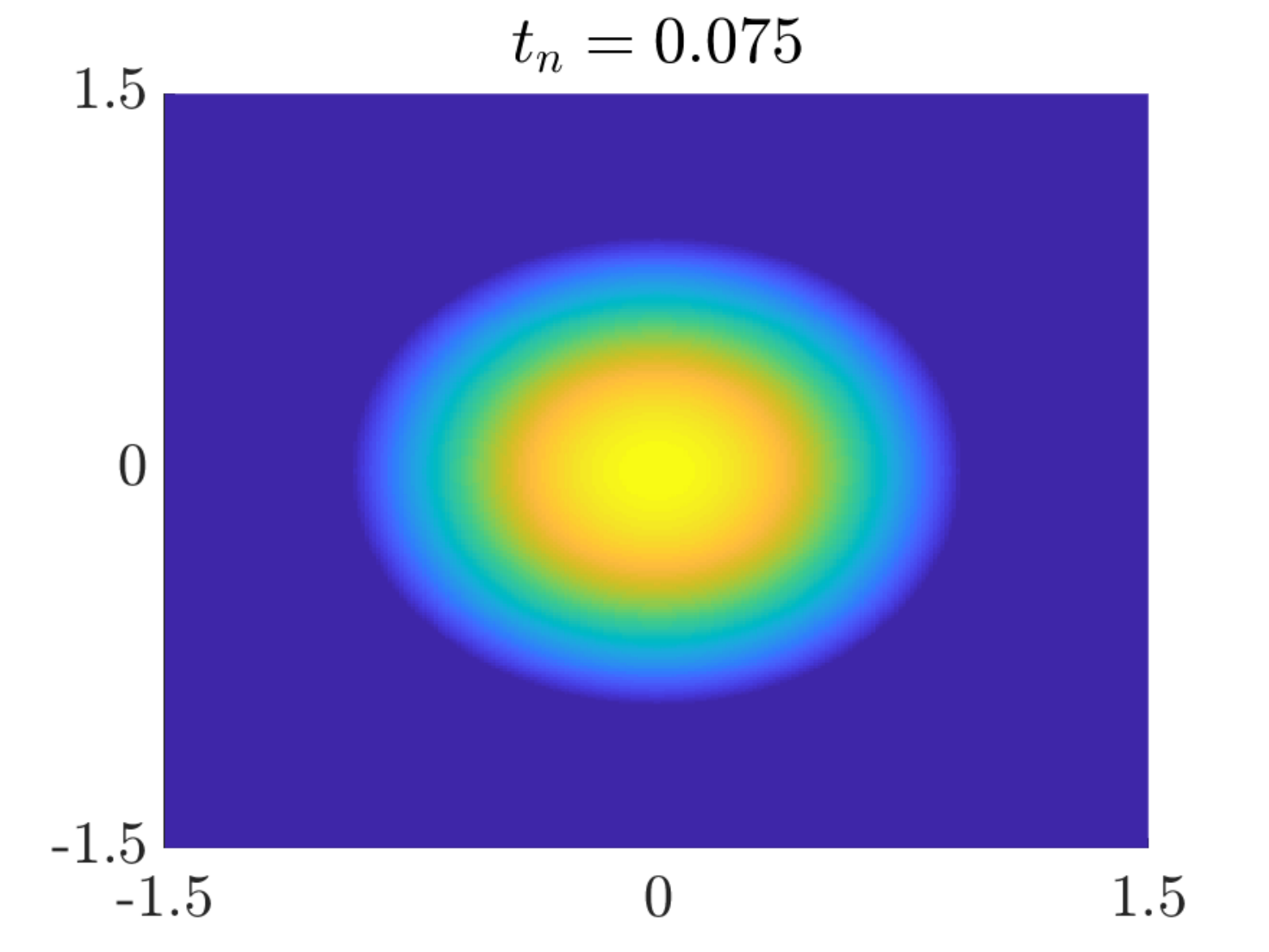}
\end{minipage}\begin{minipage}{0.49\textwidth}
\includegraphics[scale=0.5]{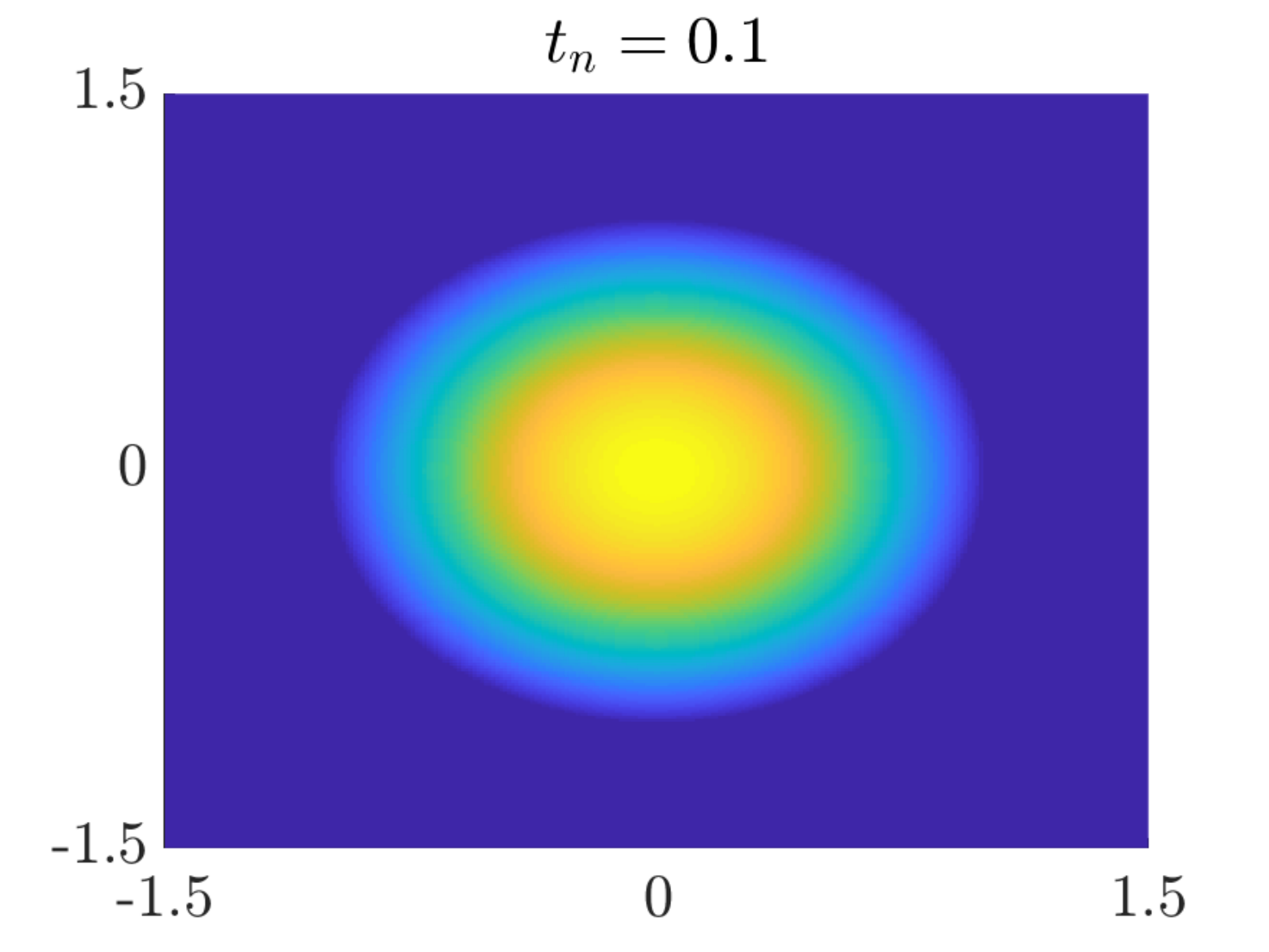}
\end{minipage}\\
\caption[Time evolution of the numerical appr. of the Barenblatt solution]{Snapshots of the numerical solution computed with $J=N=256$ at time $t = 0, 0.025, 0.05, 0.075, 0.1$.}
\label{fig_sol2d}
\end{figure}
\end{center}

\begin{figure}[htp!]
\centering
\includegraphics[trim=70 00 0 0, clip, scale=0.25]{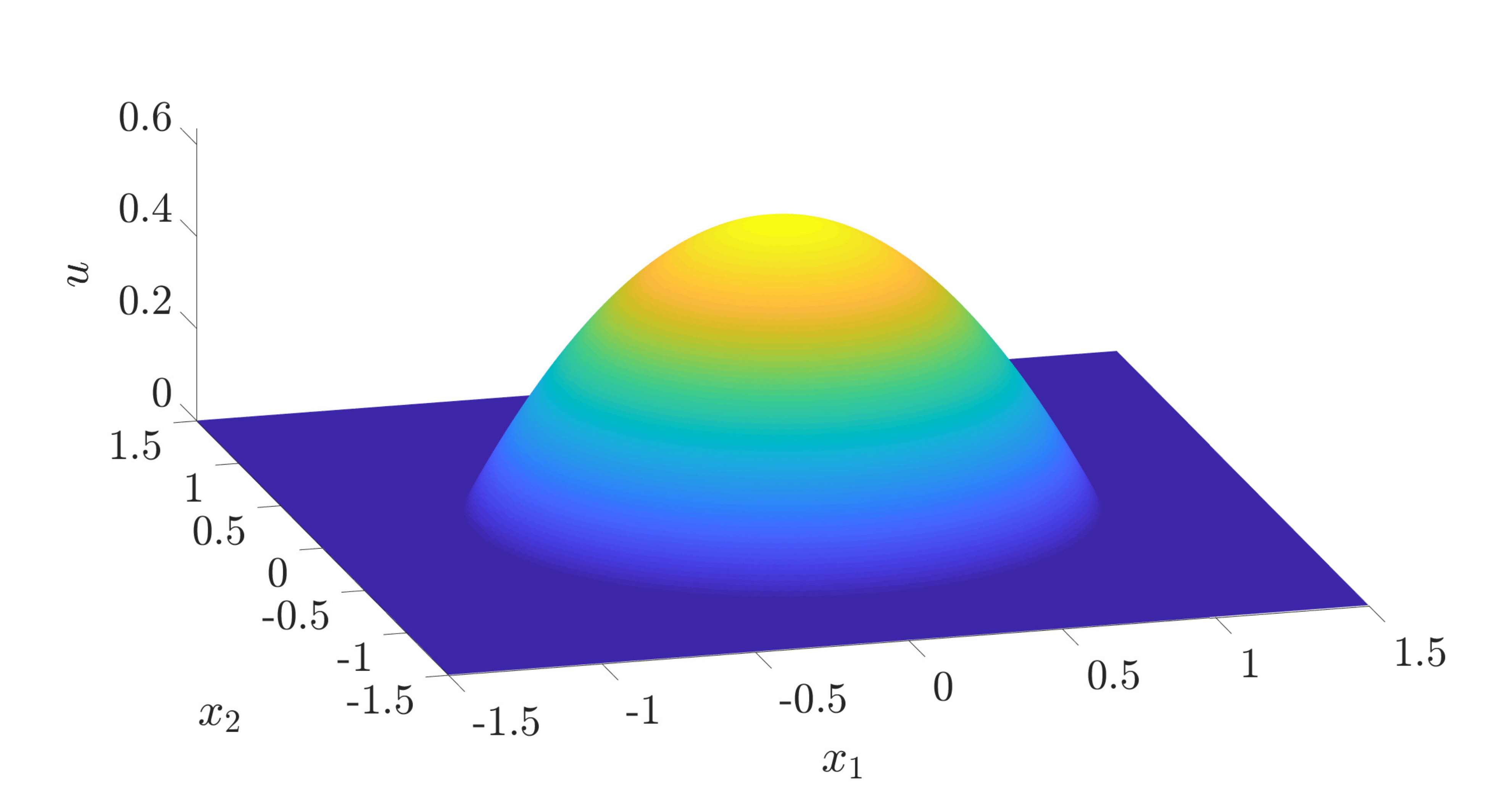}
\caption{Numerical approximation of the Barenblatt solution at time $t = T$ for $J=N=256$, $d=2$}
\label{fig_barsol2d}
\end{figure}

\subsection{Barenblatt solution for the stochastic PME}
Now we consider the stochastic equation with $\phi(u)=u$, set again $f,g=0$. Then from \cite[p. 87,88]{RoecknerSPME} and the references cited therein, we get, that for $d=1$, $\D=\mathbb{R}$ and $p=3$ the solution at time $t\in(0,T]$ is given by 
$$u_B\left(\int_0^t e^{W(s)-\frac{s}{2}}\d s, \cdot\right)e^{W(t)-\frac{t}{2}}, $$
where $u_B$ is the Barenblatt solution defined above. The support at time $t\in(0,T]$ are all $x\in \mathbb{R}$, such that
\begin{equation}\label{supp_anal_stoch}
|x|\le \sqrt{C\frac{2d(p-1)}{a(p-2)}}\left(\int_0^t e^{W(s)-\frac{s}{2}}\d s\right)^{a/d}= \sqrt{12 C} \sqrt[3]{\int_0^t e^{W(s)-\frac{s}{2}}\d s},
\end{equation}
with $C=C(d,p)$ as above. Hence we can cut-off the domain $\D$ at $\pm L$ such that $(-L,L)$ contains the support of the solution in $(0,T]$ for each considered path of $W$ - this is verified for each path during the simulation. For our simulations we take $L=1.5$.
\subsection{Numerical Results in $1d$ for the stochastic equation}
The norm of the error is computed as before, 
where in addition a Monte-Carlo approximation for the expected value is used.\\
In Table \ref{Tab:Results_spme_1d_stoch} and Figure \ref{Fig:Results_spme_1d_stoch} we see, that convergence with respect to $(\tau,h)$ also for the approximation stochastic Barenblatt solution holds. 

\begin{center}
\begin{table}[H]
\begin{tabular}{l c c c c c c}
$N$ $\setminus$ $J$ & $8$ & $16$ & $32$ & $64$ & $128$ & $256$\\
$8$ & 0.045434 & 0.022878 & 0.036722 & 0.040406 & 0.041319 & 0.041538\\
$16$ & 0.06444 & 0.01176 & 0.021544 & 0.025007 & 0.025851 & 0.026057\\
$32$ & 0.069643 & 0.012408 & 0.011412 & 0.014496 & 0.015296 & 0.01549\\
$64$ & 0.074465 & 0.016884 & 0.0065525 & 0.0087332 & 0.0095565 & 0.0097648\\
$128$ & 0.076666 & 0.020286 & 0.0057942 & 0.00491 & 0.0056214 & 0.0058322\\
$256$ & 0.077362 & 0.02164 & 0.0063097 & 0.0029489 & 0.0031048 & 0.0032878\\
$512$ & 0.077722 & 0.022367 & 0.0067772 & 0.0025539 & 0.0017667 & 0.0018573\\
$1024$ & 0.077983 & 0.022812 & 0.0070991 & 0.002603 & 0.0012162 & 0.0010803
\end{tabular}\\  \ \\
\caption{$L^p(\Omega\times (0.01,0.1) \times\D)$-error of the solution in $1d$, a Monte-Carlo Approximation with $10^6$ samples was used to approximate the expectation.}\label{Tab:Results_spme_1d_stoch}
\end{table}
\end{center}

\begin{figure}[H]
\centering
\includegraphics[width=0.48\textwidth]{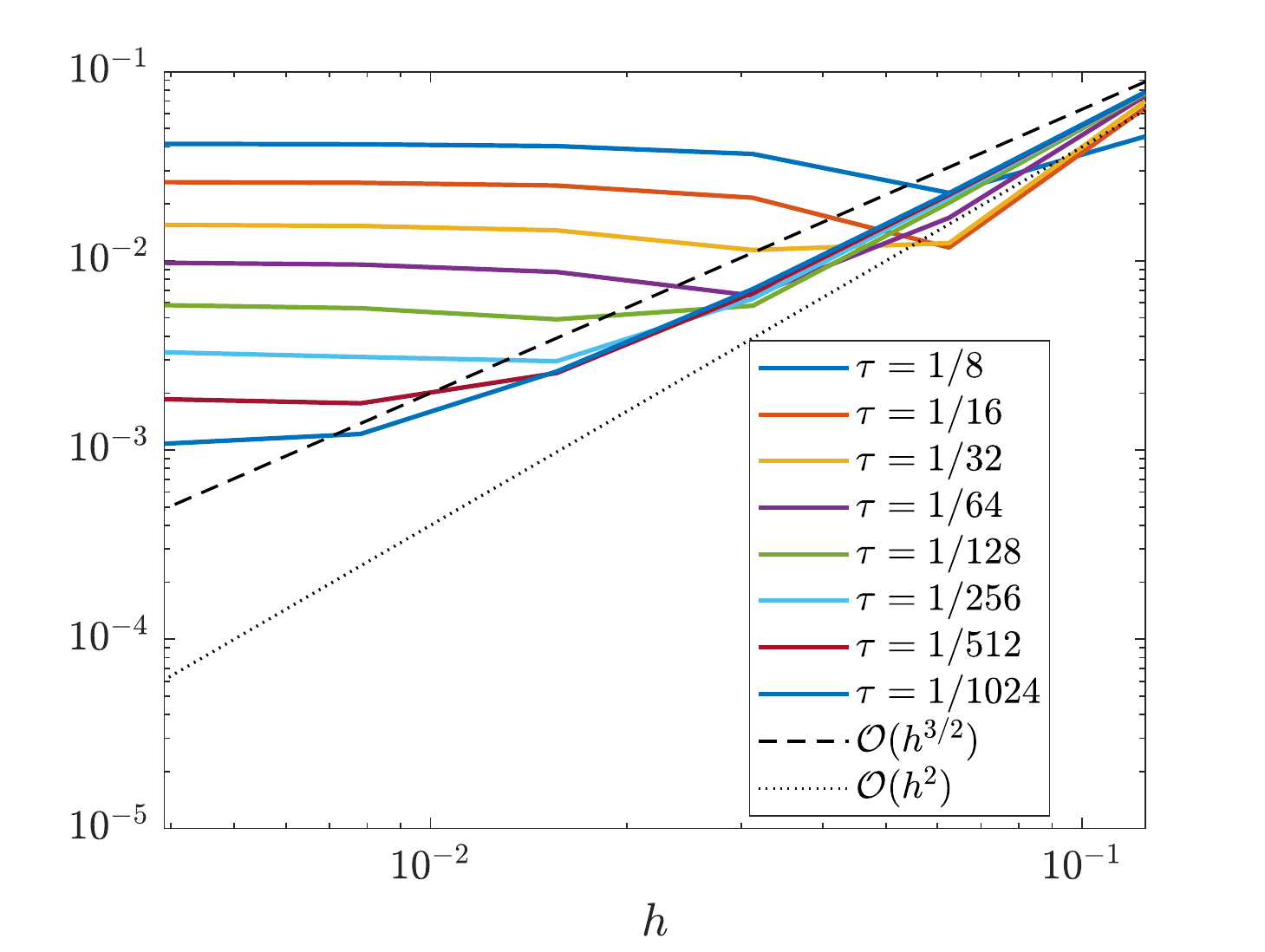}
\hfill
\includegraphics[width=0.48\textwidth]{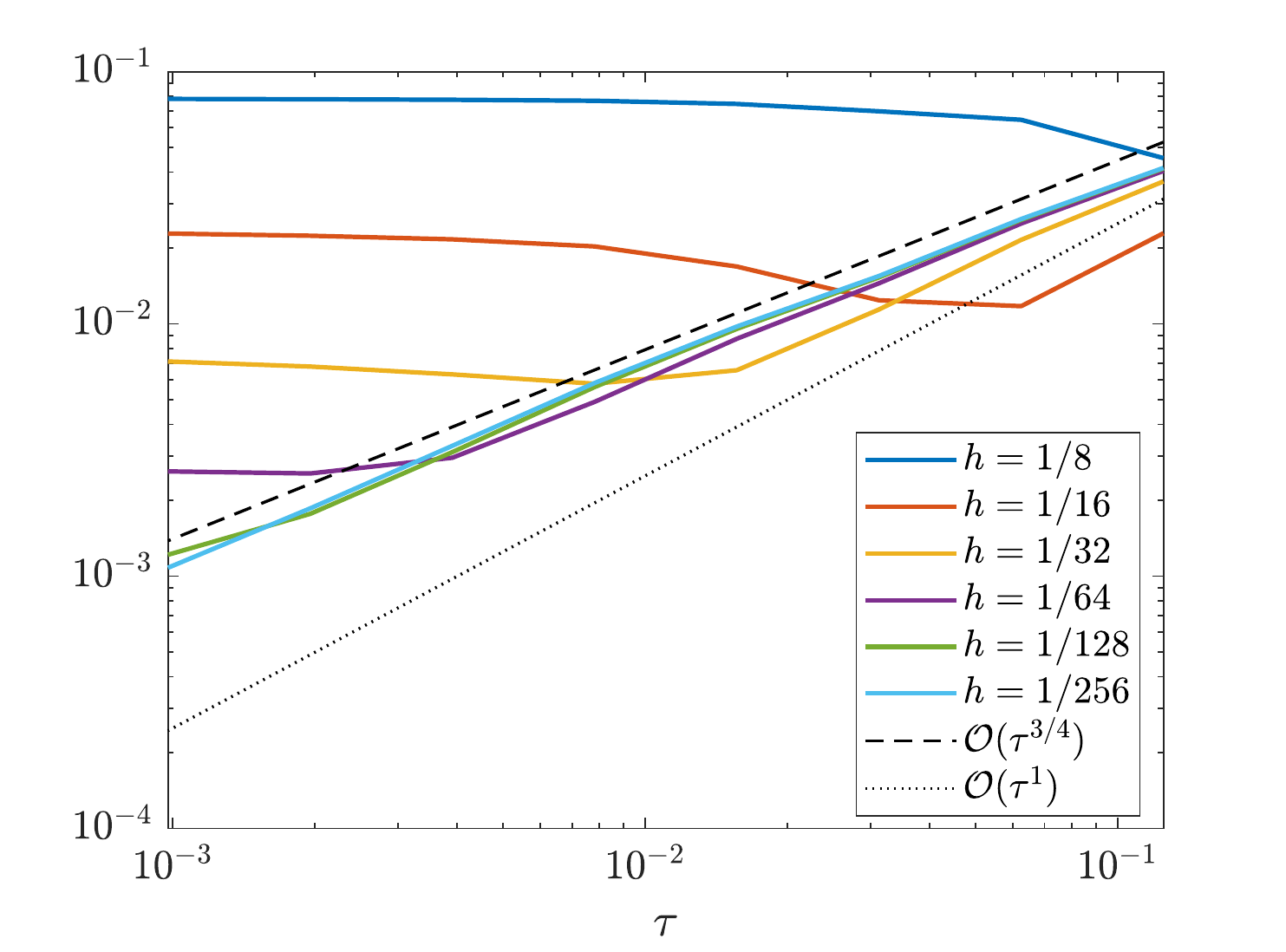}
\caption[Convergence for the stochastic solution in $1d$]{Convergence of the $L^p(\Omega \times (0.01,0.1) \times\D)$-error of the numerical solution in $1d$, a Monte-Carlo Approximation with $10^6$ samples was used to approximate the expectation. 
Left: convergence of the spacial discretization for different time step-sizes, right: convergence of the time discretization for different spacial step-sizes.}\label{Fig:Results_spme_1d_stoch}
\end{figure}

Figure \ref{Fig:Supp_spme_1d_stoch} shows one sample path, the analytical support (for this path) is plotted in red and the support of the approximation in yellow. Finite speed of propagation hold a.s. 
\begin{figure}[H]
\begin{subfigure}[c]{1\textwidth}
\centering
\includegraphics[scale=0.35]{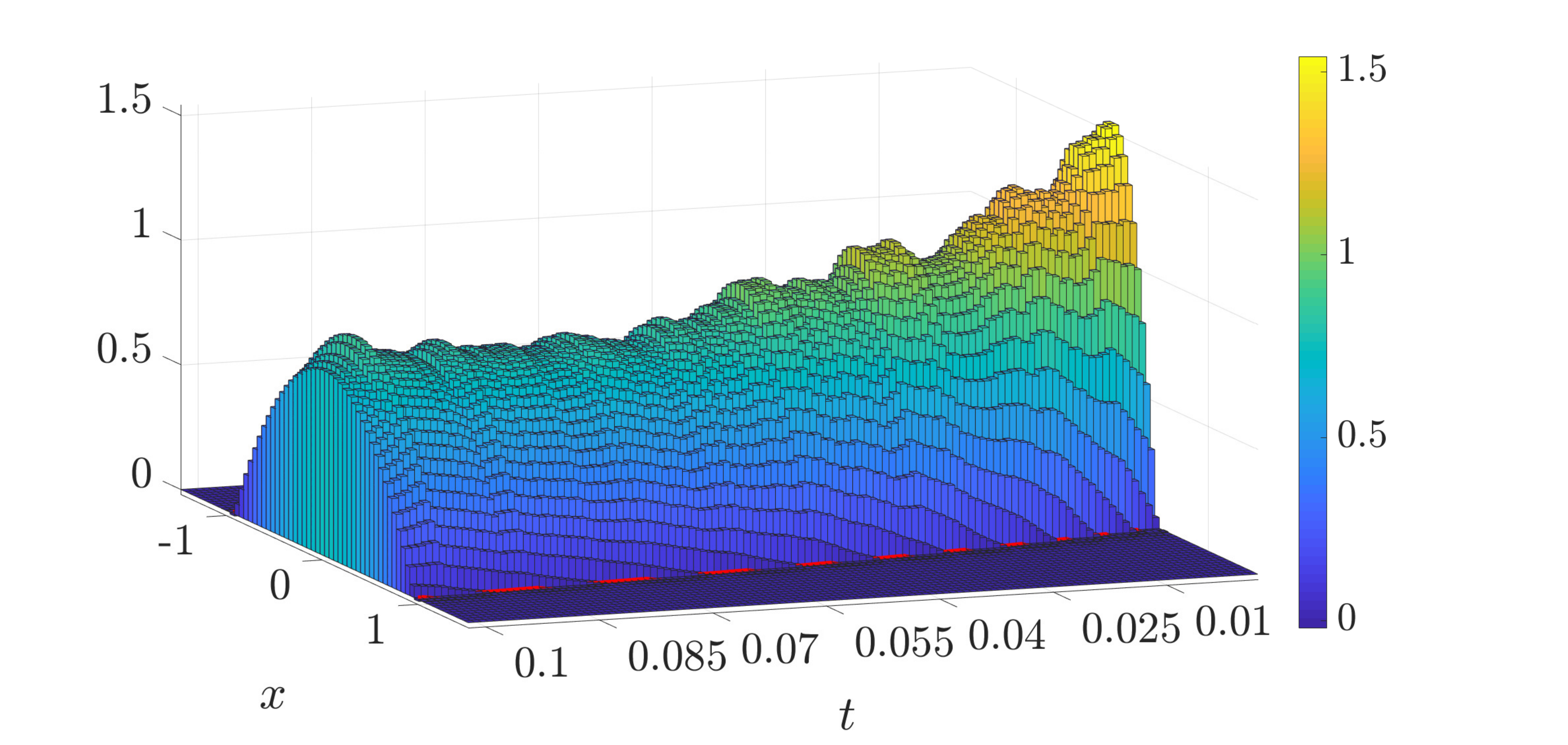}
\caption{$J = 64$, $N = 128$, $1d$}
\end{subfigure}
\begin{subfigure}[c]{1\textwidth}
\centering
\includegraphics[scale=0.35]{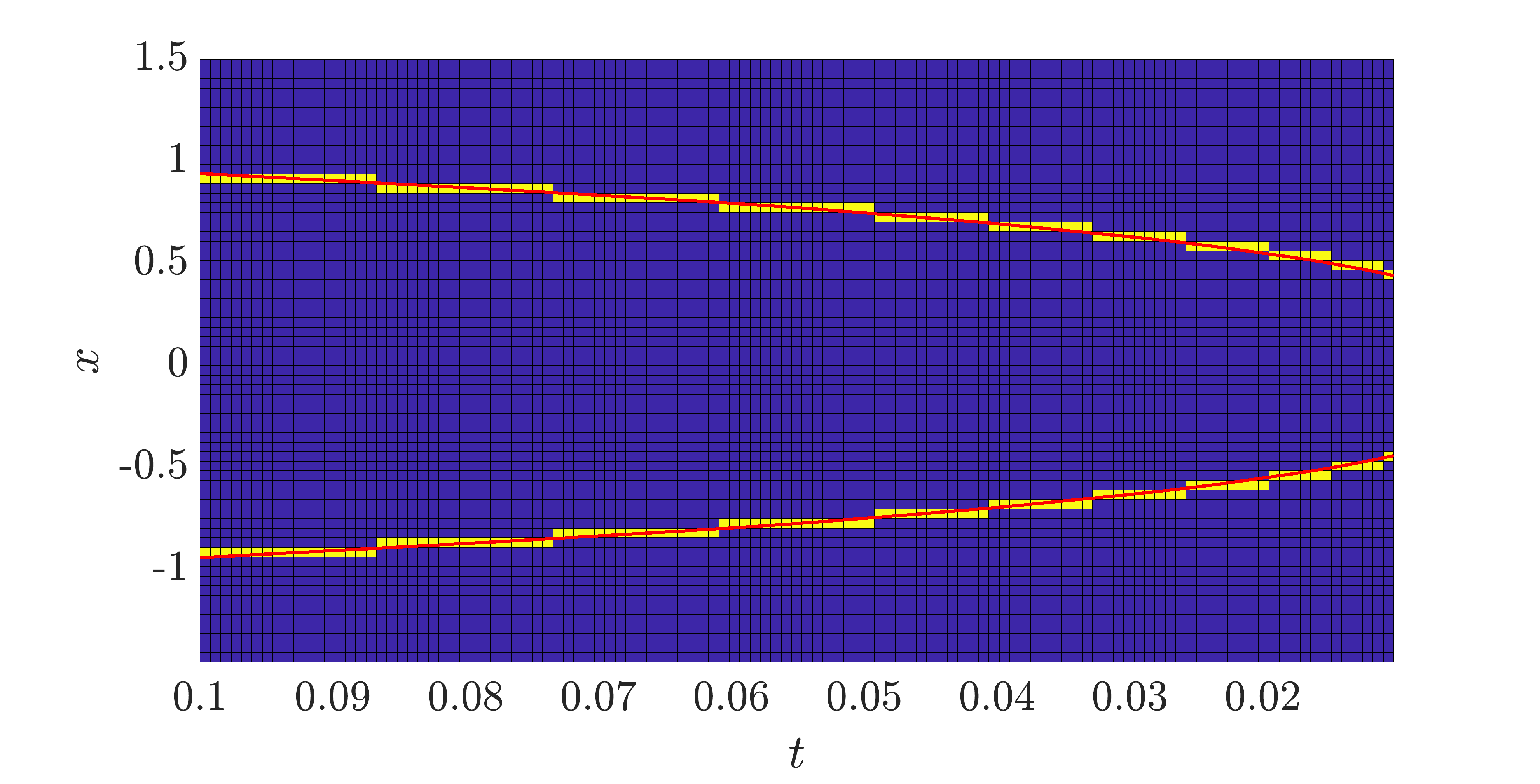}
\caption{Support of the approximation with $J = 64$, $N = 128$, $1d$. The red line indicates the analytical support from \eqref{supp_anal_stoch}.}
\end{subfigure}
\caption{Approximation of the stochastic solution in $1d$.}\label{Fig:Supp_spme_1d_stoch}
\end{figure}

\subsection{Numerical results for space-time noise}
Next, we perform simulation of the stochastic porous media equation with space-time white noise on $\D=[-L,L]$, $L=1.5$, where no analytical solution is available.
Given the 
the mesh size $h=\frac{2L}{J}$ we take $\sigma(u) \equiv \sigma_h(u) = \sigma_0\sum_{i=1}^J u\frac{\bfchi_i}{|\D_i|}$ where
where $\bfchi_i=\mathbbm{1}_{\D_i}$ are the indicator functions of $\D_i$.
We note that the $\bar{\mathbb{V}}_h$-valued noise
$\sigma_h(u) \WW(t,x) = \sigma_0 u \sum_{i=1}^{J}\frac{\bfchi_i(x)}{|\D_i|}\beta_i(t)$ in an approximation of the multiplicative noise $u \widetilde{\WW}$ where $\widetilde{\WW}$ is the space-time white noise,
cf. \cite{ANZ98}.

In Figure~\ref{Fig:spme_1d_stoch2} we display the numerical solution for one realization of the discrete space-time white noise with $\sigma_0=\frac{1}{64}$ along with the corresponding support. 
We observe that the evolution of the support for the space-time white noise does not deviate significantly from the deterministic case. In particular the numerical approximation
preserves the finite speed of propagation of the support, see Figure~\ref{Fig:Supp_spme_1d_stoch_comp}.

\begin{figure}[htp!]
\begin{subfigure}[c]{1\textwidth}
\centering
\includegraphics[trim=0 0 0 40, clip, scale=0.35]{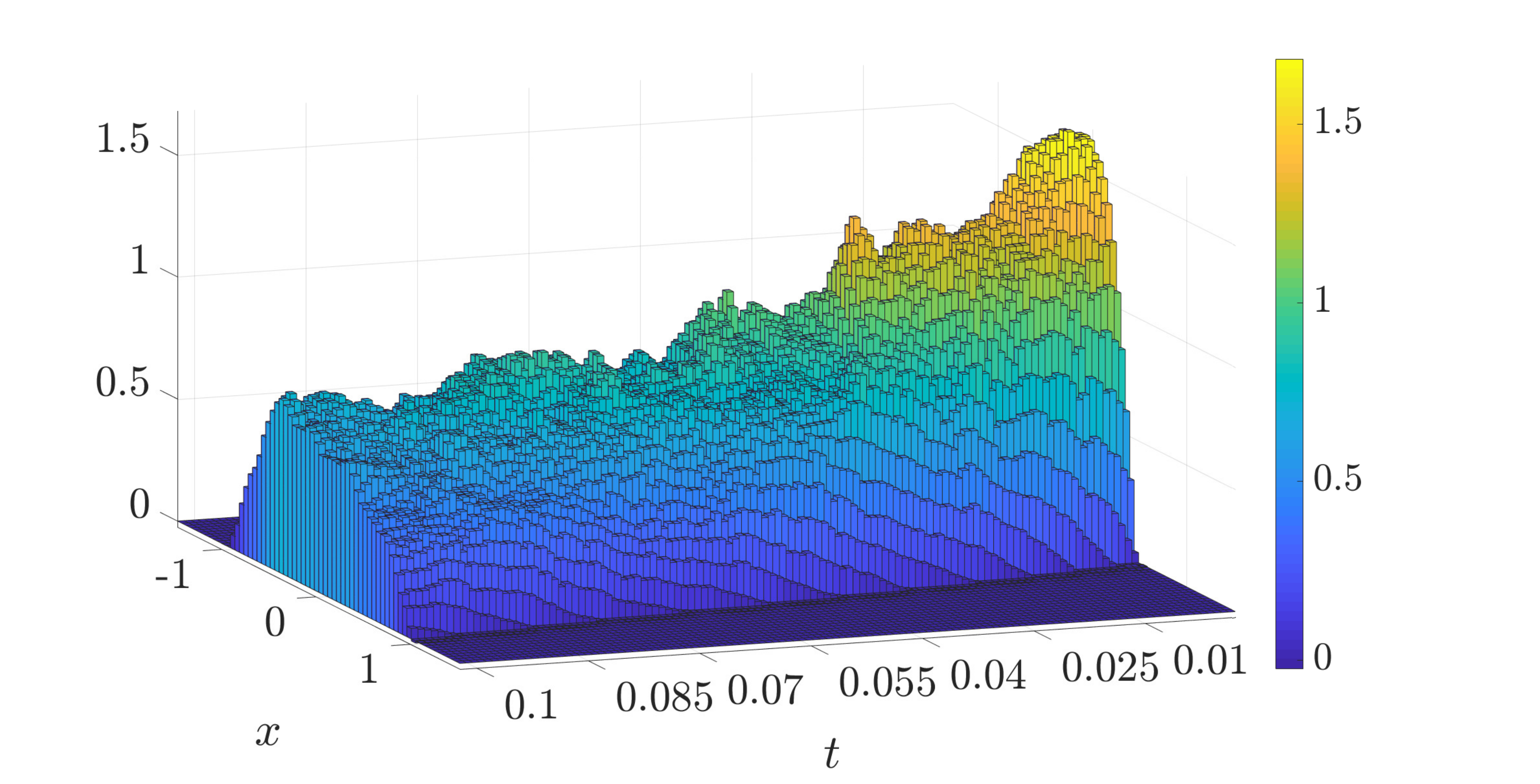}
\caption{Evolution of the numerical solution with the space-time-dependent noise, $J = 64$, $N = 128$}
\end{subfigure}
\begin{subfigure}[c]{1\textwidth}
\centering
\includegraphics[scale=0.35]{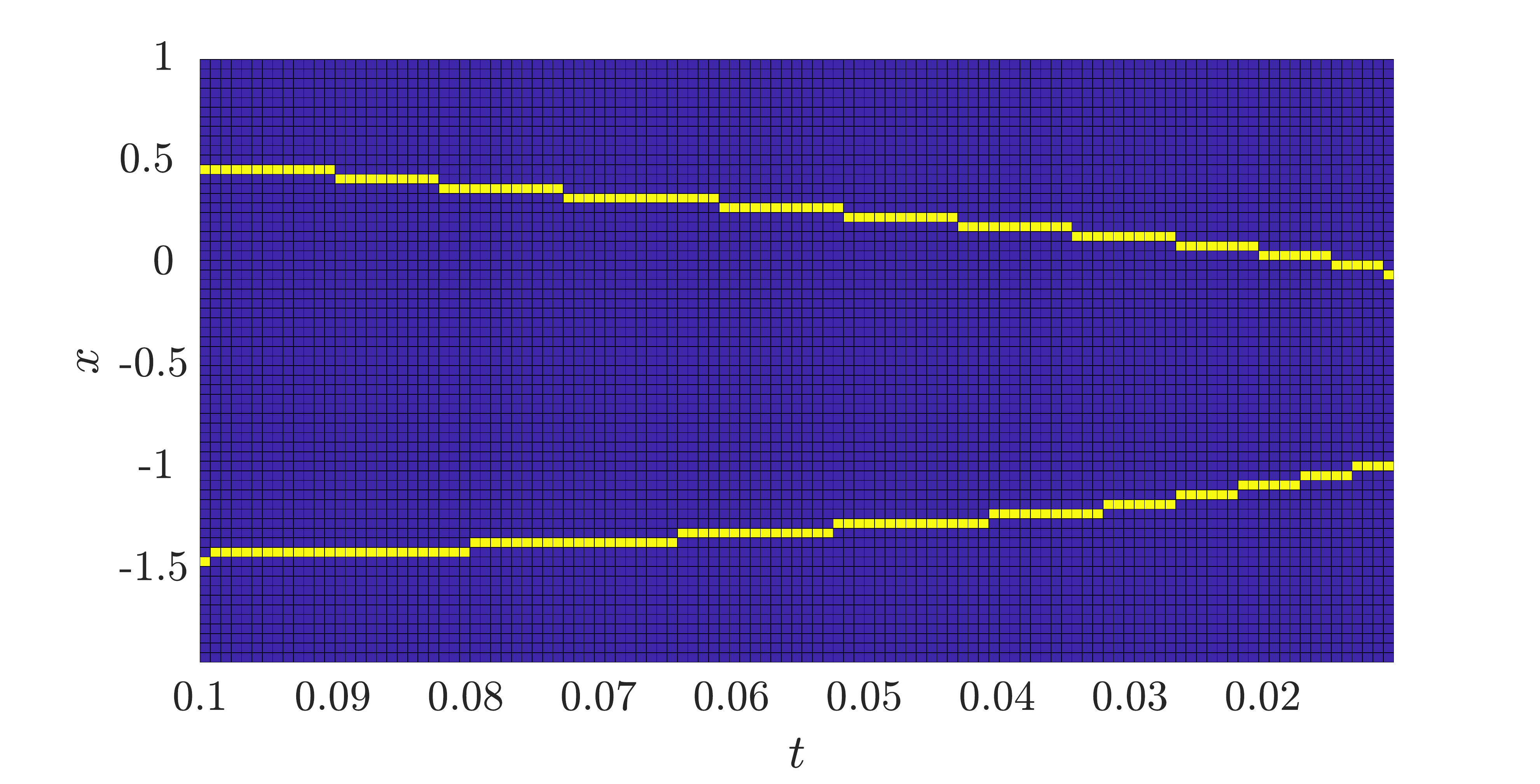}
\caption{Support of the approximation with $J = 64$, $N = 128$, $1d$ - asymmetric}
\end{subfigure}
\caption{Approximation of the stochastic solution with time-space dependent noise in $1d$.}\label{Fig:spme_1d_stoch2}
\end{figure}

\begin{figure}[htp!]
\begin{subfigure}[c]{1\textwidth}
\centering
\includegraphics[scale=0.35]{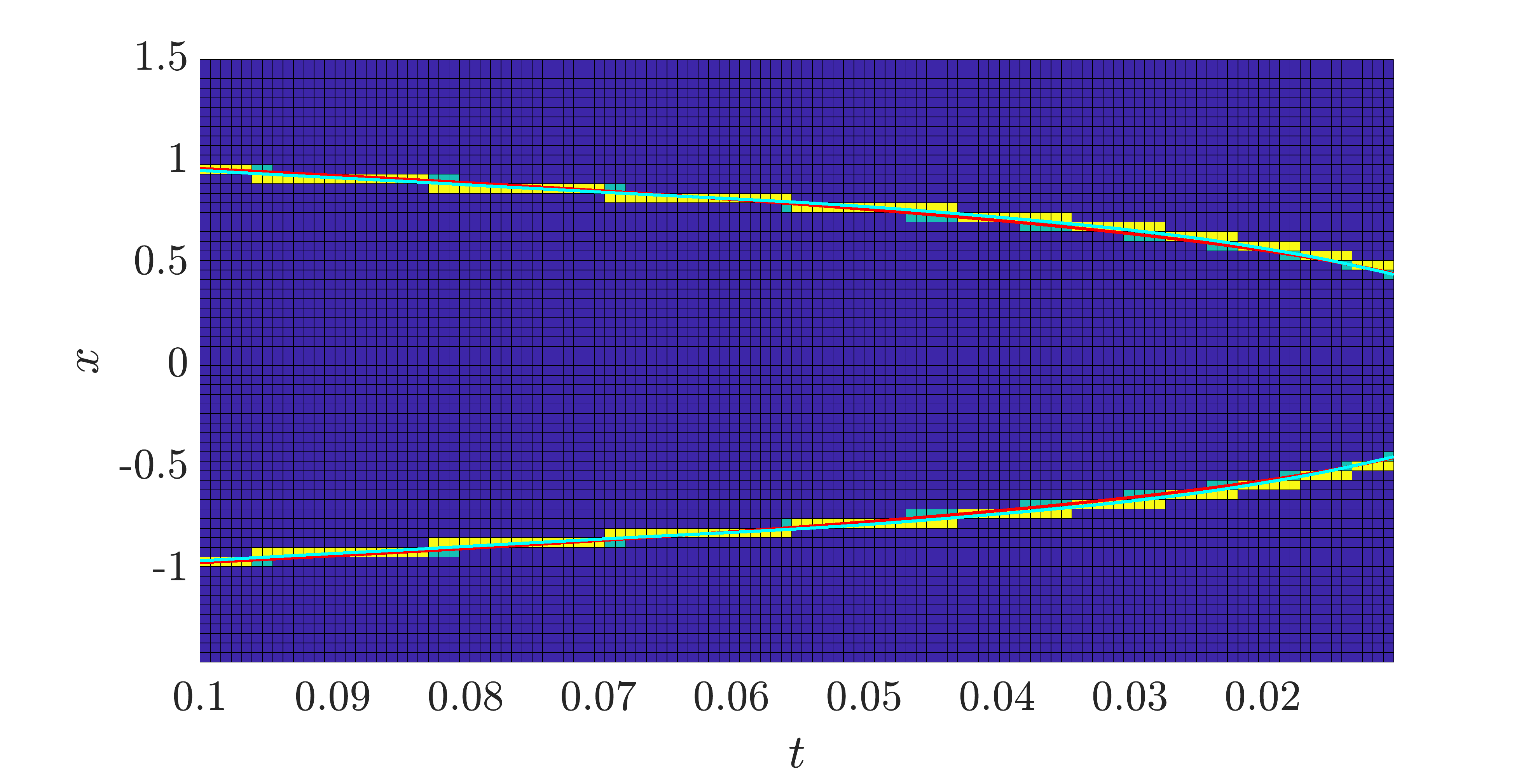}
\caption{Comparison of the spreading of the support $J = 64$, $N = 128$ in $1d$: deterministic in green, stochastic (time dependent noise) in yellow. The lines indicate the analytical support: deterministic in red and stochastic in cyan.}
\end{subfigure}
\begin{subfigure}[c]{1\textwidth}
\centering
\includegraphics[scale=0.35]{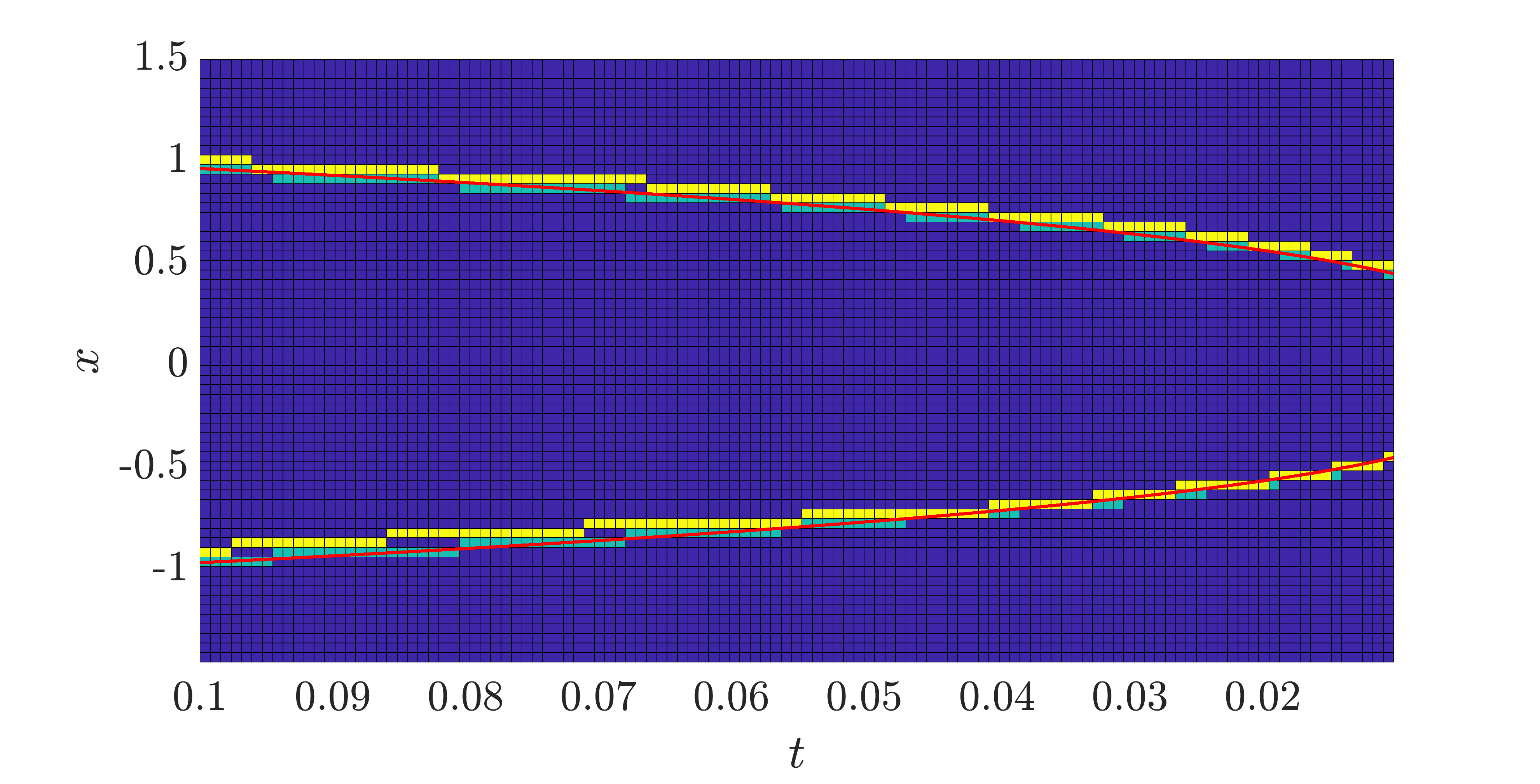}
\caption{Comparison of the spreading of the support $J = 64$, $N = 128$ in $1d$: deterministic in green, stochastic (space-time dependent) in yellow. The red line indicates the analytical support for the deterministic solution.}
\end{subfigure}
\caption{Support of the different approximations det./stochastic in $1d$.}\label{Fig:Supp_spme_1d_stoch_comp}
\end{figure}

\section*{Acknowledgement}
This work was supported by the Deutsche Forschungsgemeinschaft through SFB  1283 ''Taming  uncertainty  and  profiting  from  randomness  and  low  regularity  in  analysis, stochastics and their applications''. 

\bibliographystyle{plain}
\bibliography{spme_cv}

\begin{thebibliography}{10}

\bibitem{ANZ98}
E.~J. Allen, S.~J. Novosel, and Z.~Zhang.
\newblock Finite element and difference approximation of some linear stochastic
  partial differential equations.
\newblock {\em Stochastics Stochastics Rep.}, 64(1-2):117--142, 1998.

\bibitem{BBDPR09}
Viorel Barbu, Philippe Blanchard, Giuseppe Da~Prato, and Michael R\"{o}ckner.
\newblock Self-organized criticality via stochastic partial differential
  equations.
\newblock In {\em Potential theory and stochastics in {A}lbac}, volume~11 of
  {\em Theta Ser. Adv. Math.}, pages 11--19. Theta, Bucharest, 2009.

\bibitem{BDPR}
Viorel Barbu, Giuseppe da~Prato, and Michael Röckner.
\newblock {\em Stochastic Porous Media Equations}.
\newblock Springer International Publishing, 2016.

\bibitem{BR15}
Viorel Barbu and Michael R{\"o}ckner.
\newblock An operatorial approach to stochastic partial differential equations
  driven by linear multiplicative noise.
\newblock {\em J. Eur. Math. Soc. (JEMS)}, 17(7):1789--1815, 2015.

\bibitem{BR18}
Viorel Barbu and Michael R{\"o}ckner.
\newblock Nonlinear {{Fokker}}\textendash{{Planck}} equations driven by
  {{Gaussian}} linear multiplicative noise.
\newblock {\em Journal of Differential Equations}, 265(10):4993--5030, November
  2018.

\bibitem{BVW15}
Caroline Bauzet, Guy Vallet, and Petra Wittbold.
\newblock A degenerate parabolic-hyperbolic {{Cauchy}} problem with a
  stochastic force.
\newblock {\em J. Hyperbolic Differ. Equ.}, 12(3):501--533, 2015.

\bibitem{BGLR10}
Wolf-J{\"u}rgen Beyn, Benjamin Gess, Paul Lescot, and Michael R{\"o}ckner.
\newblock The {{Global Random Attractor}} for a {{Class}} of {{Stochastic
  Porous Media Equations}}.
\newblock {\em Comm. Partial Differential Equations}, 36(3):446--469, 2011.

\bibitem{book_fd}
Garrett Birkhoff and Robert~E. Lynch.
\newblock {\em Numerical solution of elliptic problems}, volume~6 of {\em SIAM
  Studies in Applied Mathematics}.
\newblock Society for Industrial and Applied Mathematics (SIAM), Philadelphia,
  PA, 1984.

\bibitem{DGG}
K.~Dareiotis, M.~Gerencs{\'e}r, and B.~Gess.
\newblock Entropy solutions for stochastic porous media equations.
\newblock {\em Journal of Differential Equations}, 266(6):3732 -- 3763, 2019.

\bibitem{DGG20}
Konstantinos Dareiotis, M{\'a}t{\'e} Gerencs{\'e}r, and Benjamin Gess.
\newblock Porous media equations with multiplicative space-time white noise.
\newblock {\em arXiv preprint arXiv:2002.12924}, 2020.

\bibitem{Da.Ge.Ge2020}
Konstantinos Dareiotis, M{\'a}t{\'e} Gerencs{\'e}r, and Benjamin Gess.
\newblock Porous media equations with multiplicative space-time white noise.
\newblock {\em arXiv:2002.12924 [math]}, February 2020.

\bibitem{DG17}
Konstantinos Dareiotis and Benjamin Gess.
\newblock Supremum estimates for degenerate, quasilinear stochastic partial
  differential equations.
\newblock {\em arXiv:1712.06655}, 2017.

\bibitem{Da.Ge20}
Konstantinos Dareiotis and Benjamin Gess.
\newblock Nonlinear diffusion equations with nonlinear gradient noise.
\newblock {\em Electron. J. Probab.}, 25:Paper No. 35, 43, 2020.

\bibitem{DaGeTs}
Konstantinos Dareiotis, Benjamin Gess, and Pavlos Tsatsoulis.
\newblock Ergodicity of stochastic porous media equations.
\newblock {\em arXiv:1907.04605}, 2019.

\bibitem{D96}
D.~Dean.
\newblock Langevin equation for the density of a system of interacting
  {{Langevin}} processes.
\newblock {\em Journal of Physics A: Mathematical and General}, 29(24):L613,
  1996.

\bibitem{DdMH15}
Arnaud Debussche, Sylvain {de Moor}, and Martina Hofmanov{\'a}.
\newblock A {{Regularity Result}} for {{Quasilinear Stochastic Partial
  Differential Equations}} of {{Parabolic Type}}.
\newblock {\em SIAM J. Math. Anal.}, 47(2):1590--1614, 2015.

\bibitem{De.Ho.Vo2016}
Arnaud Debussche, Martina Hofmanov{\'a}, and Julien Vovelle.
\newblock Degenerate parabolic stochastic partial differential equations:
  {{Quasilinear}} case.
\newblock {\em The Annals of Probability}, 44(3):1916--1955, May 2016.

\bibitem{DEJ19}
Felix Del~Teso, J{\o}rgen Endal, and Espen~R Jakobsen.
\newblock Robust numerical methods for nonlocal (and local) equations of porous
  medium type. part i: Theory.
\newblock {\em SIAM Journal on Numerical Analysis}, 57(5):2266--2299, 2019.

\bibitem{DFG20}
Nicolas Dirr, Benjamin Fehrman, and Benjamin Gess.
\newblock Conservative stochastic pde and fluctuations of the symmetric simple
  exclusion process.
\newblock {\em preprint}, 2020.

\bibitem{DSZ16}
Nicolas Dirr, Marios Stamatakis, and Johannes Zimmer.
\newblock Entropic and gradient flow formulations for nonlinear diffusion.
\newblock {\em J. Math. Phys.}, 57(8):081505, 13, 2016.

\bibitem{DFVE14}
A.~Donev, T.~G. Fai, and E.~{Vanden-Eijnden}.
\newblock A reversible mesoscopic model of diffusion in liquids: From giant
  fluctuations to {{Fick}}'s law.
\newblock {\em Journal of Statistical Mechanics: Theory and Experiment},
  2014(4):P04004, April 2014.

\bibitem{DroniouFastDiff}
Jérôme Droniou and Kim-Ngan Le.
\newblock The gradient discretization method for slow and fast diffusion porous
  media equations.
\newblock {\em SIAM Journal on Numerical Analysis}, 58(3):1965--1992, 2020.

\bibitem{EL08}
Carsten Ebmeyer and WB~Liu.
\newblock Finite element approximation of the fast diffusion and the porous
  medium equations.
\newblock {\em SIAM journal on numerical analysis}, 46(5):2393--2410, 2008.

\bibitem{EmmrichSiska}
Etienne Emmrich and David \v{S}i\v ska.
\newblock Full discretization of the porous medium/fast diffusion equation
  based on its very weak formulation.
\newblock {\em Commun. Math. Sci.}, 10(4):1055--1080, 2012.

\bibitem{EmmrichSiska2}
Etienne Emmrich and David \v{S}i\v ska.
\newblock Nonlinear stochastic evolution equations of second order with
  damping.
\newblock {\em Stoch. Partial Differ. Equ. Anal. Comput.}, 5(1):81--112, 2017.

\bibitem{FG18}
Benjamin Fehrman and Benjamin Gess.
\newblock Path-by-path well-posedness of nonlinear diffusion equations with
  multiplicative noise.
\newblock {\em arXiv preprint arXiv:1807.04230}, 2018.

\bibitem{FG20}
Benjamin Fehrman and Benjamin Gess.
\newblock Large deviations for conservative stochastic pde and non-equilibrium
  fluctuations.
\newblock {\em arXiv preprint arXiv:1910.11860}, 2019.

\bibitem{FG19}
Benjamin Fehrman and Benjamin Gess.
\newblock Well-posedness of nonlinear diffusion equations with nonlinear,
  conservative noise.
\newblock {\em Archive for Rational Mechanics and Analysis}, 233(1):249--322,
  2019.

\bibitem{FG15}
Julian Fischer and G{\"u}nther Gr{\"u}n.
\newblock Finite speed of propagation and waiting times for the stochastic
  porous medium equation: A unifying approach.
\newblock {\em SIAM Journal on Mathematical Analysis}, 47(1):825--854, 2015.

\bibitem{G12}
Benjamin Gess.
\newblock Strong solutions for stochastic partial differential equations of
  gradient type.
\newblock {\em J. Funct. Anal.}, 263(8):2355--2383, 2012.

\bibitem{G13-1}
Benjamin Gess.
\newblock Finite speed of propagation for stochastic porous media equations.
\newblock {\em arXiv:1210.2415}, pages 1--26, 2013.

\bibitem{G13-2}
Benjamin Gess.
\newblock Random attractors for stochastic porous media equations perturbed by
  space-time linear multiplicative noise.
\newblock {\em Ann. Probab.}, 42(2):818--864, 2014.

\bibitem{G15}
Benjamin Gess.
\newblock Finite time extinction for stochastic sign fast diffusion and
  self-organized criticality.
\newblock {\em Comm. Math. Phys.}, 335(1):309--344, 2015.

\bibitem{GH18}
Benjamin Gess and Martina Hofmanov{\'a}.
\newblock Well-posedness and regularity for quasilinear degenerate
  parabolic-hyperbolic {{SPDE}}.
\newblock {\em The Annals of Probability}, 46(5):2495--2544, 2018.

\bibitem{GS16-2}
Benjamin Gess and Panagiotis~E. Souganidis.
\newblock Stochastic non-isotropic degenerate parabolic--hyperbolic equations.
\newblock {\em Stochastic Process. Appl.}, 127(9):2961--3004, 2017.

\bibitem{Gilbarg}
D.~Gilbarg and N.S. Trudinger.
\newblock {\em Elliptic partial differential equations of second order}.
\newblock Classics in Mathematics. Springer-Verlag, Berlin, 2001.
\newblock Reprint of the 1998 edition.

\bibitem{GRZ09}
Benjamin Goldys, Michael R{\"o}ckner, and Xicheng Zhang.
\newblock Martingale solutions and {{Markov}} selections for stochastic partial
  differential equations.
\newblock {\em Stochastic Processes and their Applications}, 119(5):1725--1764,
  May 2009.

\bibitem{gg_2019}
H.~Grillmeier and G.~Gr\"{u}n.
\newblock Nonnegativity preserving convergent schemes for stochastic
  porous-medium equations.
\newblock {\em Math. Comp.}, 88(317):1021--1059, 2019.

\bibitem{book_grisvard}
P.~Grisvard.
\newblock {\em Elliptic problems in nonsmooth domains}, volume~24 of {\em
  Monographs and Studies in Mathematics}.
\newblock Pitman (Advanced Publishing Program), Boston, MA, 1985.

\bibitem{Gyongy}
Istv{\'a}n Gy{\"o}ngy.
\newblock On stochastic squations with respect to semimartingales iii.
\newblock {\em Stochastics}, 7(4):231--254, 1982.

\bibitem{GyongyMillet}
I.~n. Gyöngy and A.~Millet.
\newblock On discretization schemes for stochastic evolution equations.
\newblock {\em POTENTIAL ANALYSIS, 2005, 23, 2, 99}, 2005.

\bibitem{JK91}
Willi J{\"a}ger and Jozef Ka{\v{c}}ur.
\newblock Solution of porous medium type systems by linear approximation
  schemes.
\newblock {\em Numerische Mathematik}, 60(1):407--427, 1991.

\bibitem{K06}
Jong~U. Kim.
\newblock On the stochastic porous medium equation.
\newblock {\em J. Differential Equations}, 220(1):163--194, 2006.

\bibitem{KR_SEE}
N.~V. Krylov and B.~L. Rozovskii.
\newblock Stochastic evolution equations.
\newblock {\em Journal of Soviet Mathematics}, 16(4):1233--1277, Jul 1981.

\bibitem{LKR19}
T.~Lehmann, V.~Konarovskyi, and M.~{von Renesse}.
\newblock Dean-{{Kawasaki Dynamics}}: {{Ill}}-posedness vs. {{Triviality}}.
\newblock {\em arXiv:1806.05018}, June 2018.

\bibitem{book_lions}
J.-L. Lions.
\newblock {\em Quelques m\'{e}thodes de r\'{e}solution des probl\`emes aux
  limites non lin\'{e}aires}.
\newblock Dunod; Gauthier-Villars, Paris, 1969.

\bibitem{LR15}
Wei Liu and Michael R{\"o}ckner.
\newblock {\em Stochastic Partial Differential Equations: An Introduction}.
\newblock Universitext. {Springer, Cham}, 2015.

\bibitem{MNV87}
E~Magenes, RH~Nochetto, and C~Verdi.
\newblock Energy error estimates for a linear scheme to approximate nonlinear
  parabolic problems.
\newblock {\em ESAIM: Mathematical Modelling and Numerical
  Analysis-Mod{\'e}lisation Math{\'e}matique et Analyse Num{\'e}rique},
  21(4):655--678, 1987.

\bibitem{Meleard}
S.~M\'{e}l\'{e}ard and S.~Roelly.
\newblock Interacting measure branching processes. {S}ome bounds for the
  support.
\newblock {\em Stochastics Stochastics Rep.}, 44(1-2):103--121, 1993.

\bibitem{Necas}
J.~Ne\v{c}as.
\newblock {\em Les m{\'e}thodes directes en th{\'e}orie des {\'e}quations
  elliptiques}.
\newblock Masson, 1967.

\bibitem{O05}
Hans~Christian {\"O}ttinger.
\newblock {\em Beyond equilibrium thermodynamics}.
\newblock John Wiley \& Sons, 2005.

\bibitem{P75}
{\'E}tienne Pardoux.
\newblock Equations aux d{\'e}riv{\'e}es partielles stochastiques non
  lin{\'e}aires monotones.
\newblock {\em PhD thesis}, 1975.

\bibitem{RRW07}
Jiagang Ren, Michael R{\"o}ckner, and Feng-Yu Wang.
\newblock Stochastic generalized porous media and fast diffusion equations.
\newblock {\em Journal of Differential Equations}, 238(1):118--152, 2007.

\bibitem{Sc.St2019}
Luca Scarpa and Ulisse Stefanelli.
\newblock Doubly nonlinear stochastic evolution equations.
\newblock {\em arXiv:1905.11294 [math]}, July 2019.

\bibitem{Temam}
Roger Temam.
\newblock {\em Navier-Stokes equations}, volume~2 of {\em Studies in
  mathematics and its applications}.
\newblock North-Holland Publ., Amsterdam [u.a.], rev. ed. edition, 1979.

\bibitem{RoecknerSPME}
Michael Röckner~(auth.) Viorel~Barbu, Giuseppe Da~Prato.
\newblock {\em Stochastic Porous Media Equations}.
\newblock Lecture Notes in Mathematics 2163. Springer International Publishing,
  1 edition, 2016.

\bibitem{Vazquez}
Juan~Luis Vázquez.
\newblock {\em The porous medium equation}.
\newblock Oxford mathematical monographs. Clarendon Press, Oxford [u.a.], 2007.

\bibitem{W15-2}
Feng-Yu Wang.
\newblock Exponential convergence of non-linear monotone {{SPDEs}}.
\newblock {\em Discrete and Continuous Dynamical Systems. Series A},
  35(11):5239--5253, 2015.

\end{thebibliography}


%
\end{document}